\documentclass[10pt]{amsart}
\usepackage[left=2cm,right=2cm,top=2cm,bottom=2cm]{geometry}
\usepackage{amsmath,amssymb,amsthm,amscd,eucal,enumerate,url,longtable}

\newtheorem{theorem}{Theorem}

\newtheorem{corollary}[theorem]{Corollary}

\newtheorem{proposition}[theorem]{Proposition}

\newcommand{\h}{\mathbb{H}}
\newcommand{\p}{\mathbf{P}}
\newcommand{\Q}{\mathbb{Q}}
\newcommand{\R}{\mathbb{R}}
\newcommand{\Z}{\mathbb{Z}}
\newcommand{\C}{\mathbb{C}}
\newcommand{\OO}{\mathcal{O}}
\newcommand{\E}{\mathcal{E}}
\newcommand{\F}{\mathcal{F}}
\newcommand{\M}{\mathbb{M}}

\newcommand{\el}{\mathcal{L}}
\DeclareMathOperator{\ord}{\mathrm{ord}}
\DeclareMathOperator{\SL}{SL}
\DeclareMathOperator{\PSL}{PSL}
\DeclareMathOperator{\vol}{Vol}

\DeclareMathOperator{\Id}{Id}

\begin{document}

\title{The Hauptmodul at elliptic points of certain arithmetic groups}
\author[J.~Jorgenson]{Jay Jorgenson}
\address{Department of Mathematics, The City College of New York,
Convent Avenue at 138th Street, New York, NY 10031 USA,
e-mail: jjorgenson@mindspring.com}
\author[L.~Smajlovi\'c]{Lejla Smajlovi\'c}
\address{Department of Mathematics, University of Sarajevo,
Zmaja od Bosne 35, 71\,000 Sarajevo, Bosnia and Herzegovina,
e-mail: lejlas@pmf.unsa.ba}
\author[H.~Then]{Holger Then}
\address{Alemannenweg 1, 89537 Giengen, Germany,
e-mail: holger.then@bristol.ac.uk}

\begin{abstract}
Let $N$ be a square-free integer such that the arithmetic group
$\Gamma_0(N)^+$ has genus zero; there are $44$ such groups. Let $j_N$
denote the associated Hauptmodul normalized to have residue equal to one and
constant term equal to zero in its $q$-expansion. In this article we prove
that the Hauptmodul at any elliptic point of the surface associated to
$\Gamma_0(N)^+$ is an algebraic integer. Moreover, for each such $N$
and elliptic point $e$, we show how to explicitly evaluate $j_{N}(e)$ and provide the list of
generating polynomials (with small coefficients) of the class fields or their subfields
corresponding to the orders over the imaginary quadratic extension of rationals stemming
from the elliptic points under consideration.
\end{abstract}

\thanks{J.~J.\ acknowledges grant support from NSF and PSC-CUNY grants}

\maketitle

\section{Introduction}

\subsection{Some number theoretic considerations}

Let $\Gamma$ be a discrete group acting on the hyperbolic upper half plane $\h$ such that
the quotient $\Gamma \backslash \h$ has genus zero, and, of course,
necessarily admits some cusps and elliptic points. The function field associated to
$\Gamma \backslash \h$ has transcendence degree one and is generated over the field of
constants, which for this paper are the complex numbers $\C$, by a single indeterminant
which we denote by $j_{\Gamma}$. In general, one can normalize $j_{\Gamma}$ by choosing
a distinguished point $P$ on $\Gamma \backslash \h$ and requiring $j_{\Gamma}$ to have a first
order pole at $P$ with residue equal to one as well as zero constant term in its Laurent expansion
about $P$ having chosen a local coordinate at $P$.

In the specific case when $\Gamma = \PSL(2,\Z)$, we can take $P = i\infty$ since as
a Riemann surface the quotient space $\PSL(2,\Z) \backslash \h$ has a cusp.
Let $z$ denote the global coordinate on $\h$, and set $q = e^{2\pi i z}$ which is
a local coordinate about $i\infty$. Then the (classical) $j$-invariant
$j(z)=j_{\PSL(2,\Z)}(z)$ admits the $q$-expansion on $\p^1$ given by
\begin{align}\label{j_expansion}
j(z) =\frac{1}{q} + \sum_{k=1}^{\infty}a_k q^k
=\frac{1}{q} + 744 + 196884 q + 21493760 q^2 + O(q^3)
\quad \text{as $q \to 0$.}
\end{align}
From the point of view of automorphic forms, $j$ can be realized as a rational function of
holomorphic Eisenstein series of weight four and six.

As it turns out, the function $j(z)$ satisfies many amazing properties. T.~Schneider proved
in \cite{Sch37} that if $\tau \in \h$ is an imaginary quadratic number then $j(\tau)$ is an
algebraic integer. In addition, if $\tau$ is an algebraic number but not imaginary
quadratic then $j(\tau)$ is transcendental; see also \cite{Siegel49}. In modern language,
the points $\tau \in \h$ which are imaginary quadratic numbers are called
\emph{complex multiplication points,} or CM points. In \cite{Za02} it is shown how
to compute $j(\tau)$ at any CM point, thus giving some fantastic formulae such as
$$
j(i) = 1728, \quad j((1+\sqrt{-7})/2) = -3375, \quad \text{and} \quad
j((1+\sqrt{-163})/2) = 640320^{3}.
$$
The third example combines with \eqref{j_expansion} to yield the curiosity that the
transcendental number $e^{\pi\sqrt{163}}$ is very close to an integer, a result which
is attributed to Hermite.

More generally, the \emph{singular moduli} of the $j$-invariant, which by definition
are the values of the function $j$ at imaginary quadratic arguments, play a very important
role in the class field theory of imaginary quadratic fields; see \cite{SCM66}. Namely,
let $K$ be an imaginary quadratic field over $\Q$, of discriminant $d_K$ and let $\OO$ be
a certain order in $K$. Then for
an imaginary quadratic argument $\tau \in\h \cap \OO$, the singular modulus $j(\tau)$ is an
algebraic integer. Moreover, the extension $K[j(\tau)]$ is the ring class field of $\OO$,
which is the Hilbert class field if $\OO$ is the maximal order of $K$, and can be realized
constructively as the splitting field over $\Q$ of the class polynomial, by which we mean
the minimal polynomial of $j(\tau)$.

The seminal work of Gross-Zagier \cite{GZ85} studies the factorization of the difference
of two singular moduli $j(\tau_{1}) - j(\tau_{2})$,
from which we have a considerable amount of current research that reaches in various
directions of algebraic and arithmetic number theory including special values of
$L$-functions and the Birch-Swinnerton-Dyer conjecture, one of the six unsolved
Millennium Prize Problems.

Thus, properties of the $j$-invariant for $\PSL(2,\Z)$ play a role in algebraic number theory.

\subsection{Connections to other fields}\label{sec 1.2}

Let $f$ be a holomorphic function of one complex variable. The
\emph{Schwarzian derivative} $S(f)$ of $f$ is a classically defined function given by
$$
S(f)(z)= \left( \frac{f''(z)}{f'(z)}\right)' - \frac{1}{2}\left( \frac{f''(z)}{f'(z)}\right)^2,
$$
where, as is the convention, the prime $'$ denotes differentiation with respect
to the holomorphic parameter $z$. If $f$ and $g$ are holomorphic functions,
then the Schwarzian of the composition $f \circ g$ satisfies the relation
$$
S(f\circ g) = (S(f)\circ g)(g')^{2} + S(g).
$$
In addition, one can show that $S(g) = 0$ if and only if $g$ is a
fractional linear transformation. Therefore, if the function $f$ is a
holomorphic automorphic form with respect to some discrete group
$\Gamma\subseteq \PSL(2,\R)$, then the Schwarzian $S(f)$ is a meromorphic
automorphic form of weight four with respect to $\Gamma$.

In \cite{Masser03} it is proven that the classical $j$-invariant for $\PSL(2,\Z)$
satisfies the differential equation
\begin{align}\label{dif eq for j}
S(j)(z)+ R(j(z))(j'(z))^2=0,
\end{align}
where
\begin{align*}
R(y)= \frac{y^2-1968y+2654208}{2y^2(y-1728)^2}.
\end{align*}

\noindent Recently Freitag and Scanlon \cite{FS14} used \eqref{dif eq for j} to
define a non-$\aleph_0$-categorical strongly minimal set with trivial forking geometry,
thus answering an open problem about the existence of such sets. The authors in
\cite{FS14} attribute the question to Lascar, who himself credits the question to Poizat.
We refer the reader to \cite{FS14} for precise statements as well as numerous
applications of their result.

Thus, properties of the $j$-invariant for $\PSL(2,\Z)$ provide a means by which one can
address problems in logic and differential algebraic geometry.

\subsection{Some other genus zero groups}

For any positive integer $N$, let
\begin{align*}
\Gamma_0(N)^+=\left\{ e^{-1/2} \begin{pmatrix} a & b \\ c & d \end{pmatrix}
\in \SL(2,\R): \quad ad-bc=e, \quad a,b,c,d,e\in\Z,
\quad e\mid N,\ e\mid a,\ e\mid d,\ N\mid c \right\}
\end{align*}
and let $\overline{\Gamma_0(N)^+} = \Gamma_0(N)^+/\{\pm\Id\}$, where $\Id$
denotes the identity matrix. It has been shown that there are $43$
square-free integers $N>1$ such that the quotient space
$X_N := \overline{\Gamma_0(N)^+}\backslash \h$
has genus zero (see \cite{Cum04}); note that $\PSL(2,\Z) = \overline{\Gamma_0(1)^+}$.
We will also say that $\Gamma_0(N)^+$ is a genus zero group if $N$ corresponds to
one of the $44$ aforementioned numbers. For each such genus zero group,
the surface $X_N$ has one cusp which we can take to be at $i\infty$ and width one.
Basic properties of
$\Gamma_0(N)^+$, for square-free $N$ are derived in \cite{JST12} and references therein.

For every genus zero group $\Gamma_0(N)^+$ there exists a unique holomorphic
modular form on $\h$ with a pole at $i\infty$ of order one such that its
$q$-expansion is normalized so it begins with $1/q$ and the constant term is
equal to zero. We denote the form by $j_{\Gamma_0(N)^+}:=j_N$ and, following classical and
well-accepted terminology, refer to $j_N$ as the Hauptmodul of
$\Gamma_0(N)^+$. Therefore, each Hauptmodul $j_N$ possesses
a Fourier expansion at the cusp $i\infty$ with integer coefficients $a_N(k)$,
normalized so that $a_N(-1)=1$ and $a_N(0)=0$. In other words, the
$q$-expansion of $j_N(z)$ is given by
\begin{align}\label{q exp jN}
j_N(z)= \frac{1}{q} + \sum_{k=1}^{\infty} a_N(k)q^k
= \sum_{k=-1}^{\infty} a_N(k)q^k.
\end{align}
Since $X_N$ has genus zero, the
function field has transcendence degree one over $\C$ and is generated by one
indeterminant meaning that every $\Gamma_0(N)^+$-invariant meromorphic function
can be written as a rational function in $j_N$.

\subsection{The beginning of ``monstrous moonshine"}

Let $\M$ denote ``the monster'' group, which is the largest sporadic finite simple group.
In the mid-1900's, there were two very important and independent observations;
A.~Ogg showed that the set of primes which appear in the factorization of the order of
$\M$ is the same set of primes such that $\Gamma_0(p)^+$ has genus zero, and J.~McKay
pointed out that the linear-term coefficient in \eqref{j_expansion} is the
sum of the two smallest irreducible character degrees of $\M$. Subsequent
work by J.~Thompson resulted in specific conjectures asserting all coefficients
in the expansion \eqref{j_expansion} are related to the dimensions of the
components of a graded module admitting action by $\M$. More generally,
J.~Conway and S.~Norton established the ``monstrous moonshine'' conjectures
in \cite{CN79} which more precisely formulated relations between $\M$ and the
$j$-invariants for the genus zero groups $\Gamma_{0}(N)^{+}$, culminating in the
celebrated work of Borcherds in \cite{Bo92}.

Thus, properties of the $j$-invariants $j_{N}$ for the genus zero groups
$\Gamma_{0}(N)^{+}$ appear in group theory and all the numerous
fields touched by ``monstrous moonshine''.

\subsection{Singular moduli for $\Gamma_{0}(N)^{+}$}

For $N>1$ such that $\Gamma_0(N)^+$ has genus zero, the singular moduli were studied
by I.~Chen and N.~Yui \cite{CY93}. Analogous to Schneider's result, the authors in
\cite{CY93} proved that if $\tau$ is a CM point satisfying $az^2 + bz+c=0$ with
$(a,N)=1$ then the singular moduli $j_N(\tau)$ is an algebraic integer.
Furthermore, if we put $K=\Q[\tau]$, $b^2-4ac=m^2d_K<0$, and let $\OO$ denote the
order in $K$ of discriminant $m^2d_K$, then Theorem~3.7.5.(2) from \cite{CY93} states
that for prime levels $N$, and assuming that $(a,N)=1$, the singular moduli $j_N(\tau)$
generates over $K$ the ring class field of an imaginary quadratic order $\OO'$ of
discriminant $(mN)^2 d_K$. The assumption requiring that $(a,N)=1$ was crucial in the
proof. If $(a,N)>1$, then $K[j_N(\tau)]$ is a proper subfield of the ring class
field of $\OO'$.

The results from \cite{CY93} were expanded upon in \cite{CMcKS04} and
\cite{CK05}. Let us denote by $\Gamma_0(N)^{\ast}$ the subgroup of $\mathrm{PSL}(2,\R)$ generated by
$\overline{\Gamma_0(N)}=\Gamma_0(N)/\{\pm \mathrm{Id}\}$ and the Fricke involution
$\gamma_N=\begin{pmatrix} 0 & -1/\sqrt{N} \\ \sqrt{N} & 0 \end{pmatrix}$.
Note that for prime levels $N$, one has $\Gamma_0(N)^{\ast}=\overline{\Gamma_0(N)^+}$, otherwise
$\Gamma_0(N)^{\ast}$ is a proper subgroup of $\overline{\Gamma_0(N)^+}$. When $N$ is a product
of $r$ primes, $\overline{\Gamma_0(N)}$ is a subgroup of index $2^r$ in $\overline{\Gamma_0(N)^+}$
(see e.g. \cite{AtLeh70}, Lemma 9) while $\overline{\Gamma_0(N)}$ is a subgroup of index $2$ in
$\Gamma_0(N)^{\ast}$, hence $\Gamma_0(N)^{\ast}$ is a subgroup of index $2^{r-1}$ in
$\overline{\Gamma_0(N)^+}$.
In \cite{CMcKS04} it is proven that for any fixed elliptic point $\tau_{\gamma} \in\h$
with corresponding order two elliptic element
$\gamma\in\Gamma_0(N)^{\ast} \setminus \overline{\Gamma_0(N)}$,
the field $\Q[\tau_{\gamma}, j_N(\tau_{\gamma})]$ is the ring class field of the order
$\OO_{\gamma}$ in $K$, where $\OO_{\gamma}$ denotes the order in $K=\Q[\tau_{\gamma}]$
generated by the complex lattice $[1,\tau_{\gamma}]$. A similar statement is proven
in \cite{CK05}, Theorem~4, with $j_N$ replaced by an appropriately normalized
Hauptmodul for the genus zero group $\Gamma_0(N)^{\ast}$.

The articles \cite{CY93}, \cite{CMcKS04} and \cite{CK05} do not address the
question whether the value of $j_N(\tau_{\gamma})$ for any \emph{elliptic fixed point}
$\tau_{\gamma}$ of $\Gamma_0(N)^+$ is an algebraic integer. Partial numerical evidence
supporting this question is given in \cite{CMcKS04}. If the answer to this question is
affirmative, then, a natural follow-up problem would be to determine their minimal
polynomials whose splitting fields over the appropriate extension of the rationals
would be the ring class fields of the corresponding orders.

\subsection{Our results}

The main purpose of this paper is to answer the two questions posed above. We prove
for all genus zero groups $\Gamma_{0}(N)^{+}$ and for all corresponding elliptic
points $e$ that the singular moduli $j_N(e)$ is an algebraic integer.
Moreover, we obtain an exact evaluation in terms of radicals of each such singular
moduli after which we compute the minimal polynomials of the corresponding ring class
fields and their subfields.

Our analysis begins by studying \eqref{dif eq for j} for any genus zero group
$\Gamma_{0}(N)^{+}$. It is not difficult to show that for any Hauptmoduli $f$
on any genus zero group $\Gamma$ commensurable with $\PSL(2,\Z)$ there exists
a rational function $R_{\Gamma}(y)$ such that
$$
S(f)(z)+ R_{\Gamma}(f(z))(f'(z))^2=0;
$$
see, for example, \cite{HMcK00}, Theorem~1.1. In this paper, we specialize to the
genus zero ``moonshine groups" $\Gamma_0(N)^+$ with square-free $N$. To be precise,
we explicitly compute the rational function $R_N(y)$ such that
\begin{align*}
S(j_N)(z)+ R_N(j_N(z))(j_N'(z))^2=0.
\end{align*}
The analysis and algorithms presented in this article yield the following results.

\medskip

\textbf{Main Theorem}
\begin{it}
With the above notation, we write $R_N = P_N/Q_N$ for polynomials $P_N$ and $Q_N$.
\begin{enumerate}
\item The polynomial $P_N$ is a monic polynomial with
$\deg(P_N) = \deg(Q_N) - 2$. Furthermore,
$$
Q_N(j_N(z)) = 2 \prod_{e\in\E_N}(j_N(z)-j_N(e))^{2}
$$
where $\E_N$ is the set of inequivalent elliptic points on $\h$ with respect
to the action by $\Gamma_0(N)^+$.
\item The coeffients of $P_N$ and $Q_N$ are integers.
\item If we write $Q_{N} = 2(h_{N})^{2}$, then $h_{N}$ is a monic polynomial with integer
coefficients, thus the values of $j_N(e)$ for $e\in \E_N$ are algebraic integers.
\end{enumerate}
\end{it}

\medskip

The factorization of polynomials $h_N(y)$ into irreducible polynomials over $\Z$
is given in Table~\ref{Table h_N}. Using the $q$-expansions of $j_{N}$, which
were obtained in \cite{JST2}, we then derived a list of approximate values of $j_N(e)$
for $e\in \E_N$, as well as the roots of $h_{N}$ in terms of radicals; see
Appendices~\ref{sec A3} and \ref{sec A2}, while the list of polynomials $P_N$ and $Q_N$ is
given in Appendix~\ref{sec A1}. Finally, by pairing the values of the
roots with the approximate values of $j_N(e)$, we obtained the minimal polynomials
associated to each value of $j_N(e)$. After the above mentioned computations, we
combine with results from \cite{CMcKS04} in order to explicitly construct the
class fields of certain orders and their subfields. A summary of these results is stated in
Corollary~\ref{cor:explicit class field th.} and Tables~\ref{Table with disc} and
\ref{Table with disc-subf}.
For some class fields we get more than one generating polynomial. For example
for the class field of the order $\Z[\sqrt{-17}]$, we have three generating
polynomials over $\Q[\sqrt{-17}]$:
$$
h_{17, -4\cdot 17}= y^4+2y^3-39y^2-176y-212, \,\, h_{51, -4\cdot 17}=y^4+2y^3+3y^2 -2y+1, \,\,
\text{and} \,\, h_{119, -4\cdot 17}=y^4+2y^3+3y^2+6y+5.
$$
The notation $h_{N,D}$ is defined in section~\ref{sec 6}. Those
polynomials are all generating polynomials, with small
coefficients, of the Hilbert class field over $\Q[\sqrt{-17}]$.

In the case when the level is $71$, our computations agree with the
results of \cite{CMcKS04}, Section~4, since in this case we get the
same generating polynomials of the Hilbert class field of $\Q[\sqrt{-71}]$.

\subsection{Outline of the paper}

In section~\ref{sec:background} we cite results from the literature needed for
our paper. In section~\ref{sec:structure} we study properties of the Schwarzian
derivative $S(j_{N})$ of $j_{N}$, ultimately proving part (1) of the Main Theorem.
In section~\ref{sec 4} we describe an algorithm by which we evaluate the coefficients
of $P_{N}$ and $Q_{N}$, and compute the polynomial $h_{N}$ where $Q_{N} = 2h_{N}^{2}$.
A complete list of all the polynomials $h_{N}$ is given in Table~\ref{Table h_N}.
Finally, in section~\ref{sec 6} we discuss the applications of our results to explicit class field
theory. The result is stated in Corollary~\ref{cor:explicit class field th.}
to which we refer the reader for a precise statement.

For the convenience of the reader, we have placed in appendices lists of information
generated from our analysis.  In Appendix~\ref{sec A1} we write the polynomials
$P_{N}$ and $Q_{N}$ for all levels $N$ of genus zero groups $\Gamma_0(N)^+$.  In
Appendix~\ref{sec A2} we list the roots of each $h_{N}$, where each root is given
precisely in radicals.  In Appendix~\ref{sec A3} we list the numerical approximations
of $j_{N}(e)$ for each level $N$ of genus zero group $\Gamma_0(N)^+$
and corresponding elliptic points $e$.
As a particular example, the case $N=5$ is discussed in
the body of the paper.  When combining the discussion of the example
$N=5$ with the results listed in the three appendices, the interested
reader will arrive at similar conclusions for all other cases.

\subsection{Computer assistance}

Computer algebra was used to assist our computations. Taking results from
\cite{JST2,JST3,JST15URL} for the Hauptmoduli, we used symbolic algebra of PARI/GP
\cite{PARI2} to perform most of the algorithm of section~\ref{sec 4}.
Since we had it readily available, we used our own C-code linked against the GMP
Bignum Library \cite{GMP} to solve (\ref{differential eq}a) in rational arithmetic
for the polynomials $P_N$ and $Q_N$. Moreover, in order to produce
a part of the data in Table~\ref{Table with disc} below, related to even levels $N$,
we used Alnuth package of GAP \cite{GAP} to determine whether some irreducible factors of $h_N$
generate the same field as factors of $h_m$, for some odd divisor $m$ of $N$.

\section{Background material}\label{sec:background}

\subsection{Holomorphic modular forms}

Let $\Gamma$ be a Fuchsian group of the first kind. Following \cite{Serre73},
we define a weakly modular form $f$ of weight $2k$ for $k \geq 1$ associated
to $\Gamma$ to be a function $f$ which is meromorphic on $\h$ and satisfies
the transformation property
$$
f\left(\frac{az+b}{cz+d}\right) = (cz+d)^{2k}f(z)
\quad \text{for all $\begin{pmatrix} a&b \\ c&d \end{pmatrix} \in \Gamma$.}
$$

\noindent Assume that $\Gamma$ has at least one class of parabolic elements.
By transforming coordinates, if necessary, we may always assume that the
parabolic subgroup of $\Gamma$ has a fixed point at $i\infty$, with
identity scaling matrix. In this situation, any weakly modular form $f$
will satisfy the relation $f(z+1)=f(z)$, so we can write
$$
f(z) = \sum\limits_{n=-\infty}^{\infty}a_{n}q^{n}
\quad \text{where $q = e^{2\pi iz}$.}
$$
If $a_{n} = 0$ for all $n < 0$, then $f$ is said to be holomorphic in the
cusp. A holomorphic modular form with respect to $\Gamma$ is a weakly holomorphic modular form
which is holomorphic on $\h$ and in all of the cusps of $\Gamma$.
A weakly holomorphic modular form with respect to $\Gamma$ is called a cusp form,
if $a_{n} = 0$ for all $n \leq 0$.

\subsection{Modular forms on surfaces $X_N$}

From Proposition~7, page~II-7, of \cite{SCM66}, we immediately obtain the
following Riemann-Roch type formula which relates the number of zeros of a
modular form, counted with multiplicity, with its weight and volume of $X_N$.

\begin{proposition}\label{prop sum over zeros}
Let $f$ be a modular form on $X_N$ of weight $2k$, not identically zero.
Let $\F_N$ denote the fundamental domain of $X_N$ and let $v_z(f)$ denote the
order of zero $z$ of $f$, or minus the order of pole of $f$. Then,
\begin{align*}
k \frac{\vol(X_N)}{2\pi}
= v_{i\infty}(f) + \sum_{e \in \E_N} \frac{1}{\ord(e)} v_{e}(f)
+ \sum_{z\in \F_N \setminus \E_N} v_z(f),
\end{align*}
where $\E_N$ denotes the set of elliptic points in $\F_N$ and
$\ord(e)$ is the order of the elliptic point $e\in \E_N$.
\end{proposition}

\section{Determining the polar structure of $S(j_N)/(j_N')^2$}\label{sec:structure}

We begin with the following elementary proposition.

\begin{proposition}
For any square-free $N$ such that the group $\Gamma_0(N)^+$ has genus zero,
the function $S(j_N)(z)(j_N'(z))^2$ on $\h$ is a weight eight modular form with respect
to $\Gamma_0(N)^+$ and is holomorphic function on $\h$ whose only pole is at $i\infty$ with
order two. Furthermore, the $q$-expansion of $S(j_N)(z)(j_N'(z))^2$ is given by
\begin{align}\label{q exp S(j)j' squared}
S(j_N)(z)(j_N'(z))^2
= (2\pi)^4 \left(-\frac{1}{2q^2} + \sum_{k=0}^{\infty} b_N(k)q^k \right),
\end{align}
where, in the notation of \eqref{q exp jN},
\begin{align*}
b_N(k)= -(k+1)[(k+1)^2+3(k+1)+1]a_N(k+1)
+ \frac{1}{2}\sum_{l=1}^{k-1}l^2(k-l)(5l-3k)a_N(l)a_N(k-l).
\end{align*}
\end{proposition}

\begin{proof}
Since
$$
S(j_N)(z)(j_N'(z))^2=j_N'''(z)j_N'(z)-\frac{3}{2}(j_N''(z))^2,
$$
it is immediate that the function $S(j_N)(z)(j_N'(z))^2$ is holomorphic on
$\h$ with the only pole at $i\infty$ of order two.
From the discussion in section~\ref{sec 1.2}, we conclude that the Schwarzian $S(j_N)$
is a meromorphic modular form of weight four associated to $\Gamma_0(N)^+$. Since $j_N'(z)$ is
a meromorphic weight two form, we conclude that $S(j_N)(z)(j_N'(z))^2$ has weight eight.

Beginning with the $q$-expansion~\eqref{q exp jN} we derive the $q$-expansion
of the $l$th derivative $j_N^{(l)}(z)$ for $l\geq 1$, namely the expansion
$$
j_N^{(l)}(z)
= (2\pi i )^l\left(\frac{(-1)^l}{q} + \sum_{k=1}^{\infty} k^l a_N(k)q^k \right),
$$
hence
$$
S(j_N)(z)(j_N'(z))^2=(2\pi)^4
\left( -\frac{1}{q}+\sum_{k=1}^{\infty} k^3 a_N(k)q^k \right)
\left( -\frac{1}{q}+\sum_{k=1}^{\infty} k a_N(k)q^k \right)
-\frac{3}{2}(2\pi)^4
\left( \frac{1}{q}+\sum_{k=1}^{\infty} k^2 a_N(k)q^k \right)^2.
$$
A straightforward computation yields~\eqref{q exp S(j)j' squared}.
\end{proof}

\begin{corollary}\label{cor:weight zero}
For any square-free $N$ such that the group $\Gamma_0(N)^+$ has genus zero,
the function $S(j_N)(z)/(j_N'(z))^2$ is a weight zero modular form on $X_N$
whose zero at $i\infty$ has order two.
\end{corollary}

We now use Proposition~\ref{prop sum over zeros} in order to determine the
set of zeros of the weight two form $j_N'(z)$ and corresponding multiplicities.
From an inspection of tables given in
\cite{Cum04}, we conclude that for all square-free $N$ such that the surface $X_N$
has genus zero, the set $\E_N$ of elliptic points of $X_N$ consists of a
certain number of order two elliptic points and possibly one point of
order three, four or six. For $n\in\{3,4,6\}$, we define the symbol
$\delta_{n,N}$ to be equal to one if there exists an elliptic point on
$X_N$ of order $n$ and set $\delta_{n,N}=0$ otherwise.

\begin{proposition} \label{prop:zeros of j'}
Let $N$ be a square-free number such that the surface $X_N$ has genus zero.
Then the set of zeros of $j_N'$ is equal to the set $\E_N$ of elliptic
points of $X_N$. In the case when $X_N$ has no order four or six elliptic
points, the order $m_N(e)$ of every zero $e\in\E_N$ of $j_N'$ is
$\ord(e)-1$. In the case when $X_N$ has one order four or six
elliptic point, then either the order of all zeros $e\in\E_N$ is
$\ord(e)-1$, or the order of all but one zero at order two
elliptic points is one, there is an order two zero of $j_N'$ at some order
two elliptic point and the order of zero at order four or six
elliptic point is one or two respectively.
\end{proposition}

\begin{proof}
Let $\E_{2,N}$ denote the set of elliptic points of $X_N$ of order two.
If we insert $k=1$ into Proposition~\ref{prop sum over zeros} and combine it with the
Gauss-Bonet formula for the volume of $X_N$ we arrive at the equation
\begin{align}\label{gb_rr}
-1+\sum_{e\in\E_{2,N}}\frac{1}{2}+ \delta_{n,N}\frac{n-1}{n}
= -1+\sum_{e\in\E_{2,N}}\frac{v_e(j_N')}{2}
+ \delta_{n,N}\frac{v_{e_{n,N}}(j_N')}{n}
+ \sum_{z\in \F_N \setminus \E_N} v_z(j_N'),
\end{align}
for $n\in\{3,4,6\}$, where we denoted by $e_{n,N}$ the order $n$
elliptic point of $X_N$, if it exists.

The form $j_N'$ vanishes at all $e\in\E_{2,N}$, since the
transformation rule for $j_N'$ and arbitrary order two elliptic element
$\eta\in\Gamma_0(N)^+$ with fixed point $e$ reduces to
$$
j_N'(\eta(e))=j_N'(e)=(i)^2j_N'(e)
$$
which implies that $j_N'(e)=0$. In other words,
$v_e(j_N')\geq 1$ for all $e\in\E_{2,N}$. From \eqref{gb_rr}, we have
that $v_z(j_N')=0$ for
all $z\in \F_N \setminus \E_N$. In the case when $N$ is such that the
surface $X_N$ possesses no order four or six elliptic points, we immediately
deduce that the order of zero $e\in\E_N$ is $\ord(e)-1$. In the case when
$\delta_{4,N}=1$, either order of all zeros $e\in\E_N$ of $j_N'$ is
$\ord(e)-1$, or, writing $3/4$ as $1/2+1/4$ we deduce that there is an order
two elliptic point $e'\in\E_N$ which is the order two zero of $j_N'$ and the
order of zero at the order four elliptic point $e_{4,N}$ of $X_N$ is one.
A similar argument proves the statement in the case when $X_N$ has one order
six elliptic point.
\end{proof}

We can now determine the number and location of poles of
the meromorphic modular form $S(j_N)(z)/(j_N'(z))^2$.

\begin{proposition}\label{prop poles of S/j'2}
Let $N$ be a square-free number such that the group $\Gamma_0(N)^+$ has genus
zero. Then, the set of poles of the meromorphic modular form
$S(j_N)(z)/(j_N'(z))^2$ is exactly the set $\E_N $ of elliptic points of
$X_N$. Moreover, each pole has the order equal to $2(m_N(e)+1)$, where
$m_N(e)$ denotes the order of the elliptic point $e$ as a zero of $j_N'$.
\end{proposition}
\begin{proof}
We write
$$
\frac{S(j_N)(z)}{(j_N'(z))^2}
= \frac{j_N'''(z)j_N'(z)-\frac{3}{2}(j_N''(z))^2}{(j_N'(z))^4}.
$$
Obviously, $S(j_N(z))$ is holomorphic everywhere, except eventually at zeros of $j_N'(z)$.
Therefore, by Proposition~\ref{prop:zeros of j'}, the set of poles of $S(j_N)(z)/(j_N'(z))^2$
is exactly $\E_N$.

If $e\in\E_N$ is a zero of $j_N'$ of order $m_N(e)$, then by studying the
power series expansion about $e$ we see that
$j_N'''(z)j_N'(z)-\frac{3}{2}(j_N''(z))^2$ is a weight eight meromorphic form
with zero at $e$ of order $2(m_N(e)-1)$. In particular, if $m_N(e)=1$
then the form is non-vanishing at $e$. Therefore, the order of the
pole of $S(j_N)(z)/(j_N'(z))^2$ at $e\in\E_N$ is
$$
4m_N(e)-2(m_N(e)-1)=2(m_N(e)+1),
$$
as claimed.
\end{proof}

\begin{theorem}\label{theorem 6}
For any square-free $N$ such that the group $\Gamma_0(N)^+$ has genus zero,
there exists an integer $n_N\geq 1$ and polynomials $P_N$ and $Q_N$ of
degrees $n_N$ and $n_N+2$ respectively such that
\begin{align}\label{differential eq}
\frac{S(j_N)(z)}{(j_N'(z))^2}=-\frac{P_N(j_N(z))}{Q_N(j_N(z))}=-R_N(j_N(z)).
\end{align}
Moreover, we may take
\begin{align}\label{Q N formula}
Q_N(y)=2\prod_{e\in\E_N}(y-j_N(e))^{2}.
\end{align}
which implies that the lead coefficient of $P_N$ is equal to one and
\begin{align}\label{n N formula}
n_N\leq 2\left(\sum_{e\in\E_N}1-1 \right).
\end{align}
\end{theorem}

\begin{proof}
Since $j_N$ is the Hauptmodul, Corollary~\ref{cor:weight zero} implies that
$$
\frac{S(j_N)(z)}{(j_N'(z))^2}=\frac{S(j_N)(z)(j_N'(z))^2}{(j_N'(z))^4}
=-\frac{P_{m}(j_N)}{Q_{r}(j_N)},
$$
for some polynomials $P_m$ and $Q_r$ of degrees $m$ and $r$. The fact
that $S(j_N)(z)/(j_N'(z))^2$ possesses a zero at $i\infty$ of order two,
together with the expansion~\eqref{q exp jN} yields that $r=m+2$, so then
$n_{N} = r$.

Let us now look at the multiplicities of zeros of $Q_N(j_N(z))$ defined
by $\eqref{Q N formula}$. Obviously, the set of zeros of $Q_N(j_N(z))$
coincides with the set of poles of $S(j_N)(z)/(j_N'(z))^2$. Moreover, if
$e\in\E_N$ is a zero of $j_N'$ of order $m_N(e)$, then we have the local expression
$$
j_N(z)-j_N(e)=\frac{1}{(m_N(e)+1)!}(z-e)^{m_N(e)+1}g_{N,e}(z),
$$
where $g_{N,e}(z)$ is a holomorphic function such that $g_{N,e}(e)\neq 0$.
Therefore, the point $e$ is a zero of $Q_N(j_N(z))$ of order $2(m_N(e)+1)$.
This, together with Proposition~\ref{prop poles of S/j'2} proves that the set
of poles of $S(j_N)(z)/(j_N'(z))^2$ with corresponding orders coincides with
the set of zeros of $Q_N(j_N(z))$ with corresponding orders. Consequently, we may take
$Q_r$ to be defined by~\eqref{Q N formula}.

It remains to prove that by taking $Q_N(y)$ to be given
by~\eqref{Q N formula}, we then have that the lead coefficient of $P_N$ is
equal to one. From the $q$-expansion~\eqref{q exp S(j)j' squared} and the fact that the
$q$-expansion of $(j_N'(z))^4$ begins with $(2\pi)^4 q^{-4}$ we see that the
$q$-expansion of $S(j_N)(z)/(j_N'(z))^2$ begins with $-\frac{1}{2}q^2$.
Since the $q$-expansion of $j_N$ is normalized so it begins with $q^{-1}$
it is obvious that the $q$-expansion of $Q_N(j_N)$ begins with
$2q^{-2}\cdot q^{-n_N}$, therefore, taking the lead coefficient of $P_N$
to be equal to one we get that the $q$-expansion of the right hand side
of~\eqref{differential eq} begins with $-\frac{1}{2}q^2$, which
implies that $P_{N}$ is monic.

Finally, the bound \eqref{n N formula} for $n_{N}$ follows from the above proved product expansion
for $Q_{N}$, taking into account that $P_{N}$ and $Q_{N}$ may have common factors.
\end{proof}

\section{Evaluating the coefficients of $R_N(y)$}\label{sec 4}

\subsection{An algorithm}

There are three computational results to be obtained through computer assistance:
The first evaluates the coefficients of both the numerator and the denominator
of $R_N$; the second which determines the roots of the $Q_N$, the denominator
of $R_N$, and the third approximates $j_N$ at each elliptic point numerically, so then
we can determine the specific values of $j_N(e)$.

The algorithm for computation of polynomials $P_N(x)$ and $Q_N(x)$ is
the following.

\begin{enumerate}
\item[Step 1.] Fix $N$. Obtain from \cite{Cum04} the set $\E_N$ of
elliptic points of $X_N$ and, in an abuse of notation, define $n_N$ by taking equality
in~\eqref{n N formula}.
\item[Step 2.] Use the exact expression of the Hauptmodul $j_N(z)$ in terms
of Eisenstein series and the Kronecker limit function \cite{JST2} to compute
the $q$-expansion of $j_N(z)$, truncated at $O(q^{2n_N+2})$.
\item[Step 3.] Derive the $q$-expansions of $\displaystyle \frac{j_N'}{2\pi i}$,
$\displaystyle \frac{j_N''}{(2\pi i)^2}$, $\displaystyle \frac{j_N'''}{(2\pi i)^3}$ and compute
the $q$-expansions of
$\displaystyle \frac{2S(j_N)(j_N')^2}{(2\pi)^4}$ and $\displaystyle \frac{(j_N')^4}{(2\pi)^4}$.
\item[Step 4.] Define
\begin{align*}
&P_N(x)=x^{n_N}+\sum_{k=0}^{n_N-1}A_kx^k,\\
&Q_N(x)=2\big(x^{n_N+2}+\sum_{k=0}^{n_N+1}B_kx^k\big),
\end{align*}
and solve
\begin{align*}\tag{\ref{differential eq}a}
\frac{1}{2}Q_{N}(j_{N})\frac{2S(j_{N})(j_{N}')^2}{(2\pi)^4} + P_{N}(j)\frac{(j_{N}')^4}{(2\pi)^4} = 0
\end{align*}
for the coefficients of $P_N$ and $Q_N$ by setting each coefficient in the $q$-expansion
of (\ref{differential eq}a) equal to zero.
\end{enumerate}

\medskip

By Theorem~\ref{theorem 6} the coefficient of $q^{-n_N-4}$ in the
$q$-expansion of (\ref{differential eq}a) vanishes identically.
Comparing coefficients of $q^{-n_N-3},\ldots,q^{n_N-2}$ in the $q$-expansion
of (\ref{differential eq}a) results in $(2n_N+2)$ linear equations for the
$(2n_N+2)$ unknowns $\{A_k\}_{k=0}^{n_N-1}$ and $\{B_k\}_{k=0}^{n_N+1}$.
After implementing the algorithm, these equations turn out to be linearly independent
for each $N$, hence there is a unique solution for the coefficients of $P_N$ and $Q_N$.

For each square-free integer $N$ such that the arithmetic group
$\Gamma_0(N)^+$ has genus zero, we have evaluated the polynomials
$P_N$ and $Q_N$ as described above. We list the polynomials $P_N$ and $Q_N$ in
Appendix~\ref{sec A1} and observe that all coefficients are integers.
Moreover, $P_N$ and $\frac{1}{2}Q_N$ are monic polynomials with integer coefficients,
and we conclude that their roots are \emph{algebraic integers}.

Theorem~\ref{theorem 6}, equation~\eqref{Q N formula}, connects the roots
of the monic polynomials $\frac{1}{2}Q_N$ to the values of the Hauptmoduli
$j_N$ at the elliptic points of $X_N$. As a consequence, we have that
each $j_N(e)$ is an \emph{algebraic integer} at each elliptic point $e\in\E_N$
of the respective surface $X_N$.

\subsection{Hauptmodul values at elliptic points}

Having proven that the values of the Hauptmoduli at elliptic points
are algebraic integers, we now compute these algebraic integers
explicitly. As above, let us write
$h_N(y):=\big(\frac{1}{2}Q_N(y)\big)^{1/2}$.
According to~\eqref{Q N formula}, the function $h_N$ is a monic polynomial and has the same
roots as the polynomial $Q_N$. It remains to compute the roots.

To begin, we factor $h_N$ into irreducible polynomials and use computer algebra to
find explicit expressions for the roots in terms of radicals. From the expressions
for $Q_{N}$, we obtain the following list for $h_{N}$, see Table~\ref{Table h_N}.

\medskip

\begin{longtable}{l}
\caption{\label{Table h_N}The list of monic polynomials $h_N$ factored into
irreducible polynomials.} \\
\hline
$h_{1}(y) = (y+744)(y-984)$ \\
$h_{2}(y) = (y+104)(y-152)$ \\
$h_{3}(y) = (y+42)(y-66)$ \\
$h_{5}(y) = (y+16)(y^2-12y-464)$ \\
$h_{6}(y) = (y+14)(y+10)(y-22)$ \\
$h_{7}(y) = (y+10)(y+9)(y-18)$ \\
$h_{10}(y) = (y+8)(y+4)(y-12)$ \\
$h_{11}(y) = (y+6)(y^3-2y^2-76y-212)$ \\
$h_{13}(y) = (y+4)(y+3)(y^2-4y-48)$ \\
$h_{14}(y) = (y+6)(y+2)(y^2-6y-23)$ \\
$h_{15}(y) = (y+4)(y-8)(y^2+6y+13)$ \\
$h_{17}(y) = (y+2)(y^4+2y^3-39y^2-176y-212)$ \\
$h_{19}(y) = (y+4)(y+3)(y^3-4y^2-16y-12)$ \\
$h_{21}(y) = (y+4)(y-0)(y^2-2y-27)$ \\
$h_{22}(y) = (y+2)(y-6)(y^3+6y^2+8y+4)$ \\
$h_{23}(y) = (y^3+6y^2+11y+7)(y^3-2y^2-17y-25)$ \\
$h_{26}(y) = (y+4)(y-0)(y^3-2y^2-15y-16)$ \\
$h_{29}(y) = (y+2)(y^6+2y^5-17y^4-66y^3-83y^2-32y-4)$ \\
$h_{30}(y) = (y+4)(y+3)(y-0)(y-1)(y-5)$ \\
$h_{31}(y) = (y-0)(y^3+4y^2+3y+1)(y^3-17y-27)$ \\
$h_{33}(y) = (y-0)(y^2-2y-11)(y^3+4y^2+8y+4)$ \\
$h_{34}(y) = (y+2)(y+1)(y^2+3y-2)(y^2-5y+2)$ \\
$h_{35}(y) = (y+2)(y^3-2y^2-4y-20)(y^2+2y+5)$ \\
$h_{38}(y) = (y-0)(y^3+4y^2+4y+4)(y^3-2y^2-7y-8)$ \\
$h_{39}(y) = (y+3)(y-1)(y^2+3y-1)(y^2-5y+3)$ \\
$h_{41}(y) = (y-0)(y^8+4y^7-8y^6-66y^5-120y^4-56y^3+53y^2+36y-16)$ \\
$h_{42}(y) = (y+3)(y-0)(y-1)(y-4)(y^2+3y+4)$ \\
$h_{46}(y) = (y^2-2y-7)(y^3+2y^2+y+1)(y^3+2y^2-3y+1)$ \\
$h_{47}(y) = (y^5+4y^4+7y^3+8y^2+4y+1)(y^5-5y^3-20y^2-24y-19)$ \\
$h_{51}(y) = (y+2)(y^3-2y^2-4y-4)(y^4+2y^3+3y^2-2y+1)$ \\
$h_{55}(y) = (y+1)(y^2+3y+1)(y^2-5y+5)(y^3+3y^2-y-7)$ \\
$h_{59}(y) = (y^3+2y^2+1)(y^9+2y^8-4y^7-21y^6-44y^5-60y^4-61y^3-46y^2-24y-11)$ \\
$h_{62}(y) = (y^3+4y^2+5y+3)(y^3+y-1)(y^4-2y^3-3y^2-4y+4)$ \\
$h_{66}(y) = (y+3)(y-0)(y-1)(y^2-y-8)(y^3-4y+4)$ \\
$h_{69}(y) = (y^3+4y^2+7y+5)(y^3-y+1)(y^4-2y^3-5y^2+6y-3)$ \\
$h_{70}(y) = (y+2)(y+1)(y-3)(y^2-y+2)(y^3+2y^2+4)$ \\
$h_{71}(y) = (y^7+4y^6+5y^5+y^4-3y^3-2y^2+1)(y^7-7y^5-11y^4+5y^3+18y^2+4y-11)$ \\
$h_{78}(y) = (y+1)(y-3)(y^2+y+1)(y^2+y-3)(y^3+y^2-4)$ \\
$h_{87}(y) = (y^3+2y^2+3y+3)(y^3-2y^2-y-1)(y^6+2y^5+7y^4+6y^3+13y^2+4y+8)$ \\
$h_{94}(y) = (y^4-2y^3-3y^2+4y-4)(y^5+4y^4+3y^3-2y^2+2y+5)(y^5-y^3+2y^2-2y+1)$ \\
$h_{95}(y) = (y-1)(y^3+y^2-y+3)(y^4+y^3-2y^2+2y-1)(y^4+y^3-6y^2-10y-5)$ \\
$h_{105}(y) = (y-1)(y^2+3y+3)(y^2-y-1)(y^2-y-5)(y^3+y^2-y-5)$ \\
$h_{110}(y) = (y-1)(y^2+y+3)(y^2+y-1)(y^3+y^2+3y-1)(y^3-y^2-8)$ \\
$h_{119}(y) = (y^4+2y^3+3y^2+6y+5)(y^5+2y^4+3y^3+6y^2+4y+1)(y^5-2y^4+3y^3-6y^2-7)$ \\
\hline
\end{longtable}

\medskip

For each level $N$, one then needs to identify different roots of $h_N(y)$
and match the roots with approximate values of the Hauptmoduli $j_N(z)$. The lists of roots of
$h_N(y)$ and approximate values of the Hauptmoduli $j_N(z)$ at elliptic points are given in
Appendices~\ref{sec A2} and \ref{sec A3}, respectively.

\subsection{An example: $N=5$}\label{sec 5.2}

The polynomial $h_{5}$ has three roots. The roots are $y=-16$, and $y=6\pm\sqrt{500}$.

The surface $X_{5}$ has exactly three elliptic points $e\in\E_{5}$,
all of order two. If we identify the numerical values of the
Hauptmodul at the elliptic points with the values of the roots of $h_{5}$,
we determine which value of $j_{5}(e)$ corresponds to which root. Table~\ref{Table j5}
summarizes the end results.

\begin{longtable}{lll}
\caption{\label{Table j5}Identification of the values of $j_{5}(e)$ with the exact roots of
the polynomial $h_{5}$.} \\
$j_{5}(e)$ & numerical approximation & exact value \\
\hline
$j_{5}(i\frac{\sqrt{5}}{5})$ & $\ \ \,28.36$ & $6+\sqrt{500}$ \\
$j_{5}(\frac{2}{5}+i\frac{2}{5})$ & $-16.00$ & $-16$ \\
$j_{5}(\frac{1}{2}+i\frac{\sqrt{5}}{10})$ & $-16.36$ & $6-\sqrt{500}$ \\
\hline
\end{longtable}

\section{An application to explicit class field theory}\label{sec 6}

Let $e\in\E_N$ be an order two element which is not a fixed point of some
$\gamma \in \overline{\Gamma_0(N)}$. Then, in a slight abuse of notation, the point
$e\in\h$ is a fixed point of the order two element
\begin{align}\label{gamma e}
\gamma_e:= \begin{pmatrix} a_e\sqrt{v} & b_e/\sqrt{v} \\ c_eN/\sqrt{v} & -a_e\sqrt{v} \end{pmatrix}
\in\overline{\Gamma_0(N)^+} \setminus \overline{\Gamma_0(N)}, \,\,\,\,\, \text{for some $v\mid N$.}
\end{align}
Equivalently, we see that $e\in \h$ is a zero of the quadratic polynomial
$$
f_{\gamma_e}(X):=c_eNX^2-2a_evX-b_e\in\Z[X],
$$
where we may assume that $c_e>0$. The polynomial $f_{\gamma_e}(X)$ is
irreducible in $\Z[X]$ with discriminant $-4v$ if
$v\equiv 3\mod 4$ and either $b_e$ and $c_e$ or $b_e$ and $N$ are
both even. Otherwise, the polynomial $\frac{1}{2}f_{\gamma_e}(X)$
is irreducible in $\Z[X]$ with discriminant $-v$. Therefore, the
complex lattice $\el_{e}$ generated by $e$ and $1$ is an invertible
ideal for the quadratic order
$$
\OO_e= \begin{cases} \Z[\frac{v+\sqrt{-v}}{2}], & \text{for $v\equiv 3\mod 4$ and $b_e,c_e$
both even or $b_e,N$ both even;} \\ \Z[\sqrt{-v}], & \text{otherwise.}
\end{cases}
$$
With a slight abuse of notation, we will also say that $-4v$ and $-v$, respectively, are
discriminants of the element $e$.

In the case when $N$ is prime, from \cite{CMcKS04}, Lemma~2.2 we
have that the ideals $\el_e$ represent all ideal classes of
$\OO_e$. Therefore, in this case the class number of $\OO_e$ is
equal to the number of ideals $\el_e$. In other words, the class
number of the order $\Z[\frac{N+\sqrt{-N}}{2}]$ is equal to the
number of elements $e\in\E_N$ which are fixed points of order two
elliptic elements $\gamma_e$ given by \eqref{gamma e} for which
$N\equiv 3\mod 4$ and $b_e$ and $c_e$ are both even. Analogously,
the class number of the order $\Z[\sqrt{-N}]$ is equal to the
number of elements $e\in\E_N$ which are fixed points of order two
elliptic elements $\gamma_e$ given by \eqref{gamma e} for which
the polynomial $f_{\gamma_e}(X)$ is an irreducible polynomial of
discriminant $-4N$.

In the case when the level $N$ is composite and, set $N_1=N/v$ for
any proper divisor $1<v<N$ of $N$. Then we can write
$f_{\gamma_e}(X) = (c_eN_1)v X^2-2a_evX-b_e$. The number $N_1$ is
odd, hence the numbers $c_eN_1$ and $c_e$ are of the same parity.
Every element of $\Gamma_0(N)^+$ of the form \eqref{gamma e}
belongs to $\Gamma_0(v)^+$, hence, for any divisor $v$ of odd
level $N$, the number of elements of $\E_N$ with discriminant
equal $-4v$ or $-v$ is less than or equal to the number of
elements of $\E_v$ with the same discriminant. In the case when
the level $N$ is even, for any divisor $2<v<N$ the parity of
$c_eN/v$ changes. The discriminant of the corresponding
irreducible polynomial may change as well, so the above statement
may not be true. For example, when $N=10$ there is only one
elliptic element $e=1/2 + i \sqrt{5}/10 \in \E_5 \cap \E_{10} $
which has discriminant $-4\cdot 5$, while there are two elements
of $\E_5$ with the same discriminant. On the other hand, there
are six elements of $\E_{46}$ with discriminant $-23$ while there
are only three elements of $\E_{23}$ with the same discriminant.

Let $\widetilde{\E}_N$ denote the set of all elements of $\E_N$
which are order two and which are not fixed points of an order two
elliptic element in $\Gamma_0(N)$. We have computed elliptic
points in $\widetilde{\E}_N$ and their discriminants for all $43$
square-free levels $N>1$. It turned out that in all cases, except
for the level $N=46$ one has a bijection between the ideal classes
of orders $\OO_e$ and ideals $\el_{e}$ generated by $e$ and $1$,
where $e \in \widetilde{\E}_N$ is a fixed point of the transformation
\eqref{gamma e} whose discriminant equals either $-4v$ or $-v$ when
$v\mid N$ is prime. Moreover, the number of elliptic elements with
discriminants equal to $-4v$ or $-v$ when $v$ is composite number
with $l$ prime factors (which must be distinct, by our assumption
on levels $N$) is equal to the class number of the corresponding
order divided by $2^{l-1}$. In case when $N=46$ and $v=23$, there
are $6$ elliptic points in $\E_{46}$ with discriminant $-23$. Later,
we will see that we may group those points into two groups, each
consisting of three points according to the factorization of
$h_{46}(z)$ into irreducible factors. Tables~\ref{Table with disc}
and \ref{Table with disc-subf} provide lists of the discriminants
associated to elements of $\widetilde{\E}_{N}$.

Therefore, for all $N\neq 46,$ we can write $\widetilde{\E}_N$ as
the disjoint union
$$
\widetilde{\E}_N = \biguplus_{v \mid N} \left(\widetilde{\E}_{N,-4v} \cup
\widetilde{\E}_{N,-v} \right),
$$
where $\widetilde{\E}_{N,-D}$ denotes the set of all elements of
$\widetilde{\E}_N$ which are zeros of irreducible polynomials
$f_{\gamma_e}(X)$ in $\Z[X]$ with discriminant $-D$. It is
possible that some sets $\widetilde{\E}_{N,-D}$ are empty. Let
$m_{N,-D}$ denote the number of elements of
$\widetilde{\E}_{N,-D}$. From Tables~\ref{Table with disc} and
\ref{Table with disc-subf}, we deduce that, in the case when a
divisor $v>1$ of $N$ is prime, the number $m_{N,-4v}$ is either
zero or equals the class number of the order $\Z[\sqrt{-v}]$.
Likewise, the non-zero number $m_{N,-v}$ is the class number of
the order $\Z[\frac{v+\sqrt{-v}}{2}]$. When $v$ is a product of
$l$ prime factors, the non-zero number $2^{l-1}m_{N,-4v}$ is equal
to the class number of the order $\Z[\sqrt{-v}]$ while the
non-zero number $2^{l-1}m_{N,-v}$ equals the class number of the
order $\Z[\frac{v+\sqrt{-v}}{2}]$. (This follows from the fact
that in the latter case, $\Gamma_0(v)^{\ast}$ is a subgroup of
index $2^{l-1}$ in $\overline{\Gamma_0(v)^+}$.)

From an inspection of the values of $j_N(e)$ for $e\in\E_N$,
$N\neq 46$, listed in Appendix~\ref{sec A3}, one sees the
following. For any element $e\in\widetilde{\E}_{N,-D}\neq
\emptyset$, where $D=4v$ or $D=v$ for some $v\mid N$, the value
$j_N(e)$ is the zero of an irreducible polynomial $h_{N,
D}(y)\in\Z[y]$ of degree $m_{N,-D}$ which is a factor of $h_N(y)$.
In other words,
$$
h_N(y) = \prod_{v \mid N} h_{N, 4v}(y) h_{N, v}(y),
$$
where, in the case when $\widetilde{\E}_{N,-D}= \emptyset$ for
$D=4v$ or $D=v$ we put $h_{N, D}(y) \equiv 1$. In the case when
$N=46,$ the polynomial $h_{46}$ is the product of three irreducible,
monic polynomials: two polynomials of degree 3, whose zeros are
values of $j_{46}$ at elliptic points with discriminant $-23$ and
one degree two polynomial, whose zeros are values of $j_{46}$ at
elliptic points with discriminant $-4 \cdot 46$. With a slight
ambiguity in the notation we will denote both such degree 3
polynomials by $h_{46,-23}$.

The results of \cite{CMcKS04} imply that $j_N(e)$ for all
$e\in\widetilde{\E}_N$ belong to the ring class fields of the
corresponding orders. Moreover, for prime divisors $v$ of $N$, the
non-trivial polynomials $h_{N,-4v}$ and $h_{N,-v}$ are irreducible
polynomials of degree equal to the class number, hence they are
generating polynomials of the ring class field of the
corresponding order over $\Q[\sqrt{-v}]$. All non-trivial
polynomials $h_{N,-4v}$ and $h_{N,-v}$ with prime $v$ are listed
in Table~\ref{Table with disc}.

For composite divisors $v$ of $N$ with $l\geq 2$ prime factors,
non-trivial polynomials $h_{N,-4v}$ and $h_{N,-v}$ are irreducible
polynomials (over $\Q[\sqrt{-v}]$) of a degree equal to the class
number of the corresponding order divided by $2^{l-1}$ and hence
generate subfields of the corresponding class fields of index
$2^{l-1}$. The list of such polynomials is provided in Table~\ref{Table with disc-subf}.

The above discussion, together with results of \cite{CMcKS04}, yields the
following corollary of the Main Theorem.

\begin{corollary}\label{cor:explicit class field th.}
Let $N$ be a square-free number such that $X_N$ has genus
zero. With the above notation the following statements hold true:
\begin{itemize}
\item[(i)] $j_N(e)\in\Z$ for all elliptic elements $e\in\E_N$ of order three
or higher and for all order two elements of $\E_N$ which are
fixed points of some order two elliptic elements from $\overline{\Gamma_0(N)}$.

\item[(ii)] Let $v\mid N$ be a prime such that $\widetilde{\E}_{N,-4v}\neq \emptyset$.
Then the polynomial $ h_{N, 4v}(y)$ is the generating polynomial of
the ring class field of the order $\Z[\sqrt{-v}]$ over $\Q[\sqrt{-v}]$.

\item[(iii)] Let $v\mid N$ be a prime such that $\widetilde{\E}_{N,-v}\neq \emptyset$.
Then the polynomial $ h_{N, v}(y)$ is the generating polynomial of
the ring class field of the order $\Z[\frac{v+\sqrt{-v}}{2}]$ over $\Q[\sqrt{-v}]$.

\item[(iv)] Let $v\mid N$ be a composite number with $l$ prime factors such that
$\widetilde{\E}_{N,-4v}\neq \emptyset$. Then, the polynomial $ h_{N, 4v}(y)$ generates the
subfield of the ring class field of the order $\Z[\sqrt{-v}]$ over $\Q[\sqrt{-v}]$ of
index $2^{l-1}$ in the ring class field of $\Z[\sqrt{-v}]$.

\item[(v)] Let $v\mid N$ be a composite number with $l$ prime factors such that
$\widetilde{\E}_{N,-v}\neq \emptyset$. Then, the polynomial $ h_{N, v}(y)$ generates the
subfield of the ring class field of the order $\Z[\frac{v+\sqrt{-v}}{2}]$ over $\Q[\sqrt{-v}]$
of index $2^{l-1}$ in the ring class field of $\Z[\frac{v+\sqrt{-v}}{2}]$.
\end{itemize}
\end{corollary}

\medskip

\begin{longtable}{ccccl}
\caption{\label{Table with disc}The table lists for all genus zero square-free
levels $N>1$ the discriminant $D$
of the order two element from $\overline{\Gamma_0(N)^+} \setminus \overline{\Gamma_0(N)}$,
provided $D=4v$ and $D=v$, respectively, and $v$ is a prime divisor of $N$.
Listed is also the order generated by the elliptic element $e$ such that $f_{\gamma_e}$
and $\frac{1}{2}f_{\gamma_e}$, respectively, has discriminant $D$,
the class number of the order, and the generating polynomial of the
class field of the order over the corresponding imaginary quadratic
extension of $\Q$.} \\
$N$ & $D$ & Order & Class number & Generating polynomial \\
\hline
$2$ & $-4\cdot 2$ & $\Z[\sqrt{-2}]$ & $1$ & $y-152$ \\
$3$ & $-4\cdot 3$ & $\Z[\sqrt{-3}]$ & $1$ & $y-66$ \\
$5$ & $-4\cdot 5$ & $\Z[\sqrt{-5}]$ & $2$ & $y^2-12y-464$ \\
$7$ & $-4\cdot 7$ & $\Z[\sqrt{-7}]$ & $1$ & $y-18$ \\
$7$ & $-7$ & $\Z[7/2 + \sqrt{-7}/2]$ & $1$ & $y+10$ \\
$11$ & $-4\cdot 11$ & $\Z[\sqrt{-11}]$ & $3$ & $y^3-2y^2-76y-212$ \\
$11$ & $-11$ & $\Z[11/2 + \sqrt{-11}/2]$ & $1$ & $y+6$ \\
$13$ & $-4\cdot 13$ & $\Z[\sqrt{-13}]$ & $2$ & $y^2-4y-48$ \\
$15$ & $-4\cdot 5$ & $\Z[\sqrt{-5}]$ & $2$ & $y^2+6y+13$ \\
$17$ & $-4\cdot 17$ & $\Z[\sqrt{-17}]$ & $4$ & $y^4+2y^3-39y^2-176y-212$ \\
$19$ & $-4\cdot 19$ & $\Z[\sqrt{-19}]$ & $3$ & $y^3-4y^2-16y-12$ \\
$19$ & $-19$ & $\Z[19/2 + \sqrt{-19}/2]$ & $1$ & $y+4$ \\
$21$ & $-4\cdot 3$ & $\Z[\sqrt{-3}]$ & $1$ & $y+4$ \\
$22$ & $-4\cdot 11$ & $\Z[\sqrt{-11}]$ & $3$ & $y^3+6y^2+8y+4$ \\
$23$ & $-4\cdot 23$ & $\Z[\sqrt{-23}]$ & $3$ & $y^3-2y^2-17y-25$ \\
$23$ & $-23$ & $\Z[23/2 + \sqrt{-23}/2]$ & $3$ & $y^3+6y^2+11y+7$ \\
$29$ & $-4\cdot 29$ & $\Z[\sqrt{-29}]$ & $6$ & $y^6+2y^5-17y^4-66y^3-83y^2-32y-4$ \\
$31$ & $-4\cdot 31$ & $\Z[\sqrt{-31}]$ & $3$ & $y^3-17y-27$ \\
$31$ & $-31$ & $\Z[31/2 + \sqrt{-31}/2]$ & $3$ & $y^3+4y^2+3y+1$ \\
$33$ & $-4\cdot 11$ & $\Z[\sqrt{-11}]$ & $3$ & $y^3+4y^2+8y+4$ \\
$33$ & $-11$ & $\Z[11/2 + \sqrt{-11}/2]$ & $1$ & $y$ \\
$35$ & $-4\cdot 5$ & $\Z[\sqrt{-5}]$ & $2$ & $y^2+2y+5$ \\
$39$ & $-4\cdot 3$ & $\Z[\sqrt{-3}]$ & $1$ & $y-1$ \\
$41$ & $-4\cdot 41$ & $\Z[\sqrt{-41}]$ & $8$ & $y^8+4y^7-8y^6-66y^5-120y^4-56y^3+53y^2+36y-16$ \\
$46$ & $-23$ & $\Z[23/2 + \sqrt{-23}/2]$ & $3$ & $y^3+2y^2+y+1$ \\
$46$ & $-23$ & $\Z[23/2 + \sqrt{-23}/2]$ & $3$ & $y^3+2y^2-3y+1$ \\
$47$ & $-4\cdot 47$ & $\Z[\sqrt{-47}]$ & $5$ & $y^5-5y^3-20y^2-24y-19$ \\
$47$ & $-47$ & $\Z[47/2 + \sqrt{-47}/2]$ & $5$ & $y^5+4y^4+7y^3+8y^2+4y+1$ \\
$51$ & $-4\cdot 17$ & $\Z[\sqrt{-17}]$ & $4$ & $y^4+2y^3+3y^2-2y+1$ \\
$55$ & $-4\cdot 11$ & $\Z[\sqrt{-11}]$ & $3$ & $y^3+3y^2-y-7$ \\
$55$ & $-11$ & $\Z[11/2 + \sqrt{-11}/2]$ & $1$ & $y+1$ \\
$59$ & $-4\cdot 59$ & $\Z[\sqrt{-59}]$ & $9$ & $y^9+2y^8-4y^7-21y^6-44y^5-60y^4-61y^3-46y^2-24y-11$ \\
$59$ & $-59$ & $\Z[59/2 + \sqrt{-59}/2]$ & $3$ & $y^3+2y^2+1$ \\
$62$ & $-4\cdot 31$ & $\Z[\sqrt{-31}]$ & $3$ & $y^3+4y^2+5y+3$ \\
$62$ & $-31$ & $\Z[31/2 + \sqrt{-31}/2]$ & $3$ & $y^3+y-1$ \\
$66$ & $-4\cdot 11$ & $\Z[\sqrt{-11}]$ & $3$ & $y^3-4y+4$ \\
$69$ & $-4\cdot 23$ & $\Z[\sqrt{-23}]$ & $3$ & $y^3+4y^2+7y+5$ \\
$69$ & $-23$ & $\Z[23/2 + \sqrt{-23}/2]$ & $3$ & $y^3-y+1$ \\
$71$ & $-4\cdot 71$ & $\Z[\sqrt{-71}]$ & $7$ & $y^7-7y^5-11y^4+5y^3+18y^2+4y-11$ \\
$71$ & $-71$ & $\Z[71/2 + \sqrt{-71}/2]$ & $7$ & $y^7+4y^6+5y^5+y^4-3y^3-2y^2+1$ \\
$87$ & $-4\cdot 29$ & $\Z[\sqrt{-29}]$ & $6$ & $y^6+2y^5+7y^4+6y^3+13y^2+4y+8$ \\
$94$ & $-4\cdot 47$ & $\Z[\sqrt{-47}]$ & $5$ & $y^5+4y^4+3y^3-2y^2+2y+5$ \\
$94$ & $-47$ & $\Z[47/2 + \sqrt{-47}/2]$ & $5$ & $y^5-y^3+2y^2-2y+1$\\
$95$ & $-4\cdot 19$ & $\Z[\sqrt{-19}]$ & $3$ & $y^3+y^2-y+3$ \\
$95$ & $-19$ & $\Z[19/2 + \sqrt{-19}/2]$ & $1$ & $y-1$ \\
$105$ & $-4\cdot 5$ & $\Z[\sqrt{-5}]$ & $2$ & $y^2-y-1$ \\
$110$ & $-4\cdot 11$ & $\Z[\sqrt{-11}]$ & $3$ & $y^3+y^2+3y-1$ \\
$119$ & $-4\cdot 17$ & $\Z[\sqrt{-17}]$ & $4$ & $y^4+2y^3+3y^2+6y+5$ \\
\hline
\end{longtable}

\medskip

\begin{longtable}{cccccl}
\caption{\label{Table with disc-subf}The table lists for all genus zero square-free
levels $N>1$ the discriminant $D$
of the order two element from $\overline{\Gamma_0(N)^+} \setminus \overline{\Gamma_0(N)}$,
provided $D=4v$ and $D=v$, respectively, and $v$ is a composite divisor of $N$.
Also listed is the order generated by the elliptic element $e$ such that $f_{\gamma_e}$
and $\frac{1}{2}f_{\gamma_e}$, respectively, has discriminant $D$, the class number of the order,
the index of the subfield in the class field, and the generating polynomial of the subfield of
the class field of the order over the corresponding imaginary quadratic
extension of $\Q$.} \\
$N$ & $D$ & Order & Class Number& Index & Generating polynomial \\
\hline
$6$ & $-4\cdot 6$ & $\Z[\sqrt{-6}]$ & $2$ & $2$ & $y-22$ \\
$10$ & $-4\cdot 10$ & $\Z[\sqrt{-10}]$ & $2$ & $2$ & $ y-12$ \\
$14$ & $-4\cdot 14$ & $\Z[\sqrt{-14}]$ & $4$ & $2$ & $y^2-6y-23$ \\
$15$ & $-4\cdot 15$ & $\Z[\sqrt{-15}]$ & $2$ & $2$ & $y-8$ \\
$15$ & $-15$ & $\Z[15/2 + \sqrt{-15}/2]$ & $2$ & $2$ & $y+4$ \\
$21$ & $-4\cdot 21$ & $\Z[\sqrt{-21}]$ & $4$ & $2$ & $y^2-2y-27$ \\
$22$ & $-4\cdot 22$ & $\Z[\sqrt{-22}]$ & $2$ & $2$ & $y-6$ \\
$26$ & $-4\cdot 26$ & $\Z[\sqrt{-26}]$ & $6$ & $2$ & $y^3-2y^2-15y-16$ \\
$30$ & $-4\cdot 30$ & $\Z[\sqrt{-30}]$ & $4$ & $4$ & $ y-5$ \\
$33$ & $-4\cdot 33$ & $\Z[\sqrt{-33}]$ & $4$ & $2$ & $y^2-2y-11$ \\
$34$ & $-4\cdot 34$ & $\Z[\sqrt{-34}]$ & $4$ & $2$ & $y^2-5y+2$ \\
$35$ & $-4\cdot 35$ & $\Z[\sqrt{-35}]$ & $6$ & $2$ & $y^3-2y^2-4y-20$ \\
$35$ & $-35$ & $\Z[35/2 + \sqrt{-35}/2]$ & $2$ & $2$ & $y+2$ \\
$38$ & $-4\cdot 38$ & $\Z[\sqrt{-38}]$ & $6$ & $2$ & $y^3-2y^2-7y-8$ \\
$39$ & $-4\cdot 39$ & $\Z[\sqrt{-39}]$ & $4$ & $2$ & $y^2-5y+3$ \\
$39$ & $-39$ & $\Z[39/2 + \sqrt{-39}/2]$ & $4$ & $2$ & $y^2+3y-1$ \\
$42$ & $-4\cdot 42$ & $\Z[\sqrt{-42}]$ & $4$ & $4$ & $y-4$ \\
$42$ & $-4\cdot 14$ & $\Z[\sqrt{-14}]$ & $4$ & $2$ & $y^2+3y+4$\\
$46$ & $-4\cdot 46$ & $\Z[\sqrt{-46}]$ & $4$ & $2$ & $y^2-2y-7$ \\
$51$ & $-4\cdot 51$ & $\Z[\sqrt{-51}]$ & $6$ & $2$ & $y^3-2y^2-4y-4$ \\
$51$ & $-51$ & $\Z[51/2 + \sqrt{-51}/2]$ & $2$ & $2$ & $y+2$ \\
$55$ & $-4\cdot 55$ & $\Z[\sqrt{-55}]$ & $4$ & $2$ & $y^2-5y+5$ \\
$55$ & $-55$ & $\Z[55/2 + \sqrt{-55}/2]$ & $4$ & $2$ & $y^2+3y+1$ \\
$62$ & $-4\cdot 62$ & $\Z[\sqrt{-62}]$ & $8$ & $2$ & $y^4-2y^3-3y^2-4y+4$ \\
$66$ & $-4\cdot 66$ & $\Z[\sqrt{-66}]$ & $8$ & $4$ & $y^2-y-8$ \\
$69$ & $-4\cdot 69$ & $\Z[\sqrt{-69}]$ & $8$ & $2$ & $y^4-2y^3-5y^2+6y-3$ \\
$70$ & $-4\cdot 70$ & $\Z[\sqrt{-70}]$ & $4$ & $4$ & $y-3$ \\
$70$ & $-4\cdot 14$ & $\Z[\sqrt{-14}]$ & $4$ & $2$ & $y^2-y+2$ \\
$70$ & $-4\cdot 35$ & $\Z[\sqrt{-35}]$ & $6$ & $2$ & $y^3+2y^2+4$ \\
$78$ & $-4\cdot 78$ & $\Z[\sqrt{-78}]$ & $4$ & $4$ & $y-3$ \\
$78$ & $-4\cdot 26$ & $\Z[\sqrt{-26}]$ & $6$ & $2$ & $y^3+y^2-4$ \\
$78$ & $-4\cdot 39$ & $\Z[\sqrt{-39}]$ & $4$ & $2$ & $y^2+y-3$ \\
$78$ & $-39$ & $\Z[39/2 + \sqrt{-39}/2]$ & $4$ & $2$ & $y^2+y+1$ \\
$87$ & $-4\cdot 87$ & $\Z[\sqrt{-87}]$ & $6$ & $2$ & $y^3-2y^2-y-1$ \\
$87$ & $-87$ & $\Z[87/2 + \sqrt{-87}/2]$ & $6$ & $2$ & $y^3+2y^2+3y+3$ \\
$94$ & $-4\cdot 94$ & $\Z[\sqrt{-94}]$ & $8$ & $2$ & $y^4-2y^3-3y^2+4y-4$ \\
$95$ & $-4\cdot 95$ & $\Z[\sqrt{-95}]$ & $8$ & $2$ & $y^4+y^3-6y^2-10y-5$ \\
$95$ & $-95$ & $\Z[95/2 + \sqrt{-95}/2]$ & $8$ & $2$ & $y^4+y^3-2y^2+2y-1$ \\
$105$ & $-4\cdot 105$ & $\Z[\sqrt{-105}]$ & $8$ & $4$ & $y^2-y-5$ \\
$105$ & $-4\cdot 35$ & $\Z[\sqrt{-35}]$ & $6$ & $2$ & $y^3+y^2-y-5$ \\
$105$ & $-35$ & $\Z[35/2 + \sqrt{-35}/2]$ & $2$ & $2$ & $y+1$ \\
$105$ & $-4\cdot 21$ & $\Z[\sqrt{-21}]$ & $4$ & $2$ & $y^2+3y+3$ \\
$110$ & $-4\cdot 110$ & $\Z[\sqrt{-110}]$ & $12$ & $4$ & $y^3-y^2-8$ \\
$110$ & $-4\cdot 55$ & $\Z[\sqrt{-55}]$ & $4$ & $2$ & $y^2+y-1$ \\
$110$ & $-55$ & $\Z[55/2 + \sqrt{-55}/2]$ & $4$ & $2$ & $y^2+y+3$ \\
$119$ & $-4\cdot 119$ & $\Z[\sqrt{-119}]$ & $10$ & $2$ & $y^5-2y^4+3y^3-6y^2-7$ \\
$119$ & $-119$ & $\Z[119/2 + \sqrt{-119}/2]$ & $10$ & $2$ & $y^5+2y^4+3y^3+6y^2+4y+1$ \\
\hline
\end{longtable}

\medskip

In certain cases, when the class number of the class field is not a power of $2$, we are able to
deduce class fields of the corresponding orders, by looking at the genus fields of the imaginary
quadratic fields and checking that the generating polynomials of the subfields of the ring class
field are irreducible over genus fields (for a definition, see \cite{Cox89}, p. 121). In cases
when the class number is a power of $2$, the generating polynomials appearing in
Table~\ref{Table with disc-subf} above are all reducible over genus fields; actually
they are generating polynomials for genus fields.

For example, when $N=26$, according to \cite{Cox89}, Theorem 6.1. the genus field of
$K=\Q[\sqrt{-26}]$ is $K[\sqrt{13}]$. A simple computation using Mathematica shows that
both $y^3-2y^2-15y-16$ and $y^3+y^2-4$ are irreducible over $K[\sqrt{13}]$. Denoting by
$\alpha$ the real zero of e.g. $y^3+y^2-4$, we see that $K[\sqrt{13}, \alpha]$ is a degree 6
extension of $K$ and the subfield of the ring class field $L$ of the order $\Z[\sqrt{-26}]$
over $K$, hence $L=K[\sqrt{13}, \alpha]$. Arguing in the same way, we deduce the following:
The ring class field of $\Z[\sqrt{-35}]$ over $K=\Q[\sqrt{-35}]$ is $K[\sqrt{5}, \alpha]$
where $\alpha$ is a real zero of $y^3-2y^2-4y-20$ or $y^3+2y^2+4$ or $y^3+y^2-y-5$; The ring
class field of $\Z[\sqrt{-38}]$ over $K=\Q[\sqrt{-38}]$ is $K[\sqrt{2}, \alpha]$ where $\alpha$
is a real zero of $y^3-2y^2-7y-8$; The ring class field of $\Z[\sqrt{-51}]$ over
$K=\Q[\sqrt{-51}]$ is $K[\sqrt{-3}, \alpha]$ where $\alpha$ is a real zero of $y^3-2y^2-4y-4$;
The ring class field of $\Z[\sqrt{-87}]$ over $K=\Q[\sqrt{-87}]$ is $K[\sqrt{-3}, \alpha]$
where $\alpha$ is a real zero of $y^3-2y^2-y-1$; The ring class field of
$\Z[87/2 + \sqrt{-87}/2]$ over $K=\Q[\sqrt{-87}]$ is $K[\sqrt{-3}, \alpha]$ where $\alpha$
is a real zero of $y^3+2y^2+3y+3$; The ring class field of $\Z[\sqrt{-110}]$ over
$K=\Q[\sqrt{-110}]$ is $K[\sqrt{2},\sqrt{5} , \alpha]$ where $\alpha$ is a real zero
of $y^3-y^2-8$; The ring class field of $\Z[\sqrt{-119}]$ over $K=\Q[\sqrt{-119}]$ is
$K[\sqrt{-7}, \alpha]$ where $\alpha$ is a real zero of $y^5-2y^4+3y^3-6y^2-7$; The
ring class field of $\Z[119/2 + \sqrt{-119}/2]$ over $K=\Q[\sqrt{-119}]$ is
$K[\sqrt{-7}, \alpha]$ where $\alpha$ is a real zero of $y^5+2y^4+3y^3+6y^2+4y+1$.

\appendix

\section{Polynomials $P_N(y)$ and $Q_N(y)$}\label{sec A1}

Here we list the polynomials $P_N$ and $Q_N$ that show up in the main theorem.

\medskip

\noindent%
$P_{1}(y) = y^2 - 480y + 1743552$ \\
$Q_{1}(y) = 2y^4 - 960y^3 - 2813184y^2 + 702812160y + 1071929106432$ \\
$P_{2}(y) = y^2 - 96y + 40640$ \\
$Q_{2}(y) = 2y^4 - 192y^3 - 58624y^2 + 3035136y + 499785728$ \\
$P_{3}(y) = y^2 - 48y + 7560$ \\
$Q_{3}(y) = 2y^4 - 96y^3 - 9936y^2 + 266112y + 15367968$ \\
$P_{5}(y) = y^4 + 8y^3 + 848y^2 + 38096y + 359872$ \\
$Q_{5}(y) = 2y^6 + 16y^5 - 2592y^4 - 40192y^3 + 741888y^2 + 19480576y + 110231552$ \\
$P_{6}(y) = y^4 + 4y^3 + 492y^2 + 16352y + 120608$ \\
$Q_{6}(y) = 2y^6 + 8y^5 - 1544y^4 - 15424y^3 + 276448y^2 + 4780160y + 18972800$ \\
$P_{7}(y) = y^4 + 2y^3 + 313y^2 + 8904y + 50328$ \\
$Q_{7}(y) = 2y^6 + 4y^5 - 1006y^4 - 7488y^3 + 120528y^2 + 1632960y + 5248800$ \\
$P_{10}(y) = y^4 + 128y^2 + 2256y + 10560$ \\
$Q_{10}(y) = 2y^6 - 448y^4 - 1536y^3 + 25088y^2 + 172032y + 294912$ \\
$P_{11}(y) = y^6 + 8y^5 + 112y^4 + 2620y^3 + 23424y^2 + 86896y + 115884$ \\
$Q_{11}(y) = 2y^8 + 16y^7 - 320y^6 - 4080y^5 - 288y^4 + 214784y^3 + 1340192y^2 + 3398784y + 3235968$ \\
$P_{13}(y) = y^6 + 6y^5 + 73y^4 + 1512y^3 + 11440y^2 + 36048y + 41472$ \\
$Q_{13}(y) = 2y^8 + 12y^7 - 238y^6 - 2304y^5 + 1280y^4 + 91392y^3 + 442368y^2 + 884736y + 663552$ \\
$P_{14}(y) = y^6 + 4y^5 + 62y^4 + 1036y^3 + 7009y^2 + 18888y + 17688$ \\
$Q_{14}(y) = 2y^8 + 8y^7 - 228y^6 - 1496y^5 + 3810y^4 + 58208y^3 + 196208y^2 + 282624y + 152352$ \\
$P_{15}(y) = y^6 + 4y^5 + 46y^4 + 1040y^3 + 6841y^2 + 15628y + 10540$ \\
$Q_{15}(y) = 2y^8 + 8y^7 - 164y^6 - 1320y^5 + 82y^4 + 38640y^3 + 190624y^2 + 406016y + 346112$ \\
$P_{17}(y) = y^8 + 8y^7 + 58y^6 + 864y^5 + 7265y^4 + 30984y^3 + 70840y^2 + 84016y + 41200$ \\
$Q_{17}(y) = 2y^{10} + 16y^9 - 108y^8 - 1576y^7 - 3870y^6 + 24840y^5 + 201208y^4 + 632384y^3 + 1066976y^2 + 956544y + 359552$ \\
$P_{19}(y) = y^8 + 6y^7 + 41y^6 + 580y^5 + 4120y^4 + 14908y^3 + 29212y^2 + 29064y + 11340$ \\
$Q_{19}(y) = 2y^{10} + 12y^9 - 110y^8 - 1072y^7 - 1120y^6 + 18128y^5 + 92768y^4 + 208320y^3 + 251424y^2 + 158976y + 41472$ \\
$P_{21}(y) = y^6 + 4y^5 + 30y^4 + 352y^3 + 1705y^2 + 5916y + 11340$ \\
$Q_{21}(y) = 2y^8 + 8y^7 - 132y^6 - 712y^5 + 1586y^4 + 15120y^3 + 23328y^2$ \\
$P_{22}(y) = y^8 + 4y^7 + 28y^6 + 332y^5 + 2144y^4 + 6480y^3 + 9548y^2 + 5712y + 624$ \\
$Q_{22}(y) = 2y^{10} + 8y^9 - 104y^8 - 624y^7 + 320y^6 + 10112y^5 + 32160y^4 + 50176y^3 + 44288y^2 + 21504y + 4608$ \\
$P_{23}(y) = y^{10} + 8y^9 + 44y^8 + 462y^7 + 3686y^6 + 16546y^5 + 44173y^4 + 72694y^3 + 72287y^2 + 39090y + 8277$ \\
$Q_{23}(y) = 2y^{12} + 16y^{11} - 40y^{10} - 856y^9 - 3028y^8 + 3032y^7 + 58596y^6 + 224936y^5 + 482794y^4 + 652576y^3 + 556172y^2 + 275800y + 61250$ \\
$P_{26}(y) = y^8 + 4y^7 + 22y^6 + 228y^5 + 1345y^4 + 3888y^3 + 6624y^2 + 7376y + 3840$ \\
$Q_{26}(y) = 2y^{10} + 8y^9 - 84y^8 - 488y^7 + 194y^6 + 6480y^5 + 17440y^4 + 19456y^3 + 8192y^2$ \\
$P_{29}(y) = y^{12} + 8y^{11} + 38y^{10} + 296y^9 + 2091y^8 + 9000y^7 + 24526y^6 + 45520y^5 + 59625y^4 + 53152y^3 + 28472y^2 + 7248y + 624$ \\
$Q_{29}(y) = 2y^{14} + 16y^{13} - 20y^{12} - 608y^{11} - 2122y^{10} + 968y^9 + 27740y^8 + 95176y^7 + 175058y^6 + 197896y^5 + 140088y^4 + 60736y^3 + 15584y^2 + 2176y + 128$ \\
$P_{30}(y) = y^8 + 2y^7 + 15y^6 + 112y^5 + 311y^4 + 1386y^3 + 2613y^2 - 2340y + 2700$ \\
$Q_{30}(y) = 2y^{10} + 4y^9 - 98y^8 - 248y^7 + 1342y^6 + 3940y^5 - 3262y^4 - 8880y^3 + 7200y^2$ \\
$P_{31}(y) = y^{12} + 8y^{11} + 36y^{10} + 246y^9 + 1622y^8 + 6930y^7 + 18981y^6 + 33086y^5 + 35047y^4 + 22122y^3 + 10189y^2 + 3864y + 648$ \\
$Q_{31}(y) = 2y^{14} + 16y^{13} - 24y^{12} - 600y^{11} - 1748y^{10} + 2328y^9 + 24900y^8 + 64840y^7 + 88922y^6 + 72480y^5 + 36380y^4 + 10584y^3 + 1458y^2$ \\
$P_{33}(y) = y^{10} + 4y^9 + 14y^8 + 188y^7 + 1081y^6 + 2756y^5 + 4976y^4 + 8092y^3 + 9200y^2 + 5104y + 1452$ \\
$Q_{33}(y) = 2y^{12} + 8y^{11} - 36y^{10} - 312y^9 - 590y^8 + 1520y^7 + 10144y^6 + 23440y^5 + 28288y^4 + 16896y^3 + 3872y^2$ \\
$P_{34}(y) = y^{10} + 2y^9 + 11y^8 + 100y^7 + 411y^6 + 898y^5 + 1417y^4 + 752y^3 - 664y^2 - 912y + 432$ \\
$Q_{34}(y) = 2y^{12} + 4y^{11} - 74y^{10} - 208y^9 + 646y^8 + 2644y^7 + 1162y^6 - 3400y^5 - 1640y^4 + 2176y^3 + 352y^2 - 640y + 128$ \\
$P_{35}(y) = y^{10} + 4y^9 + 14y^8 + 180y^7 + 1273y^6 + 4780y^5 + 11632y^4 + 15964y^3 + 12112y^2 - 3360y - 6900$ \\
$Q_{35}(y) = 2y^{12} + 8y^{11} - 4y^{10} - 200y^9 - 878y^8 - 1440y^7 + 2944y^6 + 24976y^5 + 78112y^4 + 154880y^3 + 205600y^2 + 176000y + 80000$ \\
$P_{38}(y) = y^{12} + 4y^{11} + 14y^{10} + 112y^9 + 769y^8 + 2440y^7 + 4532y^6 + 5260y^5 + 3976y^4 + 3328y^3 + 3308y^2 + 2464y + 768$ \\
$Q_{38}(y) = 2y^{14} + 8y^{13} - 36y^{12} - 248y^{11} - 350y^{10} + 976y^9 + 5584y^8 + 13264y^7 + 20256y^6 + 21440y^5 + 15904y^4 + 7680y^3 + 2048y^2$ \\
$P_{39}(y) = y^{10} + 8y^8 + 42y^7 + 792y^5 + 161y^4 - 4104y^3 + 5608y^2 - 3048y + 648$ \\
$Q_{39}(y) = 2y^{12} - 80y^{10} - 24y^9 + 1056y^8 + 288y^7 - 5012y^6 + 2304y^5 + 8624y^4 - 12504y^3 + 6912y^2 - 1728y + 162$ \\
$P_{41}(y) = y^{16} + 8y^{15} + 32y^{14} + 168y^{13} + 992y^{12} + 3880y^{11} + 9914y^{10} + 18616y^9 + 27584y^8 + 28440y^7 + 11872y^6 - 8504y^5 - 8839y^4 + 848y^3 + 1184y^2 - 624y + 192$ \\
$Q_{41}(y) = 2y^{18} + 16y^{17} - 392y^{15} - 1408y^{14} - 32y^{13} + 11868y^{12} + 34464y^{11} + 42400y^{10} + 11480y^9 - 28160y^8 - 24928y^7 + 5234y^6 + 11216y^5 - 800y^4 - 2304y^3 + 512y^2$ \\
$P_{42}(y) = y^{10} + 2y^9 + 7y^8 + 88y^7 + 343y^6 + 122y^5 + 781y^4 + 3108y^3 + 1020y^2 - 288y + 1728$ \\
$Q_{42}(y) = 2y^{12} + 4y^{11} - 50y^{10} - 168y^9 + 190y^8 + 1668y^7 + 2290y^6 - 1568y^5 - 5440y^4 - 1536y^3 + 4608y^2$ \\
$P_{46}(y) = y^{14} + 4y^{13} + 10y^{12} + 62y^{11} + 359y^{10} + 1174y^9 + 1837y^8 + 484y^7 - 609y^6 + 2520y^5 + 1612y^4 - 3378y^3 - 863y^2 + 414y + 45$ \\
$Q_{46}(y) = 2y^{16} + 8y^{15} - 44y^{14} - 240y^{13} + 30y^{12} + 1736y^{11} + 2852y^{10} + 736y^9 - 1026y^8 - 768y^7 - 996y^6 + 440y^5 + 740y^4 - 376y^3 + 344y^2 - 336y + 98$ \\
$P_{47}(y) = y^{18} + 8y^{17} + 36y^{16} + 188y^{15} + 1158y^{14} + 5570y^{13} + 19584y^{12} + 51414y^{11} + 103965y^{10} + 165234y^9 + 209397y^8 + 213722y^7 + 177346y^6 + 120520y^5 + 66397y^4 + 27640y^3 + 6488y^2 - 648y - 708$ \\
$Q_{47}(y) = 2y^{20} + 16y^{19} + 40y^{18} - 96y^{17} - 1044y^{16} - 3592y^{15} - 5432y^{14} + 6504y^{13} + 63058y^{12} + 205064y^{11} + 450492y^{10} + 753224y^9 + 999080y^8 + 1066960y^7 + 918588y^6 + 631464y^5 + 339872y^4 + 138360y^3 + 40368y^2 + 7600y + 722$ \\
$P_{51}(y) = y^{14} + 4y^{13} + 10y^{12} + 92y^{11} + 499y^{10} + 1308y^9 + 2598y^8 + 4388y^7 + 5905y^6 + 5212y^5 + 2024y^4 - 2628y^3 - 4288y^2 - 1536y - 116$ \\
$Q_{51}(y) = 2y^{16} + 8y^{15} - 12y^{14} - 160y^{13} - 410y^{12} + 16y^{11} + 2580y^{10} + 7272y^9 + 10530y^8 + 8736y^7 + 4032y^6 + 1232y^5 + 1312y^4 + 1024y^3 + 288y^2 - 128y + 128$ \\
$P_{55}(y) = y^{14} + 4y^{13} + 6y^{12} + 46y^{11} + 105y^{10} - 160y^9 + 201y^8 + 3114y^7 + 3154y^6 - 6500y^5 - 15911y^4 - 8320y^3 + 12400y^2 + 17960y + 7640$ \\
$Q_{55}(y) = 2y^{16} + 8y^{15} - 52y^{14} - 272y^{13} + 290y^{12} + 3072y^{11} + 1708y^{10} - 13528y^9 - 17160y^8 + 25528y^7 + 49564y^6 - 14048y^5 - 60382y^4 - 14560y^3 + 25180y^2 + 15400y + 2450$ \\
$P_{59}(y) = y^{22} + 8y^{21} + 32y^{20} + 126y^{19} + 664y^{18} + 3064y^{17} + 10303y^{16} + 25672y^{15} + 48984y^{14} + 73000y^{13} + 85432y^{12} + 75264y^{11} + 39939y^{10} - 9016y^9 - 53200y^8 - 78750y^7 - 80728y^6 - 65192y^5 - 42791y^4 - 22808y^3 - 9496y^2 - 2936y - 536$ \\
$Q_{59}(y) = 2y^{24} + 16y^{23} + 32y^{22} - 112y^{21} - 784y^{20} - 1952y^{19} - 1672y^{18} + 5328y^{17} + 27040y^{16} + 70400y^{15} + 135632y^{14} + 212576y^{13} + 282652y^{12} + 326064y^{11} + 330592y^{10} + 296848y^9 + 236944y^8 + 168224y^7 + 105912y^6 + 58864y^5 + 28512y^4 + 11808y^3 + 4144y^2 + 1056y + 242$ \\
$P_{62}(y) = y^{18} + 4y^{17} + 10y^{16} + 54y^{15} + 343y^{14} + 998y^{13} + 1821y^{12} + 2084y^{11} + 2271y^{10} + 4360y^9 + 8436y^8 + 9862y^7 + 4081y^6 - 4002y^5 - 7211y^4 - 3184y^3 + 304y^2 + 1536y + 336$ \\
$Q_{62}(y) = 2y^{20} + 8y^{19} - 12y^{18} - 128y^{17} - 290y^{16} - 8y^{15} + 1540y^{14} + 4320y^{13} + 6382y^{12} + 4640y^{11} - 1124y^{10} - 7032y^9 - 7388y^8 - 2376y^7 + 3368y^6 + 4096y^5 + 1490y^4 - 1200y^3 - 976y^2 - 192y + 288$ \\
$P_{66}(y) = y^{14} + 2y^{13} - y^{12} + 4y^{11} - 57y^{10} + 334y^9 + 1577y^8 - 3072y^7 - 7216y^6 + 12476y^5 + 3836y^4 - 17208y^3 + 20748y^2 - 17568y + 6912$ \\
$Q_{66}(y) = 2y^{16} + 4y^{15} - 66y^{14} - 120y^{13} + 846y^{12} + 1204y^{11} - 5694y^{10} - 4368y^9 + 23248y^8 + 1168y^7 - 52352y^6 + 30720y^5 + 43808y^4 - 56832y^3 + 18432y^2$ \\
$P_{69}(y) = y^{18} + 4y^{17} + 6y^{16} + 42y^{15} + 217y^{14} + 398y^{13} + 719y^{12} + 2344y^{11} + 1952y^{10} - 5504y^9 - 7386y^8 + 4220y^7 + 5291y^6 - 6276y^5 - 4794y^4 + 3582y^3 + 1233y^2 - 2394y + 117$ \\
$Q_{69}(y) = 2y^{20} + 8y^{19} - 20y^{18} - 152y^{17} - 158y^{16} + 632y^{15} + 1816y^{14} + 824y^{13} - 2248y^{12} - 1072y^{11} + 4796y^{10} + 3288y^9 - 5300y^8 - 2208y^7 + 6784y^6 + 344y^5 - 4792y^4 + 2256y^3 + 1392y^2 - 1440y + 450$ \\
$P_{70}(y) = y^{14} + 2y^{13} + 3y^{12} + 8y^{11} + 179y^{10} + 806y^9 + 965y^8 + 676y^7 + 2004y^6 - 420y^5 - 4468y^4 - 2392y^3 - 3652y^2 - 8976y - 528$ \\
$Q_{70}(y) = 2y^{16} + 4y^{15} - 26y^{14} - 48y^{13} + 38y^{12} - 92y^{11} + 58y^{10} + 1336y^9 + 1400y^8 + 3024y^7 + 7392y^6 + 6080y^5 + 8864y^4 + 11776y^3 + 5888y^2 + 6144y + 4608$ \\
$P_{71}(y) = y^{26} + 8y^{25} + 28y^{24} + 90y^{23} + 410y^{22} + 1458y^{21} + 3073y^{20} + 3798y^{19} + 4021y^{18} + 8108y^{17} + 13349y^{16} + 88y^{15} - 40168y^{14} - 63932y^{13} - 13304y^{12} + 79626y^{11} + 94791y^{10} - 5466y^9 - 91159y^8 - 52578y^7 + 31090y^6 + 41224y^5 + 1849y^4 - 12328y^3 - 3112y^2 + 1320y + 364$ \\
$Q_{71}(y) = 2y^{28} + 16y^{27} + 24y^{26} - 184y^{25} - 908y^{24} - 1032y^{23} + 3532y^{22} + 14280y^{21} + 17294y^{20} - 10520y^{19} - 57572y^{18} - 57304y^{17} + 28686y^{16} + 114240y^{15} + 73640y^{14} - 61824y^{13} - 119262y^{12} - 32328y^{11} + 67980y^{10} + 60880y^9 - 7196y^8 - 34416y^7 - 11396y^6 + 9048y^5 + 6760y^4 - 680y^3 - 1728y^2 - 176y + 242$ \\
$P_{78}(y) = y^{16} + 2y^{15} + y^{14} + 34y^{13} + 132y^{12} - 154y^{11} - 591y^{10} + 398y^9 + 2909y^8 + 3656y^7 - 2400y^6 - 12752y^5 - 10660y^4 + 5648y^3 + 14424y^2 + 5904y + 648$ \\
$Q_{78}(y) = 2y^{18} + 4y^{17} - 30y^{16} - 104y^{15} + 48y^{14} + 760y^{13} + 1084y^{12} - 1240y^{11} - 5108y^{10} - 3736y^9 + 6144y^8 + 13304y^7 + 5682y^6 - 10188y^5 - 15454y^4 - 5136y^3 + 5616y^2 + 6912y + 2592$ \\
$P_{87}(y) = y^{22} + 4y^{21} + 14y^{20} + 66y^{19} + 305y^{18} + 1226y^{17} + 4151y^{16} + 10852y^{15} + 24748y^{14} + 44920y^{13} + 72122y^{12} + 91948y^{11} + 101547y^{10} + 80524y^9 + 44982y^8 - 6562y^7 - 39367y^6 - 55542y^5 - 45531y^4 - 29628y^3 - 14220y^2 - 5040y - 2124$ \\
$Q_{87}(y) = 2y^{24} + 8y^{23} + 28y^{22} + 24y^{21} - 78y^{20} - 584y^{19} - 1656y^{18} - 3144y^{17} - 3080y^{16} + 3152y^{15} + 22476y^{14} + 63384y^{13} + 127484y^{12} + 210752y^{11} + 293792y^{10} + 355048y^9 + 374808y^8 + 343808y^7 + 277664y^6 + 190352y^5 + 113762y^4 + 54288y^3 + 21696y^2 + 5760y + 1152$ \\
$P_{94}(y) = y^{26} + 4y^{25} + 2y^{24} - 4y^{23} + 71y^{22} + 486y^{21} + 844y^{20} - 1026y^{19} - 2089y^{18} + 5456y^{17} + 2155y^{16} - 18040y^{15} + 9269y^{14} + 28398y^{13} - 45506y^{12} - 7458y^{11} + 68336y^{10} - 55696y^9 - 26471y^8 + 69872y^7 - 43972y^6 - 9720y^5 + 29632y^4 - 22032y^3 + 7216y^2 - 992y - 232$ \\
$Q_{94}(y) = 2y^{28} + 8y^{27} - 28y^{26} - 136y^{25} + 110y^{24} + 728y^{23} - 408y^{22} - 1528y^{21} + 3246y^{20} + 2136y^{19} - 8224y^{18} + 5544y^{17} + 11698y^{16} - 20952y^{15} + 5640y^{14} + 26896y^{13} - 36336y^{12} + 7016y^{11} + 35708y^{10} - 45272y^9 + 14290y^8 + 25200y^7 - 36628y^6 + 19208y^5 + 3986y^4 - 13584y^3 + 10448y^2 - 4160y + 800$ \\
$P_{95}(y) = y^{22} + 4y^{21} + 2y^{20} - 4y^{19} + 67y^{18} + 478y^{17} + 848y^{16} - 1034y^{15} - 3347y^{14} + 1542y^{13} + 7379y^{12} + 348y^{11} - 6511y^{10} + 7508y^9 + 6564y^8 - 19472y^7 + 9653y^6 + 10646y^5 - 18870y^4 + 6050y^3 + 3245y^2 - 930y + 165$ \\
$Q_{95}(y) = 2y^{24} + 8y^{23} - 28y^{22} - 136y^{21} + 102y^{20} + 712y^{19} - 272y^{18} - 1288y^{17} + 2522y^{16} + 1584y^{15} - 6368y^{14} + 1608y^{13} + 6970y^{12} - 8656y^{11} - 2656y^{10} + 9992y^9 - 3074y^8 - 5528y^7 + 6644y^6 - 176y^5 - 3372y^4 + 1480y^3 + 680y^2 - 1200y + 450$ \\
$P_{105}(y) = y^{18} + 2y^{17} - 3y^{16} + 24y^{15} + 24y^{14} - 496y^{13} - 566y^{12} + 3432y^{11} + 5926y^{10} - 10512y^9 - 26620y^8 + 9344y^7 + 57793y^6 + 17382y^5 - 57447y^4 - 44064y^3 + 11880y^2 + 30600y + 18000$ \\
$Q_{105}(y) = 2y^{20} + 4y^{19} - 38y^{18} - 124y^{17} + 216y^{16} + 1340y^{15} + 358y^{14} - 6716y^{13} - 9090y^{12} + 15256y^{11} + 42032y^{10} - 3352y^9 - 92498y^8 - 57748y^7 + 100086y^6 + 122404y^5 - 37192y^4 - 103140y^3 - 14550y^2 + 31500y + 11250$ \\
$P_{110}(y) = y^{20} + 2y^{19} + 7y^{18} + 38y^{17} + 179y^{16} + 440y^{15} + 1002y^{14} + 44y^{13} + 1091y^{12} - 3578y^{11} + 10043y^{10} + 1742y^9 + 29593y^8 - 16792y^7 + 2440y^6 - 35328y^5 + 15300y^4 + 18128y^3 + 4568y^2 - 14736y + 5016$ \\
$Q_{110}(y) = 2y^{22} + 4y^{21} + 14y^{20} - 32y^{19} - 114y^{18} - 388y^{17} - 274y^{16} + 608y^{15} + 2916y^{14} + 5792y^{13} + 3156y^{12} - 2112y^{11} - 16674y^{10} - 12788y^9 - 2594y^8 + 16128y^7 + 28046y^6 - 16268y^5 + 8146y^4 - 37792y^3 + 33824y^2 - 10752y + 1152$ \\
$P_{119}(y) = y^{26} + 4y^{25} + 14y^{24} + 40y^{23} + 193y^{22} + 814y^{21} + 3046y^{20} + 9970y^{19} + 27348y^{18} + 65380y^{17} + 137087y^{16} + 251940y^{15} + 411609y^{14} + 591020y^{13} + 748046y^{12} + 828256y^{11} + 778011y^{10} + 618670y^9 + 394924y^8 + 179710y^7 + 62284y^6 + 36940y^5 + 41961y^4 + 43232y^3 + 34104y^2 + 12824y + 868$ \\
$Q_{119}(y) = 2y^{28} + 8y^{27} + 28y^{26} + 80y^{25} + 130y^{24} + 104y^{23} - 356y^{22} - 2128y^{21} - 6184y^{20} - 13608y^{19} - 22984y^{18} - 25760y^{17} - 3422y^{16} + 78816y^{15} + 267868y^{14} + 605744y^{13} + 1110074y^{12} + 1738848y^{11} + 2379824y^{10} + 2872592y^9 + 3054738y^8 + 2848864y^7 + 2303696y^6 + 1580888y^5 + 894230y^4 + 395752y^3 + 126308y^2 + 25480y + 2450$ \\

\medskip

\section{Roots of $h_N(y)$ in terms of radicals}\label{sec A2}

Below we list the roots, in radicals, for each irreducible factor of $h_{N}$
of degree three or higher.  As before, we consider square-free $N$ such that
$\Gamma_{0}(N)^{+}$ has genus zero and denote by $\zeta_n=\exp(2\pi i /n)$ the primitive
$n$th root of unity.

\medskip

\noindent $N=11$: The roots of the polynomial $y^{3} - 2y^{2} - 76y - 212$ are
$$y = \frac{2}{3} + \left(\frac{1}{3}\zeta_{3} - \frac{1}{3}\zeta_{3}^{2}\right)\omega_2 - \frac{1777}{26912}\omega_2^2 + \left( - \frac{33}{26912}\zeta_{3} + \frac{33}{26912}\zeta_{3}^{2}\right)\omega_1\omega_2^2,$$
 where
$\omega_1 = \sqrt[2]{ - 11}$,
$\omega_2 = \sqrt[3]{\left(\frac{3554}{9}\zeta_{3} - \frac{3554}{9}\zeta_{3}^{2}\right) + 22\omega_1}$.\\

\noindent $N=17$: The roots of the polynomial $y^{4} + 2y^{3} - 39y^{2} - 176y - 212$ are
$$y= - \frac{1}{2} - \frac{1}{2}\omega_1 + 2\omega_2,$$ where
$\omega_1 = \sqrt[2]{17}$,
$\omega_2 = \sqrt[2]{4-\omega_1}$.\\

\noindent $N=19$: The roots of the polynomial $y^{3} - 4y^{2} - 16y - 12$ are
$$y = \frac{4}{3} + \left(\frac{1}{3}\zeta_{3} - \frac{1}{3}\zeta_{3}^{2}\right)\omega_2 - \frac{257}{2048}\omega_2^2 + \left( - \frac{3}{2048}\zeta_{3} + \frac{3}{2048}\zeta_{3}^{2}\right)\omega_1\omega_2^2,$$ where
$\omega_1 = \sqrt[2]{ - 19}$,
$\omega_2 = \sqrt[3]{\left(\frac{514}{9}\zeta_{3} - \frac{514}{9}\zeta_{3}^{2}\right) + 2\omega_1}$.\\

\noindent $N=22$: The roots of the polynomial $y^{3} + 6y^{2} + 8y + 4$ are
$$y = - 2 + \left(\frac{1}{3}\zeta_{3} - \frac{1}{3}\zeta_{3}^{2}\right)\omega_2 + \frac{3}{8}\omega_2^2 + \left( - \frac{1}{24}\zeta_{3} + \frac{1}{24}\zeta_{3}^{2}\right)\omega_1\omega_2^2,$$ where
$\omega_1 = \sqrt[2]{ - 11}$,
$\omega_2 = \sqrt[3]{\left( - 6\zeta_{3} + 6\zeta_{3}^{2}\right) + 2\omega_1}$.\\

\noindent $N=23$: The roots of the polynomial $y^{3} + 6y^{2} + 11y + 7$ are
$$y = - 2 + \left(\frac{1}{3}\zeta_{3} - \frac{1}{3}\zeta_{3}^{2}\right)\omega_2 + \frac{3}{2}\omega_2^2 + \left( - \frac{1}{6}\zeta_{3} + \frac{1}{6}\zeta_{3}^{2}\right)\omega_1\omega_2^2,$$
where $\omega_1 = \sqrt[2]{ - 23}$,
$\omega_2 = \sqrt[3]{\left( - \frac{3}{2}\zeta_{3} + \frac{3}{2}\zeta_{3}^{2}\right) + \frac{1}{2}\omega_1}$.\\

\noindent $N=23$: The roots of the polynomial $y^{3} - 2y^{2} - 17y - 25$ are
$$y = \frac{2}{3} + \left(\frac{1}{3}\zeta_{3} - \frac{1}{3}\zeta_{3}^{2}\right)\omega_2 - \frac{997}{6050}\omega_2^2 + \left( - \frac{69}{6050}\zeta_{3} + \frac{69}{6050}\zeta_{3}^{2}\right)\omega_1\omega_2^2,$$ where
$\omega_1 = \sqrt[2]{ - 23}$,
$\omega_2 = \sqrt[3]{\left(\frac{997}{18}\zeta_{3} - \frac{997}{18}\zeta_{3}^{2}\right) + \frac{23}{2}\omega_1}$.\\

\noindent $N=26$: The roots of the polynomial $y^{3} - 2y^{2} - 15y - 16$ are
$$y = \frac{2}{3} + \left(\frac{1}{3}\zeta_{3} - \frac{1}{3}\zeta_{3}^{2}\right)\omega_2 - \frac{359}{2401}\omega_2^2 + \left( - \frac{12}{2401}\zeta_{3} + \frac{12}{2401}\zeta_{3}^{2}\right)\omega_1\omega_2^2,$$ where
$\omega_1 = \sqrt[2]{ - 26}$,
$\omega_2 = \sqrt[3]{\left(\frac{359}{9}\zeta_{3} - \frac{359}{9}\zeta_{3}^{2}\right) + 4\omega_1}$.\\

\noindent $N=29$: The roots of the polynomial $y^{6} + 2y^{5} - 17y^{4} - 66y^{3} - 83y^{2} - 32y - 4$ are
\medskip \begin{center} $\displaystyle
y =  - \frac{1}{3} + \frac{1}{3}\omega_1 + \left( - \frac{4675663}{69945938}\zeta_{3} + \frac{9170235}{34972969}\zeta_{3}^{2}\right)\omega_3 + \left(\frac{230955}{69945938}\zeta_{3} + \frac{1143895}{34972969}\zeta_{3}^{2}\right)\omega_1\omega_3  + \left(\frac{14913792}{3602215807}\zeta_{3} + \frac{7468392}{3602215807}\zeta_{3}^{2}\right)\omega_2\omega_3 + \left( - \frac{949738}{3602215807}\zeta_{3} - \frac{2272628}{3602215807}\zeta_{3}^{2}\right)\omega_1\omega_2\omega_3 + \left(\frac{81809110004592}{1223108560674961}\zeta_{3} - \frac{84662907394671}{1223108560674961}\zeta_{3}^{2}\right)\omega_3^2 + \left(\frac{15630127276428}{1223108560674961}\zeta_{3} + \frac{20828508011880}{1223108560674961}\zeta_{3}^{2}\right)\omega_1\omega_3^2  + \left(\frac{291290361453255}{125980181749520983}\zeta_{3} + \frac{146530487024733}{251960363499041966}\zeta_{3}^{2}\right)\omega_2\omega_3^2 - \left(  \frac{109445587541901}{125980181749520983}\zeta_{3} + \frac{134076636577647}{251960363499041966}\zeta_{3}^{2}\right)\omega_1\omega_2\omega_3^2$,
\end{center}
\medskip
where $\omega_1 = \sqrt[2]{29}$,
$\omega_2 = \sqrt[2]{ - 470 + 42\omega_1}$,\\
$\omega_3 = \sqrt[3]{\left( - \frac{233}{9}\zeta_{3} - \frac{214}{9}\zeta_{3}^{2}\right) + \left( - \frac{164}{27}\zeta_{3} - \frac{65}{27}\zeta_{3}^{2}\right)\omega_1 + \left( - \frac{427}{618}\zeta_{3} - \frac{673}{618}\zeta_{3}^{2}\right)\omega_2 + \left( - \frac{61}{1854}\zeta_{3} - \frac{155}{1854}\zeta_{3}^{2}\right)\omega_1\omega_2}$.\\

\noindent $N=31$: The roots of the polynomial $y^{3} + 4y^{2} + 3y + 1$ are
$$y = - \frac{4}{3} + \left(\frac{1}{3}\zeta_{3} - \frac{1}{3}\zeta_{3}^{2}\right)\omega_2 + \frac{47}{98}\omega_2^2 + \left( - \frac{3}{98}\zeta_{3} + \frac{3}{98}\zeta_{3}^{2}\right)\omega_1\omega_2^2,$$ where
$\omega_1 = \sqrt[2]{ - 31}$,
$\omega_2 = \sqrt[3]{\left( - \frac{47}{18}\zeta_{3} + \frac{47}{18}\zeta_{3}^{2}\right) + \frac{1}{2}\omega_1}$.\\

\noindent $N=31$: The roots of the polynomial $y^{3} - 17y - 27$ are
$$y = \left(\frac{1}{3}\zeta_{3} - \frac{1}{3}\zeta_{3}^{2}\right)\omega_2 - \frac{81}{578}\omega_2^2 + \left( - \frac{1}{1734}\zeta_{3} + \frac{1}{1734}\zeta_{3}^{2}\right)\omega_1\omega_2^2,$$ where
$\omega_1 = \sqrt[2]{ - 31}$,
$\omega_2 = \sqrt[3]{\left(\frac{81}{2}\zeta_{3} - \frac{81}{2}\zeta_{3}^{2}\right) + \frac{1}{2}\omega_1}$.\\

\noindent $N=33$: The roots of the polynomial $y^{3} + 4y^{2} + 8y + 4$ are
$$y = - \frac{4}{3} + \left(\frac{1}{3}\zeta_{3} - \frac{1}{3}\zeta_{3}^{2}\right)\omega_2 - \frac{13}{32}\omega_2^2 + \left( - \frac{3}{32}\zeta_{3} + \frac{3}{32}\zeta_{3}^{2}\right)\omega_1\omega_2^2,$$ where
$\omega_1 = \sqrt[2]{ - 11}$,
$\omega_2 = \sqrt[3]{\left(\frac{26}{9}\zeta_{3} - \frac{26}{9}\zeta_{3}^{2}\right) + 2\omega_1}$.\\

\noindent $N=35$:
The roots of the polynomial $y^{3} - 2y^{2} - 4y - 20$ are
$$y = \frac{2}{3} + \left(\frac{1}{3}\zeta_{3} - \frac{1}{3}\zeta_{3}^{2}\right)\omega_2 - \frac{157}{128}\omega_2^2 + \left( - \frac{15}{128}\zeta_{3} + \frac{15}{128}\zeta_{3}^{2}\right)\omega_1\omega_2^2,$$ where
$\omega_1 = \sqrt[2]{ - 35}$,
$\omega_2 = \sqrt[3]{\left(\frac{314}{9}\zeta_{3} - \frac{314}{9}\zeta_{3}^{2}\right) + 10\omega_1}$.\\

\noindent $N=38$: The roots of the polynomial $y^{3} + 4y^{2} + 4y + 4$ are
$$y = - \frac{4}{3} + \left(\frac{1}{3}\zeta_{3} - \frac{1}{3}\zeta_{3}^{2}\right)\omega_2 + \frac{23}{8}\omega_2^2 + \left( - \frac{3}{8}\zeta_{3} + \frac{3}{8}\zeta_{3}^{2}\right)\omega_1\omega_2^2,$$ where
$\omega_1 = \sqrt[2]{ - 19}$,
$\omega_2 = \sqrt[3]{\left( - \frac{46}{9}\zeta_{3} + \frac{46}{9}\zeta_{3}^{2}\right) + 2\omega_1}$.\\

\noindent $N=38$: The roots of the polynomial $y^{3} - 2y^{2} - 7y - 8$ are
$$y = \frac{2}{3} + \left(\frac{1}{3}\zeta_{3} - \frac{1}{3}\zeta_{3}^{2}\right)\omega_2 - \frac{179}{625}\omega_2^2 + \left( - \frac{12}{625}\zeta_{3} + \frac{12}{625}\zeta_{3}^{2}\right)\omega_1\omega_2^2,$$ where
$\omega_1 = \sqrt[2]{ - 38}$,
$\omega_2 = \sqrt[3]{\left(\frac{179}{9}\zeta_{3} - \frac{179}{9}\zeta_{3}^{2}\right) + 4\omega_1}$.\\

\noindent $N=41$: The roots of the polynomial $y^{8} + 4y^{7} - 8y^{6} - 66y^{5} - 120y^{4} - 56y^{3} + 53y^{2} + 36y - 16$ are
\medskip \begin{center} $\displaystyle
y = - \frac{1}{2} - \frac{1}{4}\omega_2 - \frac{5171}{34292}\omega_4 + \frac{1015}{34292}\omega_1\omega_4 + \frac{28441}{342920}\omega_2\omega_4 + \frac{671}{342920}\omega_1\omega_2\omega_4 - \frac{43559}{685840}\omega_3\omega_4 + \frac{4591}{685840}\omega_1\omega_3\omega_4 + \frac{307}{274336}\omega_2\omega_3\omega_4 - \frac{1055}{274336}\omega_1\omega_2\omega_3\omega_4$,
\end{center}
\medskip
where
$\omega_1 = \sqrt[2]{41}$,
$\omega_2 = \sqrt[2]{10 - 2\omega_1}$,
$\omega_3 = \sqrt[2]{10 + 2\omega_1}$,
$\omega_4 = \sqrt[2]{\frac{173}{2} + \frac{9}{2}\omega_1 + \frac{1}{2}\omega_2 - \frac{3}{2}\omega_1\omega_2 + 5\omega_3 + 3\omega_1\omega_3-\omega_2\omega_3}$.\\

\noindent $N=46$: The roots of the polynomial $y^{3} + 2y^{2} + y + 1$ are
$$y =  - \frac{2}{3} + \left(\frac{1}{3}\zeta_{3} - \frac{1}{3}\zeta_{3}^{2}\right)\omega_2 + \frac{25}{2}\omega_2^2 + \left( - \frac{3}{2}\zeta_{3} + \frac{3}{2}\zeta_{3}^{2}\right)\omega_1\omega_2^2,$$ where
$\omega_1 = \sqrt[2]{ - 23}$,
$\omega_2 = \sqrt[3]{\left( - \frac{25}{18}\zeta_{3} + \frac{25}{18}\zeta_{3}^{2}\right) + \frac{1}{2}\omega_1}$.\\

\noindent $N=46$: The roots of the polynomial $y^{3} + 2y^{2} - 3y + 1$ are
$$y =  - \frac{2}{3} + \left(\frac{1}{3}\zeta_{3} - \frac{1}{3}\zeta_{3}^{2}\right)\omega_2 + \frac{97}{338}\omega_2^2 + \left( - \frac{3}{338}\zeta_{3} + \frac{3}{338}\zeta_{3}^{2}\right)\omega_1\omega_2^2,$$ where
$\omega_1 = \sqrt[2]{ - 23}$,
$\omega_2 = \sqrt[3]{\left( - \frac{97}{18}\zeta_{3} + \frac{97}{18}\zeta_{3}^{2}\right) + \frac{1}{2}\omega_1}$.\\

\noindent $N=47$: The roots of the polynomial $y^{5} + 4y^{4} + 7y^{3} + 8y^{2} + 4y + 1$ are
\medskip \begin{center} $\displaystyle
y = - \frac{4}{5} + \left( - \frac{8}{11}\zeta_{5} - \frac{4}{11}\zeta_{5}^{2} - \frac{6}{11}\zeta_{5}^{3} - \frac{5}{11}\zeta_{5}^{4}\right)\omega_2 + \left( - \frac{970}{121}\zeta_{5} - \frac{5315}{242}\zeta_{5}^{2} - \frac{1455}{242}\zeta_{5}^{3} + \frac{1125}{242}\zeta_{5}^{4}\right)\omega_2^2 + \left(\frac{115}{121}\zeta_{5} - \frac{215}{242}\zeta_{5}^{2} - \frac{735}{242}\zeta_{5}^{3} - \frac{255}{242}\zeta_{5}^{4}\right)\omega_1\omega_2^2 + \left(\frac{1753525}{1331}\zeta_{5} + \frac{3814375}{2662}\zeta_{5}^{2} + \frac{245050}{1331}\zeta_{5}^{3} - \frac{932000}{1331}\zeta_{5}^{4}\right)\omega_2^3  + \left(\frac{59725}{1331}\zeta_{5} + \frac{643275}{2662}\zeta_{5}^{2} + \frac{423675}{1331}\zeta_{5}^{3} + \frac{224775}{1331}\zeta_{5}^{4}\right)\omega_1\omega_2^3  + \left(\frac{358621125}{14641}\zeta_{5} + \frac{224238000}{14641}\zeta_{5}^{2} - \frac{427867625}{29282}\zeta_{5}^{3} - \frac{711476625}{29282}\zeta_{5}^{4}\right)\omega_2^4 + \left(\frac{37724375}{14641}\zeta_{5} + \frac{98718750}{14641}\zeta_{5}^{2} + \frac{198648125}{29282}\zeta_{5}^{3} + \frac{76984375}{29282}\zeta_{5}^{4}\right)\omega_1\omega_2^4$,
\end{center}
\medskip
where
$\omega_1 = \sqrt[2]{ - 47}$,
$\omega_2 = \sqrt[5]{\left( - \frac{38961}{6250}\zeta_{5} - \frac{19988}{3125}\zeta_{5}^{2} - \frac{2653}{3125}\zeta_{5}^{3} + \frac{19899}{6250}\zeta_{5}^{4}\right) + \left(\frac{299}{1250}\zeta_{5} + \frac{143}{125}\zeta_{5}^{2} + \frac{182}{125}\zeta_{5}^{3} + \frac{1001}{1250}\zeta_{5}^{4}\right)\omega_1}$.\\

\noindent $N=47$: The roots of the polynomial $y^{5} - 5y^{3} - 20y^{2} - 24y - 19$ are
\medskip \begin{center} $\displaystyle
y = \left( - \frac{2}{11}\zeta_{5} - \frac{1}{11}\zeta_{5}^{2} + \frac{4}{11}\zeta_{5}^{3} - \frac{4}{11}\zeta_{5}^{4}\right)\omega_2 + \left(\frac{343205}{203522}\zeta_{5} - \frac{527035}{203522}\zeta_{5}^{2} + \frac{251290}{101761}\zeta_{5}^{3} - \frac{309475}{203522}\zeta_{5}^{4}\right)\omega_2^2 + \left(\frac{139289}{203522}\zeta_{5} + \frac{54833}{203522}\zeta_{5}^{2} + \frac{25931}{101761}\zeta_{5}^{3} + \frac{142981}{203522}\zeta_{5}^{4}\right)\omega_1\omega_2^2 + \left(\frac{568028425}{64923518}\zeta_{5} - \frac{19404515}{32461759}\zeta_{5}^{2} + \frac{196212305}{32461759}\zeta_{5}^{3} + \frac{156921685}{32461759}\zeta_{5}^{4}\right)\omega_2^3  + \left(\frac{36878525}{64923518}\zeta_{5} + \frac{27017700}{32461759}\zeta_{5}^{2} - \frac{5814525}{32461759}\zeta_{5}^{3} + \frac{1370000}{1119371}\zeta_{5}^{4}\right)\omega_1\omega_2^3 + \left(\frac{946600115000}{10355301121}\zeta_{5} + \frac{786413529375}{20710602242}\zeta_{5}^{2} + \frac{340940394375}{10355301121}\zeta_{5}^{3} + \frac{1955629024375}{20710602242}\zeta_{5}^{4}\right)\omega_2^4 + \left( - \frac{435681550}{85581001}\zeta_{5} + \frac{1176486975}{171162002}\zeta_{5}^{2} - \frac{631948875}{85581001}\zeta_{5}^{3} + \frac{635172125}{171162002}\zeta_{5}^{4}\right)\omega_1\omega_2^4$,
\end{center}
\medskip
where
$\omega_1 = \sqrt[2]{ - 47}$,
$\omega_2 = \sqrt[5]{\left(\frac{268}{5}\zeta_{5} + \frac{1679}{50}\zeta_{5}^{2} + \frac{889}{50}\zeta_{5}^{3} - \frac{49}{5}\zeta_{5}^{4}\right) + \left( - \frac{1759}{625}\zeta_{5} - \frac{9739}{1250}\zeta_{5}^{2} - \frac{9211}{1250}\zeta_{5}^{3} - \frac{4366}{625}\zeta_{5}^{4}\right)\omega_1}$.\\

\noindent $N=51$: The roots of the polynomial $y^{3} - 2y^{2} - 4y - 4$ are
$$y = \frac{2}{3} + \left(\frac{1}{3}\zeta_{3} - \frac{1}{3}\zeta_{3}^{2}\right)\omega_2 - \frac{49}{128}\omega_2^2 + \left( - \frac{3}{128}\zeta_{3} + \frac{3}{128}\zeta_{3}^{2}\right)\omega_1\omega_2^2,$$ where
$\omega_1 = \sqrt[2]{ - 51}$,
$\omega_2 = \sqrt[3]{\left(\frac{98}{9}\zeta_{3} - \frac{98}{9}\zeta_{3}^{2}\right) + 2\omega_1}$.\\

\noindent $N=51$: The roots of the polynomial $y^{4} + 2y^{3} + 3y^{2} - 2y + 1$ are
$$ y=- \frac{1}{2}-\omega_1 + \frac{1}{2}\omega_2,$$ where
$\omega_1 = \sqrt[2]{-1}$, $\omega_2 = \sqrt[2]{1 + 4\omega_1}$,
$\omega_3 = \sqrt[2]{1 - 4\omega_1}$.\\

\noindent $N=55$: The roots of the polynomial $y^{3} + 3y^{2}-y - 7$ are
$$y = -1 + \left(\frac{1}{3}\zeta_{3} - \frac{1}{3}\zeta_{3}^{2}\right)\omega_2 - \frac{3}{8}\omega_2^2 + \left( - \frac{1}{24}\zeta_{3} + \frac{1}{24}\zeta_{3}^{2}\right)\omega_1\omega_2^2,$$ where
$\omega_1 = \sqrt[2]{ - 11}$,
$\omega_2 = \sqrt[3]{\left(6\zeta_{3} - 6\zeta_{3}^{2}\right) + 2\omega_1}$.\\

\noindent $N=59$: The roots of the polynomial $y^{3} + 2y^{2} + 1$ are
$$y =  - \frac{2}{3} + \left(\frac{1}{3}\zeta_{3} - \frac{1}{3}\zeta_{3}^{2}\right)\omega_2 + \frac{43}{32}\omega_2^2 + \left( - \frac{3}{32}\zeta_{3} + \frac{3}{32}\zeta_{3}^{2}\right)\omega_1\omega_2^2,$$ where
$\omega_1 = \sqrt[2]{ - 59}$,
$\omega_2 = \sqrt[3]{\left( - \frac{43}{18}\zeta_{3} + \frac{43}{18}\zeta_{3}^{2}\right) + \frac{1}{2}\omega_1}$.\\

\noindent $N=59$: The roots of the polynomial $y^{9} + 2y^{8} - 4y^{7} - 21y^{6} - 44y^{5} - 60y^{4} - 61y^{3} - 46y^{2} - 24y - 11$ are
\medskip \begin{center} $\displaystyle
y = - \frac{2}{9} + \left(\frac{1}{9}\zeta_{3} - \frac{1}{9}\zeta_{3}^{2}\right)\omega_2 - \frac{299}{4704}\omega_2^2 + \left( - \frac{1}{1568}\zeta_{3} + \frac{1}{1568}\zeta_{3}^{2}\right)\omega_1\omega_2^2 + \left(\frac{939}{2131}\zeta_{3} - \frac{466}{2131}\zeta_{3}^{2}\right)\omega_3 + \left(\frac{158}{6393}\zeta_{3} + \frac{641}{6393}\zeta_{3}^{2}\right)\omega_1\omega_3 + \left(\frac{2543}{59668}\zeta_{3} - \frac{7419}{59668}\zeta_{3}^{2}\right)\omega_2\omega_3 + \left(\frac{409}{25572}\zeta_{3} + \frac{515}{179004}\zeta_{3}^{2}\right)\omega_1\omega_2\omega_3 + \left(\frac{4849}{1670704}\zeta_{3} - \frac{13273}{417676}\zeta_{3}^{2}\right)\omega_2^2\omega_3 + \left(\frac{10113}{1670704}\zeta_{3} + \frac{157}{119336}\zeta_{3}^{2}\right)\omega_1\omega_2^2\omega_3 + \left( - \frac{121442}{4541161}\zeta_{3} + \frac{1997833}{4541161}\zeta_{3}^{2}\right)\omega_3^2 + \left( - \frac{137714}{4541161}\zeta_{3} + \frac{53463}{4541161}\zeta_{3}^{2}\right)\omega_1\omega_3^2 + \left(\frac{53144863}{254305016}\zeta_{3} + \frac{15588401}{254305016}\zeta_{3}^{2}\right)\omega_2\omega_3^2 + \left( - \frac{26755}{254305016}\zeta_{3} + \frac{1872733}{254305016}\zeta_{3}^{2}\right)\omega_1\omega_2\omega_3^2 + \left(\frac{78731649}{1780135112}\zeta_{3} - \frac{293856}{222516889}\zeta_{3}^{2}\right)\omega_2^2\omega_3^2+\left(\frac{5454327}{1780135112}\zeta_{3} + \frac{4647105}{445033778}\zeta_{3}^{2}\right)\omega_1\omega_2^2\omega_3^2$,
\end{center}
\medskip
where
$\omega_1 = \sqrt[2]{ - 59}$,
$\omega_2 = \sqrt[3]{\left(\frac{299}{18}\zeta_{3} - \frac{299}{18}\zeta_{3}^{2}\right) + \frac{1}{2}\omega_1}$, $\omega_3 =\sqrt[3]{a}$, and
$a=\left(  \frac{-218}{27}\zeta_{3} - \frac{226}{27}\zeta_{3}^{2}\right) + \left( \frac{-22}{27}\zeta_{3} + \frac{14}{27}\zeta_{3}^{2}\right)\omega_1 - \left(  \frac{593}{189}\zeta_{3} + \frac{173}{63}\zeta_{3}^{2}\right)\omega_2 + \left(  \frac{-1}{7}\zeta_{3} + \frac{5}{21}\zeta_{3}^{2}\right)\omega_1\omega_2 - \left( \frac{3553}{5292}\zeta_{3} + \frac{445}{2646}\zeta_{3}^{2}\right)\omega_2^2 + \left(\frac{19}{588}\zeta_{3} + \frac{2}{21}\zeta_{3}^{2}\right)\omega_1\omega_2^2$.\\

\noindent $N=62$: The roots of the polynomial $y^{3} + 4y^{2} + 5y + 3$ are
$$y = - \frac{4}{3} + \left(\frac{1}{3}\zeta_{3} - \frac{1}{3}\zeta_{3}^{2}\right)\omega_2 + \frac{29}{2}\omega_2^2 + \left( - \frac{3}{2}\zeta_{3} + \frac{3}{2}\zeta_{3}^{2}\right)\omega_1\omega_2^2,$$ where
$\omega_1 = \sqrt[2]{ - 31}$,
$\omega_2 = \sqrt[3]{\left( - \frac{29}{18}\zeta_{3} + \frac{29}{18}\zeta_{3}^{2}\right) + \frac{1}{2}\omega_1}$.\\

\noindent $N=62$: The roots of the polynomial $y^{3} + y-1$ are
$$y = \left(\frac{1}{3}\zeta_{3} - \frac{1}{3}\zeta_{3}^{2}\right)\omega_2 - \frac{3}{2}\omega_2^2 + \left( - \frac{1}{6}\zeta_{3} + \frac{1}{6}\zeta_{3}^{2}\right)\omega_1\omega_2^2,$$ where
$\omega_1 = \sqrt[2]{ - 31}$,
$\omega_2 = \sqrt[3]{\left(\frac{3}{2}\zeta_{3} - \frac{3}{2}\zeta_{3}^{2}\right) + \frac{1}{2}\omega_1}$.\\

\noindent $N=62$: The roots of the polynomial $y^{4} - 2y^{3} - 3y^{2} - 4y + 4$ are
$$y = \frac{1}{2}-\omega_1 + \frac{1}{2}\omega_2,$$ where
$\omega_1 = \sqrt[2]{2}$,
$\omega_2 = \sqrt[2]{1 - 4\omega_1}$.\\

\noindent $N=66$: The roots of the polynomial $y^{3} - 4y + 4$ are
$$y = \left(\frac{1}{3}\zeta_{3} - \frac{1}{3}\zeta_{3}^{2}\right)\omega_2 + \frac{3}{8}\omega_2^2 + \left( - \frac{1}{24}\zeta_{3} + \frac{1}{24}\zeta_{3}^{2}\right)\omega_1\omega_2^2,$$ where
$\omega_1 = \sqrt[2]{ - 11}$,
$\omega_2 = \sqrt[3]{\left( - 6\zeta_{3} + 6\zeta_{3}^{2}\right) + 2\omega_1}$.\\

\noindent $N=69$: The roots of the polynomial $y^{3} + 4y^{2} + 7y + 5$ are
$$y = - \frac{4}{3} + \left(\frac{1}{3}\zeta_{3} - \frac{1}{3}\zeta_{3}^{2}\right)\omega_2 + \frac{11}{50}\omega_2^2 + \left( - \frac{3}{50}\zeta_{3} + \frac{3}{50}\zeta_{3}^{2}\right)\omega_1\omega_2^2,$$ where
$\omega_1 = \sqrt[2]{ - 23}$,
$\omega_2 = \sqrt[3]{\left( - \frac{11}{18}\zeta_{3} + \frac{11}{18}\zeta_{3}^{2}\right) + \frac{1}{2}\omega_1}$.\\

\noindent $N=69$: The roots of the polynomial $y^{3}-y + 1$ are
$$y = \left(\frac{1}{3}\zeta_{3} - \frac{1}{3}\zeta_{3}^{2}\right)\omega_2 + \frac{3}{2}\omega_2^2 + \left( - \frac{1}{6}\zeta_{3} + \frac{1}{6}\zeta_{3}^{2}\right)\omega_1\omega_2^2,$$ where
$\omega_1 = \sqrt[2]{ - 23}$,
$\omega_2 = \sqrt[3]{\left( - \frac{3}{2}\zeta_{3} + \frac{3}{2}\zeta_{3}^{2}\right) + \frac{1}{2}\omega_1}$.\\

\noindent $N=69$: The roots of the polynomial $y^{4} - 2y^{3} - 5y^{2} + 6y - 3$ are
$$\frac{1}{2}+\sqrt[2]{\frac{13}{4} + 2\omega_1},$$ where
$\omega_1 = \sqrt[2]{3}$.\\

\noindent $N=70$: The roots of the polynomial $y^{3} + 2y^{2} + 4$ are
$$y = - \frac{2}{3} + \left(\frac{1}{3}\zeta_{3} - \frac{1}{3}\zeta_{3}^{2}\right)\omega_2 + \frac{31}{8}\omega_2^2 + \left( - \frac{3}{8}\zeta_{3} + \frac{3}{8}\zeta_{3}^{2}\right)\omega_1\omega_2^2,$$ where
$\omega_1 = \sqrt[2]{ - 35}$,
$\omega_2 = \sqrt[3]{\left( - \frac{62}{9}\zeta_{3} + \frac{62}{9}\zeta_{3}^{2}\right) + 2\omega_1}$.\\

\noindent $N=71$: The roots of the polynomial $y^{7} + 4y^{6} + 5y^{5} + y^{4} - 3y^{3} - 2y^{2} + 1$ are
\medskip \begin{center} $\displaystyle
y = - \frac{4}{7} + \left( - \frac{32}{43}\zeta_{7} - \frac{16}{43}\zeta_{7}^{2} - \frac{24}{43}\zeta_{7}^{3} - \frac{20}{43}\zeta_{7}^{4} - \frac{22}{43}\zeta_{7}^{5} - \frac{21}{43}\zeta_{7}^{6}\right)\omega_2 + \left(\frac{174622105}{34794482}\zeta_{7} - \frac{9844975}{34794482}\zeta_{7}^{2} + \frac{61229896}{17397241}\zeta_{7}^{3} + \frac{36271795}{34794482}\zeta_{7}^{4} + \frac{97890387}{34794482}\zeta_{7}^{5} + \frac{90218359}{34794482}\zeta_{7}^{6}\right)\omega_2^2 + \left(\frac{8711297}{34794482}\zeta_{7} + \frac{5725349}{34794482}\zeta_{7}^{2} + \frac{100555}{17397241}\zeta_{7}^{3} + \frac{14354123}{34794482}\zeta_{7}^{4} - \frac{4750193}{34794482}\zeta_{7}^{5} + \frac{13232975}{34794482}\zeta_{7}^{6}\right)\omega_1\omega_2^2 + \left(\frac{555798451705}{72563892211}\zeta_{7} - \frac{679386405271}{145127784422}\zeta_{7}^{2} + \frac{1588398342496}{72563892211}\zeta_{7}^{3} - \frac{701595040909}{72563892211}\zeta_{7}^{4} + \frac{497352533055}{72563892211}\zeta_{7}^{5} + \frac{1596681652993}{145127784422}\zeta_{7}^{6}\right)\omega_2^3 + \left(\frac{228601539705}{72563892211}\zeta_{7} + \frac{368589545135}{145127784422}\zeta_{7}^{2} + \frac{97604929758}{72563892211}\zeta_{7}^{3} + \frac{179358836139}{72563892211}\zeta_{7}^{4} + \frac{111147366351}{72563892211}\zeta_{7}^{5} + \frac{550999991661}{145127784422}\zeta_{7}^{6}\right)\omega_1\omega_2^3  + \left(\frac{159502491007496}{7038697544467}\zeta_{7} - \frac{733719054247991}{14077395088934}\zeta_{7}^{2} - \frac{39648191022636}{7038697544467}\zeta_{7}^{3} \right. - \left. \frac{12121839188319}{14077395088934}\zeta_{7}^{4} + \frac{24475449588212}{7038697544467}\zeta_{7}^{5} - \frac{348509408996700}{7038697544467}\zeta_{7}^{6}\right)\omega_2^4 + \left( - \frac{307216450789206}{302663994412081}\zeta_{7} - \frac{2759125653830875}{605327988824162}\zeta_{7}^{2} + \frac{832323889505164}{302663994412081}\zeta_{7}^{3} \right. + \left. \frac{403893925611603}{605327988824162}\zeta_{7}^{4} - \frac{1857600026573158}{302663994412081}\zeta_{7}^{5} + \frac{640718340638572}{302663994412081}\zeta_{7}^{6}\right)\omega_1\omega_2^4 + \left(\frac{450750305278051280803}{2524823041385579702}\zeta_{7} + \frac{1078197572499656134666}{1262411520692789851}\zeta_{7}^{2} - \frac{320492123219545165049}{2524823041385579702}\zeta_{7}^{3} \right. + \left. \frac{830452693893445825419}{1262411520692789851}\zeta_{7}^{4} - \frac{67276531034397215342}{1262411520692789851}\zeta_{7}^{5} + \frac{1103882852532728227946}{1262411520692789851}\zeta_{7}^{6}\right)\omega_2^5 + \left( - \frac{209785041589961724971}{2524823041385579702}\zeta_{7} + \frac{16093904888430380676}{1262411520692789851}\zeta_{7}^{2} - \frac{260213121296421268661}{2524823041385579702}\zeta_{7}^{3} \right. - \left. \frac{40922785863208341771}{1262411520692789851}\zeta_{7}^{4} + \frac{1352721906365722788}{1262411520692789851}\zeta_{7}^{5} - \frac{136005759492209605558}{1262411520692789851}\zeta_{7}^{6}\right)\omega_1\omega_2^5 + \left( - \frac{19032215189659588806887655}{5265518452809626468521}\zeta_{7} - \frac{15856112685852658054179302}{5265518452809626468521}\zeta_{7}^{2} + \frac{944315639184097376399798}{5265518452809626468521}\zeta_{7}^{3} \right. - \left. \frac{76948579742787286330488573}{10531036905619252937042}\zeta_{7}^{4} + \frac{29839467008566836734900983}{10531036905619252937042}\zeta_{7}^{5} - \frac{35534190049531523735890946}{5265518452809626468521}\zeta_{7}^{6}\right)\omega_2^6 + \left(\frac{5728492596102583740611241}{5265518452809626468521}\zeta_{7} - \frac{684440361825865132102636}{5265518452809626468521}\zeta_{7}^{2} + \frac{4510605264243392773952976}{5265518452809626468521}\zeta_{7}^{3} \right. + \left. \frac{2598011376578631982788771}{10531036905619252937042}\zeta_{7}^{4} + \frac{5280886452082331794376421}{10531036905619252937042}\zeta_{7}^{5} + \frac{3364734822070438206187630}{5265518452809626468521}\zeta_{7}^{6}\right)\omega_1\omega_2^6,$
\end{center}
\medskip
where
$\omega_1 = \sqrt[2]{ - 71}$,
$\omega_2 = \sqrt[7]{a}$, and $a=\left( - \frac{1012993}{823543}\zeta_{7} - \frac{1543516}{823543}\zeta_{7}^{2} - \frac{2818855}{1647086}\zeta_{7}^{3} - \frac{1038797}{1647086}\zeta_{7}^{4} + \frac{388015}{823543}\zeta_{7}^{5} + \frac{586927}{823543}\zeta_{7}^{6}\right) + \left( - \frac{179}{117649}\zeta_{7} - \frac{402}{117649}\zeta_{7}^{2} - \frac{3251}{235298}\zeta_{7}^{3} - \frac{3371}{235298}\zeta_{7}^{4} - \frac{2349}{117649}\zeta_{7}^{5} + \frac{67}{117649}\zeta_{7}^{6}\right)\omega_1$.\\

\noindent $N=71$: The roots of the polynomial $y^{7} - 7y^{5} - 11y^{4} + 5y^{3} + 18y^{2} + 4y - 11$ are
\medskip \begin{center} $\displaystyle
y = \left( - \frac{8}{43}\zeta_{7} - \frac{4}{43}\zeta_{7}^{2} - \frac{6}{43}\zeta_{7}^{3} - \frac{5}{43}\zeta_{7}^{4} + \frac{16}{43}\zeta_{7}^{5} - \frac{16}{43}\zeta_{7}^{6}\right)\omega_2 + \left(\frac{11319}{3698}\zeta_{7} + \frac{9814}{1849}\zeta_{7}^{2} + \frac{17493}{3698}\zeta_{7}^{3} + \frac{8547}{3698}\zeta_{7}^{4} - \frac{2849}{3698}\zeta_{7}^{5} - \frac{4991}{3698}\zeta_{7}^{6}\right)\omega_2^2 + \left(\frac{117}{3698}\zeta_{7} + \frac{355}{1849}\zeta_{7}^{2} + \frac{1891}{3698}\zeta_{7}^{3} + \frac{55}{86}\zeta_{7}^{4} + \frac{2041}{3698}\zeta_{7}^{5} + \frac{1059}{3698}\zeta_{7}^{6}\right)\omega_1\omega_2^2 + \left(\frac{425334}{79507}\zeta_{7} + \frac{421750}{79507}\zeta_{7}^{2} + \frac{1678033}{79507}\zeta_{7}^{3} + \frac{1600445}{159014}\zeta_{7}^{4} + \frac{1175174}{79507}\zeta_{7}^{5} + \frac{113953}{159014}\zeta_{7}^{6}\right)\omega_2^3 + \left( - \frac{160622}{79507}\zeta_{7} - \frac{100940}{79507}\zeta_{7}^{2} - \frac{101871}{79507}\zeta_{7}^{3} - \frac{61201}{159014}\zeta_{7}^{4} + \frac{57232}{79507}\zeta_{7}^{5} - \frac{66591}{159014}\zeta_{7}^{6}\right)\omega_1\omega_2^3 + \left(\frac{167504736}{3418801}\zeta_{7} - \frac{80901352}{3418801}\zeta_{7}^{2} + \frac{224386484}{3418801}\zeta_{7}^{3} + \frac{12101040}{3418801}\zeta_{7}^{4} + \frac{351506743}{6837602}\zeta_{7}^{5} + \frac{222848815}{6837602}\zeta_{7}^{6}\right)\omega_2^4+ \left( - \frac{26381698}{3418801}\zeta_{7} - \frac{16942828}{3418801}\zeta_{7}^{2} - \frac{15892954}{3418801}\zeta_{7}^{3} - \frac{30669296}{3418801}\zeta_{7}^{4} - \frac{2143407}{6837602}\zeta_{7}^{5} - \frac{1450155}{159014}\zeta_{7}^{6}\right)\omega_1\omega_2^4 + \left( - \frac{7809712120}{147008443}\zeta_{7} - \frac{50788900085}{294016886}\zeta_{7}^{2} - \frac{35919887815}{294016886}\zeta_{7}^{3} - \frac{9212056301}{147008443}\zeta_{7}^{4} + \frac{5337121160}{147008443}\zeta_{7}^{5} + \frac{12886679442}{147008443}\zeta_{7}^{6}\right)\omega_2^5 + \left( - \frac{29406076}{147008443}\zeta_{7} - \frac{2910042527}{294016886}\zeta_{7}^{2} - \frac{5687228463}{294016886}\zeta_{7}^{3} - \frac{4584773641}{147008443}\zeta_{7}^{4} - \frac{3140520348}{147008443}\zeta_{7}^{5} - \frac{2175840394}{147008443}\zeta_{7}^{6}\right)\omega_1\omega_2^5 + \left(\frac{9047674208531}{12642726098}\zeta_{7} - \frac{152859160790}{6321363049}\zeta_{7}^{2} + \frac{1386633786643}{6321363049}\zeta_{7}^{3} + \frac{1281023018884}{6321363049}\zeta_{7}^{4} - \frac{418201505253}{12642726098}\zeta_{7}^{5} + \frac{4413726337888}{6321363049}\zeta_{7}^{6}\right)\omega_2^6 + \left(\frac{425471891719}{12642726098}\zeta_{7} - \frac{120550819788}{6321363049}\zeta_{7}^{2} + \frac{349074928909}{6321363049}\zeta_{7}^{3} - \frac{333690768712}{6321363049}\zeta_{7}^{4} + \frac{266737301453}{12642726098}\zeta_{7}^{5} - \frac{221726141014}{6321363049}\zeta_{7}^{6}\right)\omega_1\omega_2^6,$
\end{center}
\medskip
where
$\omega_1 = \sqrt[2]{ - 71}$,
$\omega_2 = \sqrt[7]{a}$, and $a=\left(\frac{28984}{343}\zeta_{7} + \frac{6945971}{33614}\zeta_{7}^{2} + \frac{1143959}{16807}\zeta_{7}^{3} + \frac{123772}{16807}\zeta_{7}^{4} + \frac{5818243}{33614}\zeta_{7}^{5} + \frac{384197}{2401}\zeta_{7}^{6}\right) + \left( - \frac{1977412}{117649}\zeta_{7} - \frac{826843}{235298}\zeta_{7}^{2} + \frac{867271}{117649}\zeta_{7}^{3} - \frac{1264430}{117649}\zeta_{7}^{4} - \frac{3199837}{235298}\zeta_{7}^{5} + \frac{685079}{117649}\zeta_{7}^{6}\right)\omega_1$.\\

\noindent $N=78$: The roots of the polynomial $y^{3} + y^{2} - 4$ are
$$y = - \frac{1}{3} + \left(\frac{1}{3}\zeta_{3} - \frac{1}{3}\zeta_{3}^{2}\right)\omega_2 - 53\omega_2^2 + \left( - 6\zeta_{3} + 6\zeta_{3}^{2}\right)\omega_1\omega_2^2,$$ where
$\omega_1 = \sqrt[2]{ - 26}$,
$\omega_2 = \sqrt[3]{\left(\frac{53}{9}\zeta_{3} - \frac{53}{9}\zeta_{3}^{2}\right) + 2\omega_1}$.\\

\noindent $N=87$: The roots of the polynomial $y^{3} + 2y^{2} + 3y + 3$ are
$$y = - \frac{2}{3} + \left(\frac{1}{3}\zeta_{3} - \frac{1}{3}\zeta_{3}^{2}\right)\omega_2 + \frac{43}{50}\omega_2^2 + \left( - \frac{3}{50}\zeta_{3} + \frac{3}{50}\zeta_{3}^{2}\right)\omega_1\omega_2^2,$$ where
$\omega_1 = \sqrt[2]{ - 87}$,
$\omega_2 = \sqrt[3]{\left( - \frac{43}{18}\zeta_{3} + \frac{43}{18}\zeta_{3}^{2}\right) + \frac{1}{2}\omega_1}$.\\

\noindent $N=87$: The roots of the polynomial $y^{3} - 2y^{2}-y-1$ are
$$y = \frac{2}{3} + \left(\frac{1}{3}\zeta_{3} - \frac{1}{3}\zeta_{3}^{2}\right)\omega_2 - \frac{61}{98}\omega_2^2 + \left( - \frac{3}{98}\zeta_{3} + \frac{3}{98}\zeta_{3}^{2}\right)\omega_1\omega_2^2,$$ where
$\omega_1 = \sqrt[2]{ - 87}$,
$\omega_2 = \sqrt[3]{\left(\frac{61}{18}\zeta_{3} - \frac{61}{18}\zeta_{3}^{2}\right) + \frac{1}{2}\omega_1}$.\\

\noindent $N=87$: The roots of the polynomial $y^{6} + 2y^{5} + 7y^{4} + 6y^{3} + 13y^{2} + 4y + 8$ are
$$y = - \frac{1}{3} + \frac{2}{3}\omega_1 + \left(\frac{1}{3}\zeta_{3} - \frac{1}{3}\zeta_{3}^{2}\right)\omega_3 - \frac{269}{1352}\omega_3^2 - \frac{109}{338}\omega_1\omega_3^2 + \left( - \frac{15}{1352}\zeta_{3} + \frac{15}{1352}\zeta_{3}^{2}\right)\omega_2\omega_3^2 + \left(\frac{9}{338}\zeta_{3} - \frac{9}{338}\zeta_{3}^{2}\right)\omega_1\omega_2\omega_3^2,$$ where
$\omega_1 = \sqrt[2]{-1}$,
$\omega_2 = \sqrt[2]{29}$,
$\omega_3 = \sqrt[3]{\left( - \frac{23}{18}\zeta_{3} + \frac{23}{18}\zeta_{3}^{2}\right) + \left(\frac{16}{9}\zeta_{3} - \frac{16}{9}\zeta_{3}^{2}\right)\omega_1 + \frac{1}{2}\omega_2}$.\\

\noindent $N=94$: The roots of the polynomial $y^{4} - 2y^{3} - 3y^{2} + 4y - 4$ are
$$y = \frac{1}{2}+\sqrt[2]{\frac{9}{4} + 2\omega_1},$$ where
$\omega_1 = \sqrt[2]{2}$.\\

\noindent $N=94$: The roots of the polynomial $y^{5} + 4y^{4} + 3y^{3} - 2y^{2} + 2y + 5$ are
\medskip \begin{center} $\displaystyle
y =  - \frac{4}{5} + \left(\frac{1}{11}\zeta_{5} + \frac{6}{11}\zeta_{5}^{2} - \frac{2}{11}\zeta_{5}^{3} + \frac{2}{11}\zeta_{5}^{4}\right)\omega_2 + \left(\frac{120225}{203522}\zeta_{5} - \frac{34445}{101761}\zeta_{5}^{2} - \frac{315635}{203522}\zeta_{5}^{3} - \frac{137780}{101761}\zeta_{5}^{4}\right)\omega_2^2 + \left(\frac{22595}{203522}\zeta_{5} + \frac{21695}{101761}\zeta_{5}^{2} + \frac{27685}{203522}\zeta_{5}^{3} + \frac{3290}{101761}\zeta_{5}^{4}\right)\omega_1\omega_2^2 + \left( - \frac{1484750}{2951069}\zeta_{5} - \frac{4171900}{2951069}\zeta_{5}^{2} - \frac{9577875}{5902138}\zeta_{5}^{3} - \frac{57325}{101761}\zeta_{5}^{4}\right)\omega_2^3 + \left(\frac{8757475}{32461759}\zeta_{5} + \frac{5889450}{32461759}\zeta_{5}^{2} - \frac{5829025}{64923518}\zeta_{5}^{3} - \frac{6317100}{32461759}\zeta_{5}^{4}\right)\omega_1\omega_2^3 + \left(\frac{40976233875}{20710602242}\zeta_{5} + \frac{37317619875}{20710602242}\zeta_{5}^{2} + \frac{11048949500}{10355301121}\zeta_{5}^{3} - \frac{18276875}{20710602242}\zeta_{5}^{4}\right)\omega_2^4 + \left( - \frac{448191875}{20710602242}\zeta_{5} + \frac{3454465625}{20710602242}\zeta_{5}^{2} + \frac{3788259375}{10355301121}\zeta_{5}^{3} + \frac{3261141875}{20710602242}\zeta_{5}^{4}\right)\omega_1\omega_2^4$,
\end{center}
\medskip
where
$\omega_1 = \sqrt[2]{ - 47}$,
$\omega_2 = \sqrt[5]{\left( - \frac{6568}{3125}\zeta_{5} + \frac{2649}{6250}\zeta_{5}^{2} - \frac{9681}{6250}\zeta_{5}^{3} - \frac{3263}{3125}\zeta_{5}^{4}\right) + \left( - \frac{97}{625}\zeta_{5} - \frac{49}{250}\zeta_{5}^{2} - \frac{1}{250}\zeta_{5}^{3} - \frac{178}{625}\zeta_{5}^{4}\right)\omega_1}$.\\

\noindent $N=94$: The roots of the polynomial $y^{5}-y^{3} + 2y^{2} - 2y + 1$ are
\medskip \begin{center} $\displaystyle
y = \left(\frac{4}{11}\zeta_{5} + \frac{2}{11}\zeta_{5}^{2} + \frac{3}{11}\zeta_{5}^{3} + \frac{8}{11}\zeta_{5}^{4}\right)\omega_2 + \left(\frac{13445}{242}\zeta_{5} + \frac{4310}{121}\zeta_{5}^{2} + \frac{2975}{242}\zeta_{5}^{3} + \frac{16935}{242}\zeta_{5}^{4}\right)\omega_2^2 + \left(\frac{1499}{242}\zeta_{5} - \frac{431}{121}\zeta_{5}^{2} + \frac{1457}{242}\zeta_{5}^{3} + \frac{61}{242}\zeta_{5}^{4}\right)\omega_1\omega_2^2 + \left( - \frac{326995}{1331}\zeta_{5} - \frac{270225}{1331}\zeta_{5}^{2} - \frac{35765}{1331}\zeta_{5}^{3} - \frac{946355}{2662}\zeta_{5}^{4}\right)\omega_2^3 + \left( - \frac{51750}{1331}\zeta_{5} + \frac{18125}{1331}\zeta_{5}^{2} - \frac{43350}{1331}\zeta_{5}^{3} - \frac{27425}{2662}\zeta_{5}^{4}\right)\omega_1\omega_2^3 + \left(\frac{206123125}{29282}\zeta_{5} + \frac{112721375}{14641}\zeta_{5}^{2} - \frac{11899875}{29282}\zeta_{5}^{3} + \frac{176397000}{14641}\zeta_{5}^{4}\right)\omega_2^4 + \left(\frac{46170925}{29282}\zeta_{5} - \frac{4426725}{14641}\zeta_{5}^{2} + \frac{34002775}{29282}\zeta_{5}^{3} + \frac{9838750}{14641}\zeta_{5}^{4}\right)\omega_1\omega_2^4$,
\end{center}
\medskip
where
$\omega_1 = \sqrt[2]{ - 47}$,
$\omega_2 = \sqrt[5]{\left( - \frac{1066}{125}\zeta_{5} - \frac{2223}{250}\zeta_{5}^{2} - \frac{273}{250}\zeta_{5}^{3} + \frac{559}{125}\zeta_{5}^{4}\right) + \left( - \frac{199}{625}\zeta_{5} - \frac{1949}{1250}\zeta_{5}^{2} - \frac{2501}{1250}\zeta_{5}^{3} - \frac{676}{625}\zeta_{5}^{4}\right)\omega_1}$.\\

\noindent $N=95$: The roots of the polynomial $y^{3} + y^{2}-y + 3$ are
$$y = - \frac{1}{3} + \left(\frac{1}{3}\zeta_{3} - \frac{1}{3}\zeta_{3}^{2}\right)\omega_2 + \frac{23}{8}\omega_2^2 + \left( - \frac{3}{8}\zeta_{3} + \frac{3}{8}\zeta_{3}^{2}\right)\omega_1\omega_2^2,$$ where
$\omega_1 = \sqrt[2]{ - 19}$,
$\omega_2 = \sqrt[3]{\left( - \frac{46}{9}\zeta_{3} + \frac{46}{9}\zeta_{3}^{2}\right) + 2\omega_1}$.\\

\noindent $N=95$: The roots of the polynomial $y^{4} + y^{3} - 2y^{2} + 2y-1$ are
$$y = - \frac{1}{4} - \frac{1}{4}\omega_1 + \frac{1}{2}\omega_2,$$ where
$\omega_1 = \sqrt[2]{5}$,
$\omega_2 = \sqrt[2]{\frac{7}{2} + \frac{5}{2}\omega_1}$,
$\omega_3 = \sqrt[2]{\frac{7}{2} - \frac{5}{2}\omega_1}$.\\

\noindent $N=95$: The roots of the polynomial $y^{4} + y^{3} - 6y^{2} - 10y - 5$ are
$$y = - \frac{1}{4} - \frac{1}{4}\omega_1 + \frac{1}{2}\omega_2,$$ where
$\omega_1 = \sqrt[2]{5}$,
$\omega_2 = \sqrt[2]{\frac{23}{2} - \frac{11}{2}\omega_1}$,
$\omega_3 = \sqrt[2]{\frac{23}{2} + \frac{11}{2}\omega_1}$.\\

\noindent $N=105$: The roots of the polynomial $y^{3} + y^{2}-y - 5$ are
$$y = - \frac{1}{3} + \left(\frac{1}{3}\zeta_{3} - \frac{1}{3}\zeta_{3}^{2}\right)\omega_2 - \frac{31}{8}\omega_2^2 + \left( - \frac{3}{8}\zeta_{3} + \frac{3}{8}\zeta_{3}^{2}\right)\omega_1\omega_2^2,$$ where
$\omega_1 = \sqrt[2]{ - 35}$,
$\omega_2 = \sqrt[3]{\left(\frac{62}{9}\zeta_{3} - \frac{62}{9}\zeta_{3}^{2}\right) + 2\omega_1}$.\\

\noindent $N=110$: The roots of the polynomial $y^{3} + y^{2} + 3y - 1$ are
$$y = - \frac{1}{3} + \left(\frac{1}{3}\zeta_{3} - \frac{1}{3}\zeta_{3}^{2}\right)\omega_2 - \frac{13}{32}\omega_2^2 + \left( - \frac{3}{32}\zeta_{3} + \frac{3}{32}\zeta_{3}^{2}\right)\omega_1\omega_2^2,$$ where
$\omega_1 = \sqrt[2]{ - 11}$,
$\omega_2 = \sqrt[3]{\left(\frac{26}{9}\zeta_{3} - \frac{26}{9}\zeta_{3}^{2}\right) + 2\omega_1}$.\\

\noindent $N=110$: The roots of the polynomial $y^{3}-y^{2} - 8$ are
$$y = \frac{1}{3} + \left(\frac{1}{3}\zeta_{3} - \frac{1}{3}\zeta_{3}^{2}\right)\omega_2 - 109\omega_2^2 + \left( - 6\zeta_{3} + 6\zeta_{3}^{2}\right)\omega_1\omega_2^2,$$ where
$\omega_1 = \sqrt[2]{ - 110}$,
$\omega_2 = \sqrt[3]{\left(\frac{109}{9}\zeta_{3} - \frac{109}{9}\zeta_{3}^{2}\right) + 2\omega_1}$.\\

\noindent $N=119$: The roots of the polynomial $y^{5} + 2y^{4} + 3y^{3} + 6y^{2} + 4y + 1$ are
\medskip \begin{center} $\displaystyle
y = - \frac{2}{5} + \left(\frac{1}{11}\zeta_{5} + \frac{6}{11}\zeta_{5}^{2} - \frac{2}{11}\zeta_{5}^{3} + \frac{2}{11}\zeta_{5}^{4}\right)\omega_2 + \left( - \frac{919995}{242}\zeta_{5} - \frac{189535}{242}\zeta_{5}^{2} - \frac{451475}{242}\zeta_{5}^{3} - \frac{379070}{121}\zeta_{5}^{4}\right)\omega_2^2 + \left(\frac{2155}{242}\zeta_{5} - \frac{49165}{242}\zeta_{5}^{2} + \frac{31715}{242}\zeta_{5}^{3} - \frac{23915}{121}\zeta_{5}^{4}\right)\omega_1\omega_2^2 + \left( - \frac{39112100}{1331}\zeta_{5} - \frac{200177425}{1331}\zeta_{5}^{2} + \frac{199089075}{2662}\zeta_{5}^{3} - \frac{224350400}{1331}\zeta_{5}^{4}\right)\omega_2^3 + \left(\frac{30054325}{1331}\zeta_{5} + \frac{3632400}{1331}\zeta_{5}^{2} + \frac{32659325}{2662}\zeta_{5}^{3} + \frac{22207050}{1331}\zeta_{5}^{4}\right)\omega_1\omega_2^3 + \left(\frac{1982313379875}{29282}\zeta_{5} + \frac{64837018875}{29282}\zeta_{5}^{2} + \frac{592532780000}{14641}\zeta_{5}^{3} + \frac{644987037875}{14641}\zeta_{5}^{4}\right)\omega_2^4 +\left(\frac{48932020625}{29282}\zeta_{5} + \frac{116156561875}{29282}\zeta_{5}^{2} - \frac{20773526250}{14641}\zeta_{5}^{3} + \frac{73199106875}{14641}\zeta_{5}^{4}\right)\omega_1\omega_2^4$,
\end{center}
\medskip
where
$\omega_1 = \sqrt[2]{ - 119}$,
$\omega_2 = \sqrt[5]{\left( - \frac{28853}{6250}\zeta_{5} - \frac{11994}{3125}\zeta_{5}^{2} - \frac{3339}{3125}\zeta_{5}^{3} + \frac{10127}{6250}\zeta_{5}^{4}\right) + \left(\frac{179}{1250}\zeta_{5} + \frac{337}{625}\zeta_{5}^{2} + \frac{388}{625}\zeta_{5}^{3} + \frac{521}{1250}\zeta_{5}^{4}\right)\omega_1}$.\\

\noindent $N=119$: The roots of the polynomial $y^{5} - 2y^{4} + 3y^{3} - 6y^{2} - 7$ are
\medskip \begin{center} $\displaystyle
y = \frac{2}{5} + \left(\frac{4}{11}\zeta_{5} + \frac{2}{11}\zeta_{5}^{2} + \frac{3}{11}\zeta_{5}^{3} + \frac{8}{11}\zeta_{5}^{4}\right)\omega_2 + \left( - \frac{19450}{43681}\zeta_{5} + \frac{14420}{43681}\zeta_{5}^{2} + \frac{51895}{87362}\zeta_{5}^{3} - \frac{69165}{87362}\zeta_{5}^{4}\right)\omega_2^2 + \left(\frac{240}{2299}\zeta_{5} - \frac{580}{43681}\zeta_{5}^{2} + \frac{185}{87362}\zeta_{5}^{3} - \frac{5245}{87362}\zeta_{5}^{4}\right)\omega_1\omega_2^2 + \left( - \frac{14010250}{9129329}\zeta_{5} - \frac{14410050}{9129329}\zeta_{5}^{2} - \frac{7904675}{9129329}\zeta_{5}^{3} - \frac{66744275}{18258658}\zeta_{5}^{4}\right)\omega_2^3 + \left(\frac{231300}{829939}\zeta_{5} + \frac{72475}{829939}\zeta_{5}^{2} + \frac{255150}{829939}\zeta_{5}^{3} + \frac{153275}{1659878}\zeta_{5}^{4}\right)\omega_1\omega_2^3 + \left( - \frac{97739257375}{3816059522}\zeta_{5} - \frac{5700115375}{200845238}\zeta_{5}^{2} + \frac{3054410625}{3816059522}\zeta_{5}^{3} - \frac{82637207500}{1908029761}\zeta_{5}^{4}\right)\omega_2^4 + \left(\frac{13549946875}{3816059522}\zeta_{5} - \frac{2597806875}{3816059522}\zeta_{5}^{2} + \frac{9974059375}{3816059522}\zeta_{5}^{3} + \frac{2885295625}{1908029761}\zeta_{5}^{4}\right)\omega_1\omega_2^4$,
\end{center}
\medskip
where
$\omega_1 = \sqrt[2]{ - 119}$,
$\omega_2 = \sqrt[5]{\left(\frac{56383}{6250}\zeta_{5} - \frac{4096}{3125}\zeta_{5}^{2} + \frac{19999}{3125}\zeta_{5}^{3} + \frac{25403}{6250}\zeta_{5}^{4}\right) + \left( - \frac{607}{1250}\zeta_{5} - \frac{359}{625}\zeta_{5}^{2} + \frac{34}{625}\zeta_{5}^{3} - \frac{1093}{1250}\zeta_{5}^{4}\right)\omega_1}$.\\

\medskip

\section{Approximate values of the Haupmoduli at elliptic points}\label{sec A3}

Here we give the numerical value of $j_{N}(e)$ for each elliptic point $e$.  We state the
order of the elliptic point if the order is greater than two.  As in section~\ref{sec 5.2},
one identifies each $j_{N}(e)$ with the exact value of a radical by numerically evaluating
the expressions in Appendix~\ref{sec A2} and comparing with the approximations given below.

\medskip

\begin{tabbing}
\qquad\qquad\qquad\qquad\qquad\ \= \qquad \qquad \= \kill \\

$j_{1}(i)                   $\>$ =\ \  984$ \\
$j_{1}(1/2+i\sqrt{3}/2)     $\>$ =    -744$ \> (at elliptic point of order 3) \\

$j_{2}(i\sqrt{2}/2)         $\>$ =\ \  152$ \\
$j_{2}(1/2+i/2)             $\>$ =    -104$ \> (at elliptic point of order 4) \\

$j_{3}(i\sqrt{3}/3)         $\>$ =\ \   66$ \\
$j_{3}(1/2+i\sqrt{3}/6)     $\>$ =     -42$ \> (at elliptic point of order 6) \\

$j_{5}(i\sqrt{5}/5)         $\>$ =\ \   28.36067977499789696409173669$ \\
$j_{5}(2/5+i/5)             $\>$ =     -16$ \\
$j_{5}(1/2+i\sqrt{5}/10)    $\>$ =     -16.36067977499789696409173669$ \\

$j_{6}(i\sqrt{6}/6)          $\>$ =\ \  22$ \\
$j_{6}(1/3+i\sqrt{2}/6)      $\>$ =    -10$ \\
$j_{6}(1/2+i\sqrt{3}/6)      $\>$ =    -14$ \\

$j_{7}(i\sqrt{7}/7)          $\>$ =\ \  18$ \\
$j_{7}(5/14+i\sqrt{3}/14)    $\>$ =     -9$ \> (at elliptic point of order 3) \\
$j_{7}(1/2+i\sqrt{7}/14)     $\>$ =    -10$ \\

$j_{10}(i\sqrt{10}/10)       $\>$ =\ \  12$ \\
$j_{10}(3/10+i/10)           $\>$ =     -4$ \> (at elliptic point of order 4) \\
$j_{10}(1/2+i\sqrt{5}/10)    $\>$ =     -8$ \\

$j_{11}(i\sqrt{11}/11)       $\>$ =\ \  10.82750081414703474747183074$ \\
$j_{11}(1/3+i\sqrt{11}/33)   $\>$ =     -4.413750407073517373735915369-0.3139738780838152329405829238i$ \\
$j_{11}(1/2+i\sqrt{11}/22)   $\>$ =     -6$ \\
$j_{11}(2/3+i\sqrt{11}/33)   $\>$ =     -4.413750407073517373735915369+0.3139738780838152329405829238i$ \\

$j_{13}(i\sqrt{13}/13)       $\>$ =\ \   9.211102550927978586238442535$ \\
$j_{13}(7/26+i\sqrt{3}/26)   $\>$ =     -3$ \> (at elliptic point of order 3) \\
$j_{13}(5/13+i/13)           $\>$ =     -4$ \\
$j_{13}(1/2+i\sqrt{13}/26)   $\>$ =     -5.211102550927978586238442535$ \\

$j_{14}(i\sqrt{14}/14)       $\>$ =\ \   8.656854249492380195206754897$ \\
$j_{14}(1/4+i\sqrt{7}/28)    $\>$ =     -2$ \\
$j_{14}(1/3+i\sqrt{14}/42)   $\>$ =     -2.656854249492380195206754897$ \\
$j_{14}(1/2+i\sqrt{7}/14)    $\>$ =     -6$ \\

$j_{15}(i\sqrt{15}/15)       $\>$ =\ \   8$ \\
$j_{15}(1/3+i\sqrt{5}/15)    $\>$ =     -3-2i$ \\
$j_{15}(1/2+i\sqrt{15}/30)   $\>$ =     -4$ \\
$j_{15}(2/3+i\sqrt{5}/15)    $\>$ =     -3+2i$ \\

$j_{17}(i\sqrt{17}/17)       $\>$ =\ \   7.261765309064618299064762001$ \\
$j_{17}(4/17+i/17)           $\>$ =     -2$ \\
$j_{17}(1/3+i\sqrt{17}/51)   $\>$ =     -2.561552812808830274910704928-0.7017282255051753291172214932i$ \\
$j_{17}(1/2+i\sqrt{17}/34)   $\>$ =     -4.138659683446957749243352145$ \\
$j_{17}(2/3+i\sqrt{17}/51)   $\>$ =     -2.561552812808830274910704928+0.7017282255051753291172214932i$ \\

$j_{19}(i\sqrt{19}/19)       $\>$ =\ \   6.668980143244541623355874955$ \\
$j_{19}(1/4+i\sqrt{19}/76)   $\>$ =     -1.334490071622270811677937478-0.1360581741513788488229955061i$ \\
$j_{19}(15/38+i\sqrt{3}/38)  $\>$ =     -3$ \> (at elliptic point of order 3) \\
$j_{19}(1/2+i\sqrt{19}/38)   $\>$ =     -4$ \\
$j_{19}(3/4+i\sqrt{19}/76)   $\>$ =     -1.334490071622270811677937478+0.1360581741513788488229955061i$ \\

$j_{21}(i\sqrt{21}/21)       $\>$ =\ \   6.291502622129181181003231507$ \\
$j_{21}(3/14+i\sqrt{3}/42)   $\>$ =\ \   0$ \> (at elliptic point of order 6) \\
$j_{21}(3/7+i\sqrt{3}/21)    $\>$ =     -4$ \\
$j_{21}(1/2+i\sqrt{21}/42)   $\>$ =     -4.291502622129181181003231507$ \\

$j_{22}(i\sqrt{22}/22)       $\>$ =\ \   6$ \\
$j_{22}(1/4+i\sqrt{11}/44)   $\>$ =     -0.8085121160468812529386457318-0.5088517788327379904864224393i$ \\
$j_{22}(4/11+i\sqrt{2}/22)   $\>$ =     -2$ \\
$j_{22}(1/2+i\sqrt{11}/22)   $\>$ =     -4.382975767906237494122708536$ \\
$j_{22}(3/4+i\sqrt{11}/44)   $\>$ =     -0.8085121160468812529386457318+0.5088517788327379904864224393i$ \\

$j_{23}(i\sqrt{23}/23)       $\>$ =\ \   5.729031537980930837932235460$ \\
$j_{23}(1/4+i\sqrt{23}/92)   $\>$ =     -1.337641021377626987019545573-0.5622795120623012438991821449i$ \\
$j_{23}(1/3+i\sqrt{23}/69)   $\>$ =     -1.864515768990465418966117730-0.9419767695671594951043760061i$ \\
$j_{23}(1/2+i\sqrt{23}/46)   $\>$ =     -3.324717957244746025960908855$ \\
$j_{23}(2/3+i\sqrt{23}/69)   $\>$ =     -1.864515768990465418966117730+0.9419767695671594951043760061i$ \\
$j_{23}(3/4+i\sqrt{23}/92)   $\>$ =     -1.337641021377626987019545573+0.5622795120623012438991821449i$ \\

$j_{26}(i\sqrt{26}/26)       $\>$ =\ \   5.357355625885887120858272705$ \\
$j_{26}(5/26+i/26)           $\>$ =\ \   0$ \> (at elliptic point of order 4) \\
$j_{26}(1/3+i\sqrt{26}/78)   $\>$ =     -1.678677812942943560429136352-0.4105957266856133191677959858i$ \\
$j_{26}(1/2+i\sqrt{13}/26)   $\>$ =     -4$ \\
$j_{26}(2/3+i\sqrt{26}/78)   $\>$ =     -1.678677812942943560429136352+0.4105957266856133191677959858i$ \\

$j_{29}(i\sqrt{29}/29)       $\>$ =\ \   4.924500920878021256173037168$ \\
$j_{29}(1/5+i\sqrt{29}/145)  $\>$ =     -0.2696680568717586124611633383-0.07411556647545578849796605802i$ \\
$j_{29}(1/3+i\sqrt{29}/87)   $\>$ =     -1.667896840209630800525821959-0.7898065632614205290689810980i$ \\
$j_{29}(12/29+i/29)          $\>$ =     -2$ \\
$j_{29}(1/2+i\sqrt{29}/58)   $\>$ =     -3.049371126715242430199066574$ \\
$j_{29}(2/3+i\sqrt{29}/87)   $\>$ =     -1.667896840209630800525821959+0.7898065632614205290689810980i$ \\
$j_{29}(4/5+i\sqrt{29}/145)  $\>$ =     -0.2696680568717586124611633383+0.07411556647545578849796605802i$ \\

$j_{30}(i\sqrt{30}/30)       $\>$ =\ \   5$ \\
$j_{30}(1/6+i\sqrt{5}/30)    $\>$ =\ \   1$ \\
$j_{30}(1/4+i\sqrt{15}/60)   $\>$ =\ \   0$ \\
$j_{30}(2/5+i\sqrt{6}/30)    $\>$ =     -3$ \\
$j_{30}(1/2+i\sqrt{15}/30)   $\>$ =     -4$ \\

$j_{31}(i\sqrt{31}/31)       $\>$ =\ \   4.761369303286342734709161155$ \\
$j_{31}(11/62+i\sqrt{3}/62)  $\>$ =\ \   0$ \> (at elliptic point of order 3) \\
$j_{31}(1/4+i\sqrt{31}/124)  $\>$ =     -0.4260504821476063229868925175-0.3689894074818040877620193866i$ \\
$j_{31}(2/5+i\sqrt{31}/155)  $\>$ =     -2.380684651643171367354580578-0.05457317755183967280185952747i$ \\
$j_{31}(1/2+i\sqrt{31}/62)   $\>$ =     -3.147899035704787354026214965$ \\
$j_{31}(3/5+i\sqrt{31}/155)  $\>$ =     -2.380684651643171367354580578+0.05457317755183967280185952747i$ \\
$j_{31}(3/4+i\sqrt{31}/124)  $\>$ =     -0.4260504821476063229868925175+0.3689894074818040877620193866i$ \\

$j_{33}(i\sqrt{33}/33)       $\>$ =\ \   4.464101615137754587054892683$ \\
$j_{33}(1/6+i\sqrt{11}/66)   $\>$ =\ \   0$ \\
$j_{33}(2/9+i\sqrt{11}/99)   $\>$ =     -0.7044022574779152290190034072$ \\
$j_{33}(1/3+i\sqrt{11}/33)   $\>$ =     -1.647798871261042385490498296-1.721433237247136729005106833i$ \\
$j_{33}(1/2+i\sqrt{33}/66)   $\>$ =     -2.464101615137754587054892683$ \\
$j_{33}(2/3+i\sqrt{11}/33)   $\>$ =     -1.647798871261042385490498296+1.721433237247136729005106833i$ \\

$j_{34}(i\sqrt{34}/34)       $\>$ =\ \   4.561552812808830274910704928$ \\
$j_{34}(1/6+i\sqrt{17}/102)  $\>$ =\ \   0.5615528128088302749107049280$ \\
$j_{34}(1/5+i\sqrt{34}/170)  $\>$ =\ \   0.4384471871911697250892950720$ \\
$j_{34}(5/17+i\sqrt{2}/34)   $\>$ =     -1$ \\
$j_{34}(13/34+i/34)          $\>$ =     -2$ \> (at elliptic point of order 4) \\
$j_{34}(1/2+i\sqrt{17}/34)   $\>$ =     -3.561552812808830274910704928$ \\

$j_{35}(i\sqrt{35}/35)       $\>$ =\ \   4.136147010473344946585953198$ \\
$j_{35}(2/7+i\sqrt{5}/35)    $\>$ =     -1-2i$ \\
$j_{35}(1/3+i\sqrt{35}/105)  $\>$ =     -1.068073505236672473292976599-1.922143870450998929819659750i$ \\
$j_{35}(1/2+i\sqrt{35}/70)   $\>$ =     -2$ \\
$j_{35}(2/3+i\sqrt{35}/105)  $\>$ =     -1.068073505236672473292976599+1.922143870450998929819659750i$ \\
$j_{35}(5/7+i\sqrt{5}/35)    $\>$ =     -1+2i$ \\

$j_{38}(i\sqrt{38}/38)       $\>$ =\ \   4.150770243157541006119639149$ \\
$j_{38}(3/19+i\sqrt{2}/38)   $\>$ =\ \   0$ \\
$j_{38}(1/4+i\sqrt{19}/76)   $\>$ =     -0.4348022826163606035624719868-1.043427435893032154471565986i$ \\
$j_{38}(1/3+i\sqrt{38}/114)  $\>$ =     -1.075385121578770503059819575-0.8780090918971414962402883482i$ \\
$j_{38}(1/2+i\sqrt{19}/38)   $\>$ =     -3.130395434767278792875056027$ \\
$j_{38}(2/3+i\sqrt{38}/114)  $\>$ =     -1.075385121578770503059819575+0.8780090918971414962402883482i$ \\
$j_{38}(3/4+i\sqrt{19}/76)   $\>$ =     -0.4348022826163606035624719868+1.043427435893032154471565986i$ \\

$j_{39}(i\sqrt{39}/39)       $\>$ =\ \   4.302775637731994646559610634$ \\
$j_{39}(2/13+i\sqrt{3}/39)   $\>$ =\ \   1$ \\
$j_{39}(1/5+i\sqrt{39}/195)  $\>$ =\ \   0.6972243622680053534403893663$ \\
$j_{39}(1/4+i\sqrt{39}/156)  $\>$ =\ \   0.3027756377319946465596106337$ \\
$j_{39}(11/26+i\sqrt{3}/78)  $\>$ =     -3$ \> (at elliptic point of order 6) \\
$j_{39}(1/2+i\sqrt{39}/78)   $\>$ =     -3.302775637731994646559610634$ \\

$j_{41}(i\sqrt{41}/41)       $\>$ =\ \   4.031464197748130120960449213$ \\
$j_{41}(1/6+i\sqrt{41}/246)  $\>$ =\ \   0.4326235943581880128425063252-0.03675155710672643686408328572i$ \\
$j_{41}(9/41+i/41)           $\>$ =\ \   0$ \\
$j_{41}(1/3+i\sqrt{41}/123)  $\>$ =     -1.432623594358188012842506325-0.8008414936741550679635492458i$ \\
$j_{41}(2/5+i\sqrt{41}/205)  $\>$ =     -1.693897202308099089470896574-0.2949750344960732524774427655i$ \\
$j_{41}(1/2+i\sqrt{41}/82)   $\>$ =     -2.643669793131931942018656064$ \\
$j_{41}(3/5+i\sqrt{41}/205)  $\>$ =     -1.693897202308099089470896574+0.2949750344960732524774427655i$ \\
$j_{41}(2/3+i\sqrt{41}/123)  $\>$ =     -1.432623594358188012842506325+0.8008414936741550679635492458i$ \\
$j_{41}(5/6+i\sqrt{41}/246)  $\>$ =\ \   0.4326235943581880128425063252+0.03675155710672643686408328572i$ \\

$j_{42}(i\sqrt{42}/42)       $\>$ =\ \   4$ \\
$j_{42}(1/7+i\sqrt{6}/42)    $\>$ =\ \   1$ \\
$j_{42}(3/14+i\sqrt{3}/42)   $\>$ =\ \   0$ \\
$j_{42}(1/3+i\sqrt{14}/42)   $\>$ =     -1.500000000000000000000000000-1.322875655532295295250807877i$ \\
$j_{42}(1/2+i\sqrt{21}/42)   $\>$ =     -3$ \\
$j_{42}(2/3+i\sqrt{14}/42)   $\>$ =     -1.500000000000000000000000000+1.322875655532295295250807877i$ \\

$j_{46}(i\sqrt{46}/46)       $\>$ =\ \   3.828427124746190097603377448$ \\
$j_{46}(1/6+i\sqrt{23}/138)  $\>$ =\ \   0.5397978117457193930052088754-0.1825822545574429926939882837i$ \\
$j_{46}(1/4+i\sqrt{23}/92)   $\>$ =     -0.1225611668766536199752455518-0.7448617666197442365931704286i$ \\
$j_{46}(3/8+i\sqrt{23}/184)  $\>$ =     -1.754877666246692760049508896$ \\
$j_{46}(2/5+i\sqrt{46}/230)  $\>$ =     -1.828427124746190097603377448$ \\
$j_{46}(1/2+i\sqrt{23}/46)   $\>$ =     -3.079595623491438786010417751$ \\
$j_{46}(3/4+i\sqrt{23}/92)   $\>$ =     -0.1225611668766536199752455518+0.7448617666197442365931704286i$ \\
$j_{46}(5/6+i\sqrt{23}/138)  $\>$ =\ \   0.5397978117457193930052088754+0.1825822545574429926939882837i$ \\

$j_{47}(i\sqrt{47}/47)       $\>$ =\ \   3.585625494439901516579962811$ \\
$j_{47}(1/6+i\sqrt{47}/282)  $\>$ =     -0.3627686404360593234652810510-0.3432609660330357799566369234i$ \\
$j_{47}(1/4+i\sqrt{47}/188)  $\>$ =     -0.5870496965473447751879045292-1.250665180042536212344253698i$ \\
$j_{47}(2/7+i\sqrt{47}/329)  $\>$ =     -0.7564040798852488091169062879-1.125144262819566624706178489i$ \\
$j_{47}(1/3+i\sqrt{47}/141)  $\>$ =     -1.036408667334701949173075118-1.344877518150426533127664948i$ \\
$j_{47}(1/2+i\sqrt{47}/94)   $\>$ =     -2.100363326033191802693628840$ \\
$j_{47}(2/3+i\sqrt{47}/141)  $\>$ =     -1.036408667334701949173075118+1.344877518150426533127664948i$ \\
$j_{47}(5/7+i\sqrt{47}/329)  $\>$ =     -0.7564040798852488091169062879+1.125144262819566624706178489i$ \\
$j_{47}(3/4+i\sqrt{47}/188)  $\>$ =     -0.5870496965473447751879045292+1.250665180042536212344253698i$ \\
$j_{47}(5/6+i\sqrt{47}/282)  $\>$ =     -0.3627686404360593234652810510+0.3432609660330357799566369234i$ \\

$j_{51}(i\sqrt{51}/51)       $\>$ =\ \   3.479815748755145556651390591$ \\
$j_{51}(1/6+i\sqrt{17}/102)  $\>$ =\ \   0.3002425902201204191589098208-0.3751894661561734131203955526i$ \\
$j_{51}(1/4+i\sqrt{51}/204)  $\>$ =     -0.7399078743775727783256952956-0.7759010795250288681466316105i$ \\
$j_{51}(1/3+i\sqrt{17}/51)   $\>$ =     -1.300242590220120419158909821-1.624810533843826586879604448i$ \\
$j_{51}(1/2+i\sqrt{51}/102)  $\>$ =     -2$ \\
$j_{51}(2/3+i\sqrt{17}/51)   $\>$ =     -1.300242590220120419158909821+1.624810533843826586879604448i$ \\
$j_{51}(3/4+i\sqrt{51}/204)  $\>$ =     -0.7399078743775727783256952956+0.7759010795250288681466316105i$ \\
$j_{51}(5/6+i\sqrt{17}/102)  $\>$ =\ \   0.3002425902201204191589098208+0.3751894661561734131203955526i$ \\

$j_{55}(i\sqrt{55}/55)       $\>$ =\ \   3.618033988749894848204586834$ \\
$j_{55}(2/15+i\sqrt{11}/165) $\>$ =\ \   1.382975767906237494122708537$ \\
$j_{55}(1/7+i\sqrt{55}/385)  $\>$ =\ \   1.381966011250105151795413166$ \\
$j_{55}(1/4+i\sqrt{55}/220)  $\>$ =     -0.3819660112501051517954131656$ \\
$j_{55}(3/10+i\sqrt{11}/110) $\>$ =     -1$ \\
$j_{55}(2/5+i\sqrt{11}/55)   $\>$ =     -2.191487883953118747061354268-0.5088517788327379904864224393i$ \\
$j_{55}(1/2+i\sqrt{55}/110)  $\>$ =     -2.618033988749894848204586834$ \\
$j_{55}(3/5+i\sqrt{11}/55)   $\>$ =     -2.191487883953118747061354268+0.5088517788327379904864224393i$ \\

$j_{59}(i\sqrt{59}/59)       $\>$ =\ \   3.250604776417835784425587796$ \\
$j_{59}(1/6+i\sqrt{59}/354)  $\>$ =\ \   0.1027847152002951558510143089-0.6654569511528134767061906116i$ \\
$j_{59}(1/5+i\sqrt{59}/295)  $\>$ =\ \   0.07010628989751548539073292583-0.8513902924458037909600187007i$ \\
$j_{59}(1/4+i\sqrt{59}/236)  $\>$ =     -0.1930341320501256966883920109-1.006043605799810240847794450i$ \\
$j_{59}(1/3+i\sqrt{59}/177)  $\>$ =     -0.9318210843800659261856599331-0.9605029196296829425065118808i$ \\
$j_{59}(3/7+i\sqrt{59}/413)  $\>$ =     -1.570553461676241754729474880-0.02768497922071816524290388412i$ \\
$j_{59}(1/2+i\sqrt{59}/118)  $\>$ =     -2.205569430400590311702028618$ \\
$j_{59}(4/7+i\sqrt{59}/413)  $\>$ =     -1.570553461676241754729474880+0.02768497922071816524290388412i$ \\
$j_{59}(2/3+i\sqrt{59}/177)  $\>$ =     -0.9318210843800659261856599331+0.9605029196296829425065118808i$ \\
$j_{59}(3/4+i\sqrt{59}/236)  $\>$ =     -0.1930341320501256966883920109+1.006043605799810240847794450i$ \\
$j_{59}(4/5+i\sqrt{59}/295)  $\>$ =\ \   0.07010628989751548539073292583+0.8513902924458037909600187007i$ \\
$j_{59}(5/6+i\sqrt{59}/354)  $\>$ =\ \   0.1027847152002951558510143089+0.6654569511528134767061906116i$ \\

$j_{62}(i\sqrt{62}/62)       $\>$ =\ \   3.204257578045853033477539268$ \\
$j_{62}(1/8+i\sqrt{31}/248)  $\>$ =\ \   0.6823278038280193273694837397$ \\
$j_{62}(1/7+i\sqrt{62}/434)  $\>$ =\ \   0.6241695467003370641258381804$ \\
$j_{62}(1/4+i\sqrt{31}/124)  $\>$ =     -0.3411639019140096636847418699-1.161541399997251936087917687i$ \\
$j_{62}(1/3+i\sqrt{62}/186)  $\>$ =     -0.9142135623730950488016887242-1.078987285547468834863038539i$ \\
$j_{62}(3/8+i\sqrt{31}/248)  $\>$ =     -0.7672143840616159866716343874-0.7925519925154478483258983006i$ \\
$j_{62}(1/2+i\sqrt{31}/62)   $\>$ =     -2.465571231876768026656731225$ \\
$j_{62}(5/8+i\sqrt{31}/248)  $\>$ =     -0.7672143840616159866716343874+0.7925519925154478483258983006i$ \\
$j_{62}(2/3+i\sqrt{62}/186)  $\>$ =     -0.9142135623730950488016887242+1.078987285547468834863038539i$ \\
$j_{62}(3/4+i\sqrt{31}/124)  $\>$ =     -0.3411639019140096636847418699+1.161541399997251936087917687i$ \\

$j_{66}(i\sqrt{66}/66)       $\>$ =\ \   3.372281323269014329925305734$ \\
$j_{66}(4/33+i\sqrt{2}/66)   $\>$ =\ \   1$ \\
$j_{66}(1/6+i\sqrt{11}/66)   $\>$ =\ \   1.191487883953118747061354268-0.5088517788327379904864224393i$ \\
$j_{66}(3/11+i\sqrt{6}/66)   $\>$ =\ \   0$ \\
$j_{66}(2/5+i\sqrt{66}/330)  $\>$ =     -2.372281323269014329925305734$ \\
$j_{66}(5/12+i\sqrt{11}/132) $\>$ =     -2.382975767906237494122708537$ \\
$j_{66}(1/2+i\sqrt{33}/66)   $\>$ =     -3$ \\
$j_{66}(5/6+i\sqrt{11}/66)   $\>$ =\ \   1.191487883953118747061354268+0.5088517788327379904864224393i$ \\

$j_{69}(i\sqrt{69}/69)       $\>$ =\ \   3.091158353929329921425925022$ \\
$j_{69}(1/6+i\sqrt{23}/138)  $\>$ =\ \   0.6623589786223730129804544272-0.5622795120623012438991821449i$ \\
$j_{69}(1/5+i\sqrt{69}/345)  $\>$ =\ \   0.5000000000000000000000000000-0.4627111573517053057910121464i$ \\
$j_{69}(1/3+i\sqrt{23}/69)   $\>$ =     -1.215079854500973367044300021-1.307141278682045480492352574i$ \\
$j_{69}(5/12+i\sqrt{23}/276) $\>$ =     -1.324717957244746025960908855$ \\
$j_{69}(4/9+i\sqrt{23}/207)  $\>$ =     -1.569840290998053265911399958$ \\
$j_{69}(1/2+i\sqrt{69}/138)  $\>$ =     -2.091158353929329921425925022$ \\
$j_{69}(2/3+i\sqrt{23}/69)   $\>$ =     -1.215079854500973367044300021+1.307141278682045480492352574i$ \\
$j_{69}(4/5+i\sqrt{69}/345)  $\>$ =\ \   0.5000000000000000000000000000+0.4627111573517053057910121464i$ \\
$j_{69}(5/6+i\sqrt{23}/138)  $\>$ =\ \   0.6623589786223730129804544272+0.5622795120623012438991821449i$ \\

$j_{70}(i\sqrt{70}/70)       $\>$ =\ \   3$ \\
$j_{70}(1/5+i\sqrt{14}/70)   $\>$ =\ \   0.5000000000000000000000000000-1.322875655532295295250807877i$ \\
$j_{70}(1/4+i\sqrt{35}/140)  $\>$ =\ \   0.2971565081774243724678302298-1.205625150602912946591254240i$ \\
$j_{70}(5/14+i\sqrt{5}/70)   $\>$ =     -1$ \\
$j_{70}(3/7+i\sqrt{10}/70)   $\>$ =     -2$ \\
$j_{70}(1/2+i\sqrt{35}/70)   $\>$ =     -2.594313016354848744935660460$ \\
$j_{70}(3/4+i\sqrt{35}/140)  $\>$ =\ \   0.2971565081774243724678302298+1.205625150602912946591254240i$ \\
$j_{70}(4/5+i\sqrt{14}/70)   $\>$ =\ \   0.5000000000000000000000000000+1.322875655532295295250807877i$ \\

$j_{71}(i\sqrt{71}/71)       $\>$ =\ \   3.070135611475866755721739857$ \\
$j_{71}(1/8+i\sqrt{71}/568)  $\>$ =\ \   0.8763529352828540784184526952-0.01624652480075642036729645049i$ \\
$j_{71}(1/6+i\sqrt{71}/426)  $\>$ =\ \   0.6110278612229944404427962377-0.1571340927617910068858071215i$ \\
$j_{71}(1/4+i\sqrt{71}/284)  $\>$ =     -0.4483036287186081843216160178-0.7133079298898067087253786821i$ \\
$j_{71}(1/3+i\sqrt{71}/213)  $\>$ =     -1.121100466735809900767738347-1.004964301540170776567151801i$ \\
$j_{71}(3/8+i\sqrt{71}/568)  $\>$ =     -1.141263470404508727609883652-0.6558035915274742285685986381i$ \\
$j_{71}(2/5+i\sqrt{71}/355)  $\>$ =     -1.290320274284977555511584278-0.6264473428891745139938991768i$ \\
$j_{71}(1/2+i\sqrt{71}/142)  $\>$ =     -2.042921524199755057022593135$ \\
$j_{71}(3/5+i\sqrt{71}/355)  $\>$ =     -1.290320274284977555511584278+0.6264473428891745139938991768i$ \\
$j_{71}(5/8+i\sqrt{71}/568)  $\>$ =     -1.141263470404508727609883652+0.6558035915274742285685986381i$ \\
$j_{71}(2/3+i\sqrt{71}/213)  $\>$ =     -1.121100466735809900767738347+1.004964301540170776567151801i$ \\
$j_{71}(3/4+i\sqrt{71}/284)  $\>$ =     -0.4483036287186081843216160178+0.7133079298898067087253786821i$ \\
$j_{71}(5/6+i\sqrt{71}/426)  $\>$ =\ \   0.6110278612229944404427962377+0.1571340927617910068858071215i$ \\
$j_{71}(7/8+i\sqrt{71}/568)  $\>$ =\ \   0.8763529352828540784184526952+0.01624652480075642036729645049i$ \\

$j_{78}(i\sqrt{78}/78)$\>$ =\ \  3$ \\
$j_{78}(1/9+i\sqrt{26}/234)  $\>$ =\ \   1.314596212276751981650111040$ \\
$j_{78}(1/8+i\sqrt{39}/312)  $\>$ =\ \   1.302775637731994646559610634$ \\
$j_{78}(1/4+i\sqrt{39}/156)  $\>$ =     -0.5000000000000000000000000000-0.8660254037844386467637231707i$ \\
$j_{78}(1/3+i\sqrt{26}/78)   $\>$ =     -1.157298106138375990825055520-1.305151526504743953668268182i$ \\
$j_{78}(11/26+i\sqrt{3}/78)  $\>$ =     -1$ \\
$j_{78}(1/2+i\sqrt{39}/78)   $\>$ =     -2.302775637731994646559610634$ \\
$j_{78}(2/3+i\sqrt{26}/78)   $\>$ =     -1.157298106138375990825055520+1.305151526504743953668268182i$ \\
$j_{78}(3/4+i\sqrt{39}/156)  $\>$ =     -0.5000000000000000000000000000+0.8660254037844386467637231707i$ \\

$j_{87}(i\sqrt{87}/87)       $\>$ =\ \   2.546818276884082079135997509$ \\
$j_{87}(1/6+i\sqrt{29}/174)  $\>$ =\ \   0.1353981908291098795018506027-1.060602944919640288072682303i$ \\
$j_{87}(2/9+i\sqrt{29}/261)  $\>$ =     -0.2419441275215273555706181970-1.347810384779931028708183550i$ \\
$j_{87}(1/4+i\sqrt{87}/348)  $\>$ =     -0.3036766091486795941175602026-1.435949864109956088410719015i$ \\
$j_{87}(1/3+i\sqrt{29}/87)   $\>$ =     -0.8934540633075825239312324056-1.712792560139709259364498753i$ \\
$j_{87}(3/7+i\sqrt{87}/609)  $\>$ =     -0.2734091384420410395679987544-0.5638210928291186663377083166i$ \\
$j_{87}(1/2+i\sqrt{87}/174)  $\>$ =     -1.392646781702640811764879595$ \\
$j_{87}(4/7+i\sqrt{87}/609)  $\>$ =     -0.2734091384420410395679987544+0.5638210928291186663377083166i$ \\
$j_{87}(2/3+i\sqrt{29}/87)   $\>$ =     -0.8934540633075825239312324056+1.712792560139709259364498753i$ \\
$j_{87}(3/4+i\sqrt{87}/348)  $\>$ =     -0.3036766091486795941175602026+1.435949864109956088410719015i$ \\
$j_{87}(7/9+i\sqrt{29}/261)  $\>$ =     -0.2419441275215273555706181970+1.347810384779931028708183550i$ \\
$j_{87}(5/6+i\sqrt{29}/174)  $\>$ =\ \   0.1353981908291098795018506027+1.060602944919640288072682303i$ \\

$j_{94}(i\sqrt{94}/94)       $\>$ =\ \   2.753536581630347108664494311$ \\
$j_{94}(1/8+i\sqrt{47}/376)  $\>$ =\ \   0.7011860182624305300552726709-0.3777117782814777269670820954i$ \\
$j_{94}(1/6+i\sqrt{47}/282)  $\>$ =\ \   0.6707366869939796995036167910-0.7209727443145135069237190188i$ \\
$j_{94}(1/5+i\sqrt{94}/470)  $\>$ =\ \   0.5000000000000000000000000000-0.7605439663465815028787295055i$ \\
$j_{94}(1/4+i\sqrt{47}/188)  $\>$ =\ \   0.1661596545838042464568125856-0.9387127931245796992422807331i$ \\
$j_{94}(3/8+i\sqrt{47}/376)  $\>$ =     -1.454395369393579551699989786-0.06575939136352121386510913084i$ \\
$j_{94}(5/12+i\sqrt{47}/564) $\>$ =     -1.734691345692469553024170513$ \\
$j_{94}(3/7+i\sqrt{94}/658)  $\>$ =     -1.753536581630347108664494311$ \\
$j_{94}(1/2+i\sqrt{47}/94)   $\>$ =     -2.432682635200800295607254011$ \\
$j_{94}(5/8+i\sqrt{47}/376)  $\>$ =     -1.454395369393579551699989786+0.06575939136352121386510913084i$ \\
$j_{94}(3/4+i\sqrt{47}/188)  $\>$ =\ \   0.1661596545838042464568125856+0.9387127931245796992422807331i$ \\
$j_{94}(4/5+i\sqrt{94}/470)  $\>$ =\ \   0.5000000000000000000000000000+0.7605439663465815028787295055i$ \\
$j_{94}(5/6+i\sqrt{47}/282)  $\>$ =\ \   0.6707366869939796995036167910+0.7209727443145135069237190188i$ \\
$j_{94}(7/8+i\sqrt{47}/376)  $\>$ =\ \   0.7011860182624305300552726709+0.3777117782814777269670820954i$ \\

$j_{95}(i\sqrt{95}/95)       $\>$ =\ \   2.748195845763711583531182345$ \\
$j_{95}(1/10+i\sqrt{19}/190) $\>$ =\ \   1$ \\
$j_{95}(1/8+i\sqrt{95}/760)  $\>$ =\ \   0.6984784404232374822470263773$ \\
$j_{95}(1/5+i\sqrt{19}/95)   $\>$ =\ \   0.5651977173836393964375280133-1.043427435893032154471565986i$ \\
$j_{95}(1/4+i\sqrt{95}/380)  $\>$ =\ \   0.3090169943749474241022934172-0.7228710022800531180923645586i$ \\
$j_{95}(1/3+i\sqrt{95}/285)  $\>$ =     -0.8090169943749474241022934172-0.4467588488907755619527569969i$ \\
$j_{95}(4/9+i\sqrt{95}/855)  $\>$ =     -2.130161857013816735326595511$ \\
$j_{95}(9/20+i\sqrt{19}/380) $\>$ =     -2.130395434767278792875056027$ \\
$j_{95}(1/2+i\sqrt{95}/190)  $\>$ =     -2.316512429173132330451613212$ \\
$j_{95}(2/3+i\sqrt{95}/285)  $\>$ =     -0.8090169943749474241022934172+0.4467588488907755619527569969i$ \\
$j_{95}(3/4+i\sqrt{95}/380)  $\>$ =\ \   0.3090169943749474241022934172+0.7228710022800531180923645586i$ \\
$j_{95}(4/5+i\sqrt{19}/95)   $\>$ =\ \   0.5651977173836393964375280133+1.043427435893032154471565986i$ \\

$j_{105}(i\sqrt{105}/105)     $\>$ =\ \  2.791287847477920003294023597$ \\
$j_{105}(2/21+i\sqrt{5}/105)  $\>$ =\ \  1.618033988749894848204586834$ \\
$j_{105}(1/9+i\sqrt{35}/315)  $\>$ =\ \  1.594313016354848744935660460$ \\
$j_{105}(1/6+i\sqrt{35}/210)  $\>$ =\ \  1$ \\
$j_{105}(5/21+i\sqrt{5}/105)  $\>$ =    -0.6180339887498948482045868344$ \\
$j_{105}(1/3+i\sqrt{35}/105)  $\>$ =    -1.297156508177424372467830230-1.205625150602912946591254240i$ \\
$j_{105}(2/5+i\sqrt{21}/105)  $\>$ =    -1.500000000000000000000000000-0.8660254037844386467637231707i$ \\
$j_{105}(1/2+i\sqrt{105}/210) $\>$ =    -1.791287847477920003294023597$ \\
$j_{105}(3/5+i\sqrt{21}/105)  $\>$ =    -1.500000000000000000000000000+0.8660254037844386467637231707i$ \\
$j_{105}(2/3+i\sqrt{35}/105)  $\>$ =    -1.297156508177424372467830230+1.205625150602912946591254240i$ \\

$j_{110}(i\sqrt{110}/110)     $\>$ =\ \  2.394858673866065943118600576$ \\
$j_{110}(1/11+i\sqrt{10}/110) $\>$ =\ \  1$ \\
$j_{110}(1/8+i\sqrt{55}/440)  $\>$ =\ \  0.6180339887498948482045868344$ \\
$j_{110}(3/20+i\sqrt{11}/220) $\>$ =\ \  0.2955977425220847709809965929$ \\
$j_{110}(1/4+i\sqrt{55}/220)  $\>$ =    -0.5000000000000000000000000000-1.658312395177699924557466368i$ \\
$j_{110}(3/10+i\sqrt{11}/110) $\>$ =    -0.6477988712610423854904982964-1.721433237247136729005106833i$ \\
$j_{110}(1/3+i\sqrt{110}/330) $\>$ =    -0.6974293369330329715593002878-1.689402768409776232971874028i$ \\
$j_{110}(1/2+i\sqrt{55}/110)  $\>$ =    -1.618033988749894848204586834$ \\
$j_{110}(2/3+i\sqrt{110}/330) $\>$ =    -0.6974293369330329715593002878+1.689402768409776232971874028i$ \\
$j_{110}(7/10+i\sqrt{11}/110) $\>$ =    -0.6477988712610423854904982964+1.721433237247136729005106833i$ \\
$j_{110}(3/4+i\sqrt{55}/220)  $\>$ =    -0.5000000000000000000000000000+1.658312395177699924557466368i$ \\

$j_{119}(i\sqrt{119}/119)     $\>$ =\ \  2.187815800253075177796221462$ \\
$j_{119}(1/5+i\sqrt{119}/595) $\>$ =\ \  0.3059015743907980337804293266-1.622693339344373412242882441i$ \\
$j_{119}(3/14+i\sqrt{17}/238) $\>$ =\ \  0.3002425902201204191589098208-1.624810533843826586879604448i$ \\
$j_{119}(1/4+i\sqrt{119}/476) $\>$ =\ \  0.2419079387990809430243016505-1.586661022593502553061512903i$ \\
$j_{119}(1/3+i\sqrt{119}/357) $\>$ =    -0.3998094745173356226785400578-1.006756526422602861301612469i$ \\
$j_{119}(3/8+i\sqrt{119}/952) $\>$ =    -0.4246516023520581806572584412-0.2391047313257304592960825073i$ \\
$j_{119}(3/7+i\sqrt{17}/119)  $\>$ =    -1.300242590220120419158909821-0.3751894661561734131203955526i$ \\
$j_{119}(1/2+i\sqrt{119}/238) $\>$ =    -1.634512672894045524734086419$ \\
$j_{119}(4/7+i\sqrt{17}/119)  $\>$ =    -1.300242590220120419158909821+0.3751894661561734131203955526i$ \\
$j_{119}(5/8+i\sqrt{119}/952) $\>$ =    -0.4246516023520581806572584412+0.2391047313257304592960825073i$ \\
$j_{119}(2/3+i\sqrt{119}/357) $\>$ =    -0.3998094745173356226785400578+1.006756526422602861301612469i$ \\
$j_{119}(3/4+i\sqrt{119}/476) $\>$ =\ \  0.2419079387990809430243016505+1.586661022593502553061512903i$ \\
$j_{119}(11/14+i\sqrt{17}/238)$\>$ =\ \  0.3002425902201204191589098208+1.624810533843826586879604448i$ \\
$j_{119}(4/5+i\sqrt{119}/595) $\>$ =\ \  0.3059015743907980337804293266+1.622693339344373412242882441i$ \\
\end{tabbing}

\medskip


\begin{thebibliography}{99}

\bibitem{AtLeh70}
A.~O.~L.\ Atkin, J.~Lehner,
Hecke operators on $\Gamma_0(m)$,
Math.\ Ann.\ \textbf{185} (1970), 134--160.

\bibitem{Bo92}
R.~E.\ Borcherds,
Monstrous moonshine and monstrous Lie superalgebras,
Invent.\ Math.\ \textbf{109} (1992), 405--444.

\bibitem{SCM66}
\emph{Seminar on complex multiplication},
eds.\ A.~Borel, S.~Chowla, C.~S.\ Herz, K.~Iwasawa, J.~P.\ Serre,
Lecture Notes in Mathematics \textbf{21} Springer-Verlag, Berlin-New York, 1966.

\bibitem{CY93}
I.~Chen, N.~Yui,
Singular values of Thompson series.
In \emph{Groups, difference sets, and the Monster}
(Columbus, OH, 1993), 255--326, Ohio State University
Mathematics Research Institute Publications, 4, de Gruyter, Berlin, 1996.

\bibitem{CK05}
S.~Y.\ Choi, J.~K.\ Koo,
Class fields from the fundamental Thompson series of level $N=o(g)$,
J.~Korean.\ Math.\ Soc.\ \textbf{42} No.~2 (2005), 203--222.

\bibitem{CN79}
J.~H.\ Conway, S.~P.\ Norton,
Monstrous moonshine,
Bull.\ London Math.\ Soc.\ \textbf{11} (1979), 308--339.

\bibitem{Cox89} D. Cox, Primes of the form $x^2+ny^2$, John WileySons, New York.

\bibitem{CMcKS04}
D.~Cox, J.~McKay, P.~Stevenhagen,
Principal moduli and class fields,
Bull.\ London Math.\ Soc.\ \textbf{36} vol.~1 (2004), 3--12.

\bibitem{Cum04}
C.~J.\ Cummins,
Congruence subgroups of groups commensurable with $\PSL(2,\Z)$ of genus
$0$ and $1$, Experiment.\ Math.\ \textbf{13} (2004), 361--382.

\bibitem{FS14}
J.~Freitag and T.~Scanlon, Strong minimality and the $j$-function,
\url{http://arxiv.org/abs/1402.4588},
accepted for publication in J.\ European Math.\ Soc.

\bibitem{GZ85}
B.~Gross, D.~Zagier,
On singular moduli,
J.~Reine Angew.\ Math.\ \textbf{355} (1985), 191--220.

\bibitem{HMcK00}
J.~Harnad, J.~McKay,
Modular solutions to equations of generalized Halphen type,
Proc.\ R.\ Soc.\ Lond.\ A \textbf{456} (2000), 261--294.

\bibitem{GMP}
T.~Granlund,
\emph{GMP -- The GNU Multiple Precision Arithmetic Library, Version 5.0.5}; 2012,
\url{https://gmplib.org/}

\bibitem{GAP}
The GAP~Group,
\emph{GAP -- Groups, Algorithms, and Programming, Version 4.7.9}; 2015,
\url{http://www.gap-system.org}

\bibitem{JST12}
J.~Jorgenson, L.~Smajlovi\'c, and H.~Then,
On the distribution of eigenvalues of Maass forms on certain moonshine groups,
Math.\ Comp.\ \textbf{83} (2014), 3039--3070.

\bibitem{JST2}
J.~Jorgenson, L.~Smajlovi\'c, and H.~Then,
Kronecker's limit formula, holomorphic modular functions and $q$-expansions
on certain arithmetic groups,
Experiment.\ Math. \textbf{25} No. 3 (2016), 295--320.

\bibitem{JST3}
J.~Jorgenson, L.~Smajlovi\'c, and H.~Then,
Certain aspects of holomorphic function theory on some genus zero
arithmetic groups, LMS J. Comput. Math. \textbf{19} (2016), no. 2, 360--381.

\bibitem{JST15URL}
J.~Jorgenson, L.~Smajlovi\'c, and H.~Then, data page
\url{http://www.efsa.unsa.ba/~lejla.smajlovic/}

\bibitem{Masser03}
D.~Masser,
Heights, transcendence, and linear independence on commutative group varieties.
In \emph{Diophantine approximation (Cetraro 2000)},
Lecture Notes in Mathematics \textbf{1819}, Springer, Berlin, 2003, 1--51.

\bibitem{PARI2}
The PARI~Group,
\emph{PARI/GP version 2.5.1},
Bordeaux, 2014, \url{http://pari.math.u-bordeaux.fr/}.

\bibitem{Sch37}
T.~Schneider,
Arithmetische Untersuchungen elliptischer Integrale,
Math.\ Annalen \textbf{113} (1937), 1--13,

\bibitem{Serre73}
J.-P.\ Serre,
\emph{A Course in Arithmetic},
Graduate Texts in Mathematics, \textbf{7},
Springer-Verlag, New York, 1973.

\bibitem{Siegel49}
C.~L.\ Siegel,
\emph{Transcendental Numbers},
Annals of Mathematics Studies, \textbf{16},
Princeton University Press, Princeton, NJ, 1949.

\bibitem{Za02}
D.~Zagier,
Traces of singular moduli.
In \emph{Motives, polylogarithms and Hodge theory, Part I (Irvine, CA, 1998)},
Int.\ Press Lect.\ Ser., vol.~3, Int.\ Press, Somerville, MA, 2002, 211--244.

\end{thebibliography}
\end{document}